\RequirePackage{fix-cm}
\documentclass[smallextended,numbook,runningheads]{svjour3}
\smartqed  % flush right qed marks, e.g.~at end of proof
\usepackage{xcolor}
\usepackage{amsmath,amssymb,bm}
\usepackage{graphicx,graphics,color}
\usepackage[colorlinks,linkcolor=blue]{hyperref}
\usepackage{tikz}
\usepackage[normalem]{ulem}
\usepackage{epstopdf}

\journalname{...}
\date{ \phantom{b} \vspace{45mm}\phantom{e}}
\def\makeheadbox{\hspace{-4mm}Version of \today}%, file: esep.tex}

\usepackage{color}

\newcommand\bfd{{\mathbf d}}
\newcommand\bfe{{\mathbf e}}
\newcommand\bff{{\mathbf f}}
\newcommand\bfg{{\mathbf g}}

\newcommand\bfr{{\mathbf r}}

\newcommand\bfu{{\mathbf u}}

\newcommand\bfw{{\mathbf w}}

\newcommand\bfz{{\mathbf z}}
\newcommand\bfdu{{\mathbf {d_u}}}
\newcommand\bfdw{{\mathbf {d_w}}}
\newcommand\bfeu{{\mathbf {e_u}}}
\newcommand\bfew{{\mathbf {e_w}}}
\newcommand\bfud{{\mathbf {\dot{u}}}}

\newcommand\bfdud{{\mathbf {\dot{d}_u}}}

\newcommand\bfdwd{{\mathbf {\dot{d}_w}}}
\newcommand\bfeud{{\mathbf {\dot{e}_u}}}
\newcommand\bfewd{{\mathbf {\dot{e}_w}}}

\newcommand\bfwt{\bfw^T}

\newcommand\bfeut{\eu^T}
\newcommand\bfewt{\ew^T}

\newcommand{\bfeudt}{\bfeud^T}
\newcommand{\bfewdt}{\bfewd^T}
\newcommand\bfA{{\mathbf A}}

\newcommand\bfK{{\mathbf K}}
\newcommand\bfM{{\mathbf M}}
\newcommand\bfR{{\mathbf R}}

\newcommand\bfMd{{\mathbf {\dot{M}}}}
\newcommand\bfAd{{\mathbf {\dot{A}}}}
\newcommand\bfMdd{{\mathbf {\ddot{M}}}}
\newcommand\bfAdd{{\mathbf {\ddot{A}}}}

\newcommand\bfvartheta{{\boldsymbol \vartheta}}

\newcommand\andquad{\quad\hbox{ and }\quad}

\newcommand\quadand{\quad\hbox{ and }\quad}

\newcommand\quadfor{\quad\hbox{ for }\quad}
\newcommand\quadfora{\quad\hbox{ for all }\quad}

\newcommand{\diff}{\frac{\d}{\d t}}

\renewcommand\ss{s}
\renewcommand\d{\hbox{\rm{d}}}

\newcommand{\Ga}{\Gamma}

\newcommand{\Gat}{\Gamma(t)}
\newcommand{\GT}{\mathcal{G}_T}
\newcommand{\laplace}{\Delta}

\newcommand{\nbg}{\nabla_{\Gamma}}
\newcommand{\nbgh}{\nabla_{\Gamma_h}}
\newcommand{\mat}{\partial^{\bullet}}

\newcommand{\eps}{\varepsilon}

\newcommand{\inv}{^{-1}}

\newcommand{\nb}{\nabla}

\newcommand{\pa}{\partial}
\newcommand{\R}{\mathbb{R}}

\def \t {(t)}
\newcommand{\ct}{(\cdot,t)}

\def \to {\rightarrow}
\newcommand{\vphi}{\varphi}

\def \s {(s)}

\newcommand{\us}{\bfu^\ast}

\newcommand{\ws}{\bfw^\ast}
\newcommand{\dotus}{\bfud^\ast}

\newcommand{\eu}{\bfe_\bfu}

\newcommand{\ew}{\bfe_\bfw}

\newcommand{\doteu}{\dot\bfe_\bfu}

\newcommand{\du}{\bfd_\bfu}

\newcommand{\dw}{\bfd_\bfw}

\newcommand{\N}{{\mathbb N}}

\newcommand{\blueoff}{\color{black}}

\begin{document}

\title{ Stability and error estimates for non-linear Cahn--Hilliard-type  equations on evolving surfaces}
% Error estimates for Cahn--Hilliard equations on evolving surfaces: a new stability result via energy estimates

\titlerunning{Numerical analysis for Cahn--Hilliard-type equations on evolving surfaces}        % if too long for running head

\author{Cedric Aaron Beschle \and Bal\'{a}zs~Kov\'{a}cs
}

\authorrunning{C.A.~Beschle and B.~Kov\'{a}cs } % if too long for running head

\institute{C.A.~Beschle\at
			Institut f\"ur angewandte Analysis und numerische Simulation (IANS), Universit\"at Stuttgart, \\
			Pfaffenwaldring 57, 70569 Stuttgart, Germany \\
			\email{cedric.beschle@ians.uni-stuttgart.de} \\
			and B.~Kov\'{a}cs\at
			Faculty of Mathematics, University of Regensburg, \\
			Universit\"atsstra\ss{}e 31, 93049 Regensburg, Germany \\
			\email{balazs.kovacs@mathematik.uni-regensburg.de}
}

\date{}

\maketitle

\begin{abstract}
 In this paper, we consider a non-linear fourth-order evolution equation of Cahn--Hilliard-type on evolving surfaces with prescribed velocity, where the non-linear terms are only assumed to have locally Lipschitz derivatives.  High-order evolving surface finite elements are used to discretise the weak equation system in space, and a modified matrix--vector formulation for the semi-discrete problem is derived. The anti-symmetric structure of the equation system is preserved by the spatial discretisation. A new stability proof, based on this structure, combined with consistency bounds proves optimal-order and uniform-in-time error estimates. The paper is concluded by a variety of numerical experiments. 

%[cb: There needs to be a differentiation between CH with general non-linear potential and general non-linear CH to avoid confusion.] 

  \keywords{Cahn--Hilliard equation \and evolving surfaces \and evolving surface finite elements \and error estimates \and stability \and energy estimates \and general non-linear problems} 
  \subclass{65M60 \and 35R01 \and 35K55 \and 65M12 \and 65M15}
\end{abstract}

\section{Introduction}

This paper studies  non-linear fourth-order evolution equations of Cahn--Hilliard-type  on evolving surfaces with prescribed surface velocity. The nonlinearities and their derivatives are only assumed to satisfy locally Lipschitz-type assumptions. The Cahn--Hilliard-type equation is formulated as a system of second-order equations, exhibiting an \emph{anti-symmetric} structure:

\begin{equation}
\label{eq:CH intro}
	\begin{aligned}
		\mat u - \laplace_{\Ga\t} w = &\ f(u,\nb_{\Ga\t} u)  - u (\nb_{\Ga\t} \cdot v) , \\
		w + \laplace_{\Ga\t} u = &\ g(u,\nb_{\Ga\t} u) ,
		\end{aligned}
		\qquad \text{on } \Ga\t .
\end{equation}

The semi-discretisation of the system by high-order evolving surface finite elements, cf.~\cite{DziukElliott_ESFEM,highorderESFEM}, preserves this anti-symmetric structure, which is utilised to prove a convergence result, via a new stability proof exploiting this structure. Optimal-order uniform-in-time error estimates in the $L^2$ and $H^1$ norms (depending on the $\nb_{\Ga\t} u$-dependence of the nonlinearities) for both solution variables are proved. 

Cahn and Hilliard first described an equation modelling phase separation processes in \cite{CahnHilliard}. Since then it found many applications in an evolving surface setting as well: 
\cite{OConnorStinner} investigates the asymptotic limit, % (as the interfacial width parameter $\eps$ tends to zero) of the Cahn--Hilliard equation on evolving surfaces with prescribed velocity,
and the effect of a mobility term leading to a degenerate Cahn--Hilliard equation.
In \cite{FemTwoPhase} a discretisation of a coupled Cahn--Hilliard/Navier--Stokes system for lipid bilayer membranes is studied. 
In \cite{LatPhaseSep} the authors simulated lateral phase separation and coarsening in biological membranes by comparing surface Cahn--Hilliard and surface Allen--Cahn equations using unfitted finite elements. 
In \cite{PhaseSepDynSurf} a model of lateral phase separation in a two component material surface is presented. %, that can be regarded as a Cahn--Hilliard equation on an evolving surface. 
In \cite{IsogeomFem} %the Cahn--Hilliard equation on evolving surfaces is used to 
a model for phase transitions on deforming surfaces is studied using isogeometric finite elements. 
For singular non-linearities, well-posedness and global-in-time existence results are established in the recent preprint \cite{CaetanoElliott2021}. 
A review of the planar case is found, e.g. in \cite{Elliott_flatCHsurvey}.

The Cahn--Hilliard equation on a \emph{stationary} surface with boundary was first investigated by Du, Ju and Tian in \cite{DuJuTian}. They study a full discretisation of the Cahn--Hilliard equation with homogeneous Dirichlet boundary conditions, and prove optimal-order error estimates in the $L^2$ norm for $u$, using linear finite elements.

Elliott and Ranner were the first to consider the Cahn--Hilliard equation on a closed \emph{evolving} surface with a prescribed velocity in \cite{ElliottRanner_CH}. They proved optimal-order uniform-in-time error estimates in the $L^2$ and $H^1$ norms for the concentration difference and optimal-order $L^2$-in-time error estimates in the $L^2$ and $H^1$ norms for the chemical potential using a discretisation by linear evolving surface finite elements.
Using a new stability proof, the results of this paper improve the error estimates for the chemical potential from optimal-order $L^2$-in-time to optimal-order uniform-in-time estimates. 
%Furthermore, the proof is built, such that it is extendable to general non-linear Cahn--Hilliard equations in a straightforward way.

 In \cite{LatPhaseSep,PhaseSepDynSurf} phase separation on dynamic membranes was approximated by a mixed finite difference--finite element discretisation of the Cahn--Hilliard equation on evolving surfaces.

%%%%%

\medskip

The main results of this paper are  stability and  optimal-order uniform-in-time semi-discrete error estimates for the evolving surface  non-linear Cahn--Hilliard-type equations: 
 (a) in the $H^1$ norm if the nonlinearities depend on $u$ and $\nbg u$, requiring at least quadratic finite elements, and (b) both in the $L^2$ and $H^1$ norms if both nonlinearities are \emph{independent} of the surface gradient, using finite elements of degree $k \geq 1$. 
Convergence is proved via a new  stability estimate and showing consistency of the semi-discretisation.   

 The rather general model \eqref{eq:CH intro} includes the Cahn--Hilliard equation with proliferation terms \cite[equation~$(3.1)$]{Miranville}, with advection terms on the surface cf.~\cite{advectiveCH},
the generalised Cahn--Hilliard-type equation of Cherfils, Miranville and Zelik \cite[equation~$(1.7)$]{CherfilsMiranvilleZelik2014}, see also \cite{DuanZhao2017,Miranville2019} and the reference therein for theoretical results,
and the generalised Cahn--Hilliard equation from \cite{KhainSander2008generalized}, etc.
To correct mesh deformations of the evolving discrete surface arbitrary Lagrangian Eulerian (ALE) methods have been proposed and analysed, see, e.g. \cite{ALE1,ALE2}, the correcting advection-like term with the tangential ALE velocity also fit into the framework of \eqref{eq:CH intro}.

Another main contribution of the paper is a new stability proof based on multiple energy estimates (summarised in Figure~\ref{fig:energy estimates}). 
The main idea is to exploit the anti-symmetric structure of the second-order system corresponding to the Cahn--Hilliard(-type) equation. The generality of the stability proof can also be seen through the related results in \cite{Willmore} and \cite{CHdynbc}. A further advantage of this stability proof, is that we strongly expect it to translate to proving stability and convergence of full discretisations using linearly implicit backward difference formulae. This is, however, beyond the scope of this paper.

%A Ritz map \cite{LubichMansour_wave,highorderESFEM}, that maps quantities onto the evolving surface finite element space, is used to split the error into two error terms. The first term is bounded by the Ritz map estimates while the bound for the second term is proved in this paper. This proof is subdivided into a 

In the presented stability analysis, the difference between the Ritz map of the exact solution and the numerical solution is estimated in terms of defects and their time derivatives. To account for initial errors in the chemical potential, a modification of the semi-discrete system is required. The stability proof uses \emph{energy estimates}, performed in the matrix--vector formulation, and utilises the \emph{anti-symmetric} structure of the error equations, testing the error equations with the errors and also with their time derivatives. The stability analysis was first developed for Willmore flow in \cite{Willmore}. A uniform-in-time $L^\infty$ bound for the numerical solution is key to estimate the non-linear term. It is obtained from the time-uniform $H^1$ norm error bounds using an inverse estimate and exists for a small time due to a continuous initial function. The stability proof is independent of geometric errors.

%We also give details to an extension of the stability proof to rather general non-linear Cahn--Hilliard equations, which, e.g. include proliferation terms \cite[equation~$(3.1)$]{Miranville}, where an additional non-linearity appears in both variables in both equations of the system. The generality of the stability proof can also be seen through the related results in \cite{Willmore} and \cite{CHdynbc}. A further advantage of this stability proof, is that it is expected to be generalisable to a full discrete method based on linearly implicit backward difference methods. 
%%, cf.~\cite{LubichMansourVenkataraman_bdsurf,ALE2,KovacsPower_quasilinear}.

In the consistency analysis the $L^2$ norms of the defects and their time derivatives are estimated. The bounds use geometric error estimates, including interpolation and Ritz map error estimates, bounds on the discrete surface velocity, and geometric approximation errors for high-order evolving surface finite elements, see~\cite{highorderESFEM}.

%These consistency estimates are straightforwardly extended to general non-linear Cahn--Hilliard equations. This leads to a theorem stating optimal-order error estimates as formulated above for the general non-linear Cahn--Hilliard equation.

%%%%%

The paper is structured as follows. In Section \ref{section:CH}, based on the papers \cite{DziukElliott_ESFEM} and \cite{ElliottRanner_CH}, the weak formulation for the Cahn--Hilliard equation on evolving surfaces is derived as a system of equations. In Section \ref{section:semidiscrete CH equations} the evolving surface finite element method is used to discretise this system of equations in space. The obtained semi-discrete problem is written as a matrix--vector formulation. In Section \ref{section:error estimates} the novel error estimates proved in this work are stated and discussed in comparison to the existing results by Elliott and Ranner \cite{ElliottRanner_CH}. Section \ref{section:stability} contains the stability part of the proof. Section \ref{section:consistency} treats the consistency part of the proof. In Section \ref{section:proof of main theorem} the two parts are combined to prove the main result.
%In Section \ref{section:extensions} the result is extended to general non-linear Cahn--Hilliard equations. 
In Section \ref{section:BDF} a full discretisation to the problem is given, cf.~\cite{AkrivisLubich_quasilinBDF,AkrivisLiLubich_quasilinBDF}. In Section \ref{section:numerics} the theoretical results are complemented by numerical experiments.

\section{Cahn--Hilliard equation on evolving surfaces}
\label{section:CH}

In the following we consider a smoothly evolving closed surface $\Ga\t \subset \R^{d+1}$, with $d = 1,2$, for $0 \leq t \leq T$. The initial surface $\Ga(0) = \Ga^0$ is given (and at least $C^2$), and it evolves with the \emph{given} and sufficiently smooth velocity $v$. The surface $\Ga\t$ is given as the image of a smooth mapping $X: \Ga^0 \times [0,T] \to \R^{d+1}$, by $\Ga\t = \{X(p,t) \mid p \in \Ga^0 \}$. The embedding $X$ and the velocity $v$ satisfy the ordinary differential equation (ODE):
\begin{equation}
\label{eq:ODE for positions}
	\pa_t X(p,t) = v(X(p,t),t) \qquad p \in \Ga^0, \ 0 \leq t \leq T .
\end{equation}
Let $\nu$ denote the unit outward normal vector to $\Ga\t$. 
Then the surface (or tangential) gradient on $\Ga\t$, of a function $u : \Ga\t \to \R$, is denoted by $\nb_{\Ga\t} u$, and is given by $\nb_{\Ga\t} u = \nb \bar u -( \nb \bar u \cdot \nu) \nu$ (the surface gradient is independent of the extension $\bar u$ into a small neighbourhood of $\Ga\t$), while the Laplace--Beltrami operator on $\Ga\t$ is given by $\laplace_{\Ga\t} u = \nb_{\Ga\t} \cdot \nb_{\Ga\t} u$.
Moreover, $\mat u$ denotes the material derivative of $u$, i.e.~$\mat u(\cdot,t) = \d / \d t (u(X(\cdot,t),t)) = \pa_t \bar u(\cdot,t) + v \cdot \nb \bar u(\cdot,t)$. The space-time manifold will be denoted by $\GT = \cup_{t \in [0,T]} \Ga\t \times \{t\}$. For more details on these notions we refer to \cite{DziukElliott_ESFEM,DziukElliott_acta,Demlow,highorderESFEM}.
%In addition to its definition we are going to describe $\Ga\t$ as the zero-level set of a smooth signed distance function $d(\cdot,t): \R^{d+1} \to \R$ such that 
%\begin{equation*}
%\Ga\t := \{ x \in \R^{d+1} \, |\, d(x,t) =0 \}.
%%\label{eq: EvolvingSurface}
%\end{equation*}
%
%[kb: Surface representation through signed distance function: here or later...] 
%[cb: added the definition] 

In this paper we consider the  \emph{general non-linear} Cahn--Hilliard-type equation on evolving surfaces. It is a second-order system of partial differential equations for scalar functions $u, w: \GT \to \R$ given by
\begin{subequations}
\label{eq:CH system}
	\begin{alignat}{3}
		\label{eq:CH system - u}
		\mat u - \laplace_{\Ga\t} w = &\ f(u,\nb_{\Ga \t} u) - u (\nb_{\Ga\t} \cdot v) & \qquad & \text{on } \Ga\t , \\
		\label{eq:CH system - w}
		 w + \laplace_{\Ga\t} u   = &\  g(u,\nb_{\Ga \t} u) & \qquad & \text{on } \Ga\t ,
	\end{alignat}
\end{subequations}
 with continuous (and sufficiently regular) initial condition $u(\cdot,0) = u^0$ on the initial surface $\Ga^0$. 
% The parameter $\eps > 0$ describes the length of the transition regions, we are however not interested in taking the limit $\eps \to 0$, and thus will set $\eps=1$.
 The scalar functions $f,g:\R \times \R^d \to \R$ and their derivatives $\pa_i f, \pa_i g$ are only assumed to be locally Lipschitz continuous. A typical example is a double-well potential, i.e.~for the Cahn--Hilliard equation sets  $f(u) = 0$ and $g(u) = \frac{1}{4}((u^2 - 1)^2)'$. In this case, the  solution $u \in [-1,1]$ models the concentration of surfactant fluids, with $u=\pm 1$ indicating the pure occurrences of each, cf.~\cite{CahnHilliard}. 

 The  classical Cahn--Hilliard equation on a stationary surface $\Ga$ can be derived as the $H\inv(\Ga)$ gradient flow of the Ginzburg--Landau energy 
\begin{equation}
\label{eq:CH energy}
	E(u) = \int_{\Ga} \!\! \Big( \, \frac{1}{2} |\nbg u|^2 + F(u) \, \Big) ,
\end{equation}
cf.~\cite[Remark~2.1]{ElliottRanner_CH}. In \cite{ElliottRanner_CH} it is stated, that to obtain a gradient flow on an evolving surface, a model for the surface velocity $v$ is needed, leading to a coupled system for $u$ and $v$. 
In the evolving surface case, $w = - \laplace_{\Ga\t} u + f(u)$  (with $f = F'$) is the variation of the evolving surface Ginzburg--Landau energy, see \cite{OConnorStinner}.   

\subsection{Weak formulation}
\label{section:weak form}

% By introducing the chemical potential $w: \GT \to \R$ we rewrite the Cahn--Hilliard equation \eqref{eq:CH - 4th order form} into a system of second-order partial differential equations: For $u,w: \GT \to \R$
% \begin{subequations}
% \label{eq:CH system}
% 	\begin{alignat}{3}
% 		\label{eq:CH system - u}
% 		\mat u - \laplace_{\Ga\t} w = &\ g'(u) - u (\nb_{\Ga\t} \cdot v) & \qquad & \text{on } \Ga\t , \\
% 		\label{eq:CH system - w}
% 		w +  \laplace_{\Ga\t} u  = &\  f'(u) & \qquad & \text{on } \Ga\t ,
% 	\end{alignat}
% \end{subequations}
% with initial data $u(\cdot,0) = u_0$ on $\Ga^0$. 

On the evolving surface $\Ga\t$ we recall the definition of standard Sobolev spaces $L^2(\Ga\t)$, and $H^1(\Ga\t)$ and its high-order variants, endowed with their usual norms, see \cite{DziukElliott_ESFEM,DziukElliott_L2}. We also refer to \cite{AlphonseElliottStinner_abstract,AlphonseElliottStinner_linear} for the definition of space-time  function spaces. 

The weak formulation of the Cahn--Hilliard system \eqref{eq:CH system} reads: Find $u(\cdot,t) \in H^1(\Ga\t)$ with a continuous-in-time material derivative $\mat u(\cdot,t) \in L^2(\Ga\t)$ and $w(\cdot,t) \in H^1(\Ga\t)$ such that for all test functions $\vphi^u(\cdot,t) \in H^1(\Ga\t)$ and $\vphi^w(\cdot,t) \in H^1(\Ga\t)$ 
 \begin{subequations}
\label{eq:CH WeakSolution}
	\begin{align}
	\label{eq:CH WeakSolution1}
	\int_{\Ga\t} \mat u \vphi^u + \int_{\Ga\t} \nb_{\Ga\t} w \cdot \nb_{\Ga\t} \vphi^u = &\ \int_{\Ga\t}  f(u,\nb_{\Ga \t} u) \, \vphi^u \\
	&- \int_{\Ga\t} u \, \vphi^u (\nb_{\Ga\t} \cdot v), \nonumber\\
	\label{eq:CH WeakSolution2}
	\int_{\Ga\t} w \vphi^w  -  \int_{\Ga\t} \nb_{\Ga\t} u \cdot \nb_{\Ga\t} \vphi^w = &\ \int_{\Ga\t}  g(u,\nb_{\Ga \t} u) \, \vphi^w ,
	\end{align}
\end{subequations}
with initial data $u(\cdot,0) = u_0$ on $\Ga^0$. 

%  [cb: Should we say something about the time regularity of $u$? Elliott a. Ranner use function spaces over the space-time manifold to conquer this, we seem to omit it here.]
%[kb: We tackle this differently: by explicitly sating to have a ``time-continuous material derivative'' in $L^2(\Gat)$. I think, that is equivalent to the abstract $L^2_{H^1}$ Hilbert space formulation of Elliott and Ranner, but avoids introducing all these spaces. If you prefer, we can write a remark-like paragraph on this equivalence. 

It is important to note here that the anti-symmetric structure of the above systems (\eqref{eq:CH system} and \eqref{eq:CH WeakSolution}) will serve as a key property which will be heavily used in the stability analysis. 

Using the Leibniz formula \cite{DziukElliott_ESFEM}, an equivalent weak form reads as: Find $u(\cdot,t) \in H^1(\Ga\t)$ with a continuous-in-time material derivative $\mat u(\cdot,t) \in L^2(\Ga\t)$ and $w(\cdot,t) \in H^1(\Ga\t)$ such that for all test functions $\vphi^u(\cdot,t) \in H^1(\Ga\t)$, with $\mat \vphi^u(\cdot,t) = 0$, and $\vphi^w(\cdot,t) \in H^1(\Ga\t)$ 
 \begin{subequations}
	\label{eq:CH WeakSolution - diff}
	\begin{align}
	\label{eq:CH WeakSolution - diff - u}
	\diff \Big( \int_{\Ga\t} u \vphi^u \Big) + \int_{\Ga\t} \nb_{\Ga\t} w \cdot \nb_{\Ga\t} \vphi^u = &\ \int_{\Ga\t}  f(u,\nb_{\Ga \t} u) \, \vphi^w , \\
	\label{eq:CH WeakSolution - diff - w}
	\int_{\Ga\t} w \vphi^w  -  \int_{\Ga\t} \nb_{\Ga\t} u \cdot \nb_{\Ga\t} \vphi^w = &\ \int_{\Ga\t}  g(u,\nb_{\Ga \t} u) \, \vphi^w .
	\end{align}
\end{subequations}

We note that as solution spaces for the weak problems one can equivalently use space-time Hilbert spaces, as it was done in \cite[Definition~2.1]{ElliottRanner_CH} (denoted, e.g. by $L^\infty_{H^1}$ and $L^2_{H^1}$ therein). For more details on these spaces we refer to \cite{AlphonseElliottStinner_abstract,AlphonseElliottStinner_linear,ElliottRanner_unified}.

\subsection{Abstract formulation}

%[kb: I would prefer not to change $u$ to $\eta$. I think it can cause more harm then benefit. I would rather emphasize that $u$ can be \emph{``any''}, \emph{``arbitrary''}, etc.] 

We will use the time-dependent bilinear forms, cf.~\cite{DziukElliott_L2,DziukElliott_acta}, for any $u, \vphi \in H^1(\Ga\t)$:
\begin{equation}
\label{eq:bilinear forms}
	\begin{aligned}
		m(t;u,\vphi) = &\ \int_{\Ga\t} \!\!\!  { u \, \vphi}, 
		\qquad 
		a(t;u,\vphi) = \int_{\Ga\t} \!\!\!  \nb_{\Ga\t} u \cdot \nb_{\Ga\t}\vphi, \\
		&\ \quad r(t;v;u,\vphi) = \int_{\Ga\t} \!\!\!  { u \,\vphi \, (\nb_{\Ga\t} \cdot v)}, 
		%\qquad & 
		%b(t;v;u,\vphi) = &\ \int_{\Ga\t} \!\!\!  {\mathcal{B}(v) \,\nb_{\Ga\t} u \cdot \nb_{\Ga\t} %\vphi},
	\end{aligned}
\end{equation}
%where
%
%\begin{equation*}
%	\mathcal{B}(v) = \tfrac{1}{2} (\nb_{\Ga\t} \cdot v) \text{ Id } - \tfrac{1}{2} \big( \nb_{\Ga\t} v + (\nb_{\Ga\t} v)^T \big) . 
%	D(v), \quadwith D(v)_{ij} = \frac{1}{2} (\underline{D}_{i} v_{j} + \underline{D}_{j} v_{i}) .
%\end{equation*}
We further define $a^*(t;\cdot,\cdot) = a(t;\cdot,\cdot) + m(t;\cdot,\cdot)$. All bilinear forms are symmetric in $u$ and $\vphi$, $m$ and $a^*$ are positive definite, while $a$ is positive semi-definite. Whenever it is possible, without confusion, we will omit the omnipresent time-dependence of the bilinear forms and write $m(\cdot,\cdot)$ instead of $m(t;\cdot,\cdot)$.%, cf.~\cite{DziukElliott_L2}.

We note here that the bilinear forms directly generate the (semi-)norms, for any $u \in H^1(\Ga\t)$: 
\begin{align*}
	\|u\|_{L^2(\Ga\t)}^2 = &\ m(u,u) , \\
	\|\nb_{\Ga\t} u\|_{L^2(\Ga\t)}^2 = &\ a(u,u) , \\
	\|u\|_{H^1(\Ga\t)}^2 = &\ a^*(u,u) .
\end{align*}

The weak formulation $\eqref{eq:CH WeakSolution}$ is rewritten, using the bilinear forms from above, as 
\begin{subequations}
	\begin{align*}
	%\label{eq: BilinearWeakSolution1}
	m(\mat u,\vphi^u) + a(w,\vphi^u) = &\ m(f(u, \nb_{\Ga \t} u),\vphi^w) - r(v;u,\vphi^u) , \\
	%\label{eq: BilinearWeakSolution2}
	m(w,\vphi^w) -  a(u,\vphi^w) = &\   m(g(u, \nb_{\Ga \t} u),\vphi^w) ,
	\end{align*}
\end{subequations}
and $\eqref{eq:CH WeakSolution - diff}$ is rewritten as
\begin{subequations}
	\begin{align*}
	\diff m(u,\vphi^u) + a(w,\vphi^u)  = &\ m(f(u, \nb_{\Ga \t} u),\vphi^w),\\
	m(w,\vphi^w) - a(u,\vphi^w) = &\  m(g(u, \nb_{\Ga \t} u),\vphi^w) .
	\end{align*}
\end{subequations}

The transport formula for the above bilinear forms, \cite[Remark~3.3]{DziukElliott_L2}, is used later on, and reads, for any $u(\cdot,t), \vphi(\cdot,t) \in L^2(\Ga\t)$ with $\mat u(\cdot,t), \mat \vphi(\cdot,t) \in L^2(\Ga\t)$ for all $0\leq t\leq T$:  
%\begin{subequations}
%	\begin{align}
%	\label{eq:transport formula - m}
%	\diff m(u,\vphi) &= m(\mat u,\vphi) + m(u,\mat \vphi) + r(v;u,\vphi) .%\\
%	%\label{eq:transport formula - a}
%	%\diff a(u,\vphi) &= a(\mat u,\vphi) + a(u,\mat \vphi) + b(v;u,\vphi) .
%	\end{align}
%	\label{eq:transport formula}
%\end{subequations}
\begin{align}	
	\diff m(u,\vphi) &= m(\mat u,\vphi) + m(u,\mat \vphi) + r(v;u,\vphi) .%\\
	%\diff m(u,\vphi^u) &= m(\mat u,\vphi^u) + m(u,\mat \vphi^u) + r(v;u,\vphi^u) .%\\
	\label{eq:transport formula - m}
\end{align}

\section{Semi-discretisation on evolving surfaces}
\label{section:semidiscrete CH equations}

For the numerical solution of the above examples we consider a high-order evolving surface finite element method. 
In the following, from \cite{DziukElliott_ESFEM,DziukElliott_acta,Demlow,highorderESFEM}, we briefly recall the construction of the discrete evolving surface, the high-order evolving surface finite element space, the lift operation, and the discrete bilinear forms, etc., which are used to discretise the Cahn--Hilliard equation of Section~\ref{section:CH}. 

\subsection{Evolving surface finite elements}
\label{subsection:ESFEM}

The smooth \emph{initial} surface $\Ga(0)$ is approximated by a $k$-order interpolating discrete surface, (a continuous, piecewise polynomial interpolation of $\Ga(0)$ of degree $k$ over a reference element), denoted by $\Ga_h(0) := \Ga_h^k(0)$, with vertices $p_j \in \Ga(0)$, $j=1,\dotsc,N$, and is given by the (high-order) triangulation, with maximal mesh width $h$. In the following, we refer to $\Ga_h$ as a triangulation, and to the Lagrange points $p_j$ as nodes. 
More details and the properties of such a discrete high-order initial surface are found in \cite[Section~2]{Demlow} and \cite[Section~3]{highorderESFEM}. 

The triangulation of the surface $\Ga\t$, denoted by $\Ga_h\t := \Ga_h^k\t$, is obtained by integrating the ODE \eqref{eq:ODE for positions} (with the known velocity $v$) from time $0$ to $t$ for all the nodes $p_j$ of the initial (high-order) triangulation. The nodes $x_j\t$ are on the exact surface $\Ga\t$ for all times. 
The discrete surface $\Ga_h\t$ remains to be an interpolation of $\Gat$ for all times. We always assume that the evolving (high-order) triangles are forming an admissible triangulation of the surface $\Ga\t$, which includes quasi-uniformity, and that the discrete surface is not a global double covering, cf.\ Section~5.1 of \cite{DziukElliott_ESFEM}. For more details (e.g.~on time-uniformity of geometric bounds) we refer to \cite[Section~3]{highorderESFEM}. 

The discrete tangential gradient on the discrete surface $\Ga_h\t$, of a function $\vphi_h: \Ga_h\t \to \mathbb{R}$ , is given by $\nb_{\Ga_h\t} \vphi_h = \nb { \bar \vphi_h} - (\nb { \bar\vphi_h} \cdot \nu_h) \nu_h$, understood in an element-wise sense, with $\nu_h$ denoting the normal to $\Ga_h\t$. (The discrete tangential gradient is independent of the arbitrary smooth extension $\bar \vphi_h$ onto a small neighbourhood of $\Ga_h\t$.)

The high-order evolving surface finite element space $S_h\t \nsubseteq H^1(\Ga\t)$ on $\Ga_h\t$ is spanned by continuous, piecewise linear nodal basis functions on $\Ga_h\t$ satisfying for each node $(x_j\t)_{j=1}^N$
$$
	\phi_i(x_j\t,t) = \delta_{ij}, \quadfor i,j = 1, \dotsc, N  \quadand 0 \leq t \leq T .
$$
The finite element space is given as
$$
	S_h\t = \textnormal{span}\{\phi_1(\cdot, t), \dotsc, \phi_N(\cdot, t)\} \quadfor 0 \leq t \leq T .
$$

%In the following, we will often write $\vphi_h \in S_h\t$, or $\vphi_h(\cdot,t) \in S_h\t$ instead of $\vphi_h(\cdot,t) \in S_h(t)$ for all $0 \leq t \leq T$ (i.e.~suppressing the time interval). 

The discrete velocity $V_h$ of the surface $\Ga_h\t$ is the evolving surface finite element interpolation of the surface velocity $v$ of $\Ga\t$, i.e.
\begin{equation}
\label{eq:discrete velocity}
	V_h(\cdot,t) = \sum_{j=1}^N v(x_j\t,t) \phi_j(\cdot,t) \quadfor 0 \leq t \leq T .
\end{equation}
The discrete material derivative is, for $0 \leq t \leq T$ , given by
\begin{equation}
\label{eq:discrete material derivative on Ga_h}
	\mat_h \vphi_h(\cdot,t) = \pa_t \bar \vphi_h(\cdot,t) + V_h \cdot \nb \bar \vphi_h(\cdot,t), \quadfora \vphi_h(\cdot,t) \in S_h\t,
\end{equation}
independent of $\bar \vphi_h$ as an arbitrary smooth extension of $\vphi_h$ onto a small neighbourhood of $\Ga_h \t$. 
The key \textit{transport property} of basis functions derived in Proposition 5.4 in \cite{DziukElliott_ESFEM}, is
\begin{equation}
\label{eq:transport property}
	\mat_h \phi_j(\cdot,t) = 0 , \qquad \textrm{for} \quad j=1,\dotsc,N \quadand 0 \leq t \leq T .
\end{equation}

\subsection{Lift}
\label{section:lift}

Following \cite{DziukElliott_ESFEM,Demlow}, we define the \emph{lift} operator $\cdot^\ell$ to compare functions on $\Ga_h\t$, with a sufficiently small $h \leq h_0$ (such that $\Ga_h\t$ is in a sufficiently small neighbourhood of $\Ga\t$),  with functions on $\Ga\t$. For functions $\vphi_h:\Ga_h\t \to \R$, we define the lift as 
\begin{equation}
\label{eq:lift definition}
	\vphi_h^\ell \colon \Ga\t \to\R \quad \text{with} \quad \vphi_h^\ell(y)=\vphi_h(x), \quad \forall x\in\Ga_h\t \quadfor 0 \leq t \leq T ,
\end{equation} 
where $y=y(x,t) \in \Ga\t$ is the \emph{unique} point on $\Ga\t$ with $x-y$  orthogonal to the tangent space $T_y\Ga\t$.
The \emph{inverse lift} $\vphi^{-\ell}:\Ga_h\t \to \R$ denotes a function whose lift is $\vphi:\Ga\t \to \R$.
Finally, the lifted finite element space is denoted by $S_h^\ell\t$, and is given as $$S_h^\ell\t = \big\{  \vphi_h^\ell \mid \vphi_h \in S_h\t \big\} , \quadfor 0 \leq t \leq T .$$

\subsection{Discrete bilinear forms}

The time-dependent discrete bilinear forms on $S_h\t$, i.e.~the discrete counterparts of $m,a$ and $g$, are given, for any $u_h,\vphi_h \in S_h\t$, by
\begin{equation}
\label{eq:discrete bilinear forms}
\begin{aligned}
m_h(t;u_h,\vphi_h) &= \int_{\Ga_h \t} \!\!\!\!\! { u_h \, \vphi_h}, 
\qquad 
a_h(t;u_h,\vphi_h) = \int_{\Ga_h \t} \!\!\!\!\! { \nb_{\Ga_h \t} u_h \cdot \nb_{\Ga_h \t} \vphi_h} , \\
& \quad r_h(t;V_h;u_h,\vphi_h) = \ \int_{\Ga_h\t} \!\!\!\!\! { u_h \,\vphi_h \, (\nb_{\Ga_h \t} \cdot V_h)}.
%\qquad & 
%b(t;v;u,\vphi) = &\ \int_{\Ga\t} \!\!\!  {\mathcal{B}(v) \,\nb_{\Ga\t} u \cdot \nb_{\Ga\t} %\vphi},
\end{aligned}
\end{equation}
%\begin{equation}
%\label{eq:discrete bilinear forms}
%	\begin{aligned}
%		m_h(t;u_h,\vphi_h) = \int_{\Ga_h \t} \!\!\! { u_h \, \vphi_h}, 
%		\qquad 
%		a_h(t;u_h,\vphi_h) = \int_{\Ga_h \t} \!\!\! { \nb_{\Ga_h \t} u_h \cdot \nb_{\Ga_h \t} \vphi_h} \\
%		\begin{aligned}
%			g_h(t;V_h;u_h,\vphi_h) = \ \int_{\Ga_h\t} \!\!\! { u \,\vphi_h \, (\nb_{\Ga_h \t} \cdot V_h)}, %\\
%%			b_h(t;V_h;u_h,\vphi_h) = &\ \int_{\Ga_h\t} \!\!\! {\mathcal{B}(V_h) \,\nb_{\Ga_h \t} u_h \cdot \nb_{\Ga_h \t} \vphi_h}, \qquad \quad 
%		\end{aligned}
%	\end{aligned}
%\end{equation}
%\begin{align*}
%m_h(t;u_h,\phi_h) := \int_{\Ga_h\t} {u_h \, \phi_h}, \qquad
%a_h(t;u_h,\phi_h) := \int_{\Ga_h\t} {\nbgh u_h \cdot \nbgh \phi_h},
%\end{align*}
%with $a_h^*(t;\cdot,\cdot) := a_h(t;\cdot,\cdot) + m_h(t;\cdot,\cdot)$ correspondingly and
%\begin{align*}
%&g_h(t;V_h;u_h,\phi_h) := \int_{\Ga_h\t} {u_h \, \phi_h \, (\nbgh \cdot V_h)}, \\
%&b_h(t;V_h;u_h,\phi_h) := \int_{\Ga_h\t} {\mathcal{B}_h(V_h) \, \nbgh u_h \cdot \nbgh \phi_h},
%%\label{eq: BilinearFormsDiscrete}
%\end{align*}
%where
%%
%\begin{equation}
%	\mathcal{B}_h(V_h) = \tfrac{1}{2} (\nb_{\Ga_h \t} \cdot V_h) \, \text{Id} - \tfrac{1}{2} \big( \nb_{\Ga_h \t} V_h + (\nb_{\Ga_h \t} V_h)^T \big) . 
%%	D_h(V_h), \quadwith D_h(V_h)_{ij} = \frac{1}{2} (\underline{D}_{h,i} (V_h)_{j} + \underline{D}_{h,j} V_{h,i}) .
%\label{eq:BhDh}
%\end{equation}
As in the continuous case we let $a_h^*(t;\cdot,\cdot) = a_h(t;\cdot,\cdot) + m_h(t;\cdot,\cdot)$. The discrete bilinear forms, clearly inherit the properties of their continuous counterparts, such as the transport formula \eqref{eq:transport formula - m}, see, e.g.~\cite{DziukElliott_L2,highorderESFEM}.

As in the continuous case, the discrete bilinear forms directly generate the discrete (semi-)norms, for any $u_h \in S_h\t$, 
\begin{align*}
\|u_h\|_{L^2(\Ga_h\t)}^2 = &\ m_h(u_h,u_h) , \\
\|\nb_{\Ga_h\t} u_h\|_{L^2(\Ga_h\t)}^2 = &\ a_h(u_h,u_h) , \\
\|u_h\|_{H^1(\Ga_h\t)}^2 = &\ a_h^*(u_h,u_h) .
\end{align*}

According to \cite{DziukElliott_ESFEM,Demlow}, the discrete norms and their continuous counterparts are $h$-uniformly equivalent, for any $\vphi_h \in S_h\t$ and $1 \leq q \leq \infty$,
\begin{equation}
\label{eq:norm equivalence}
	\begin{aligned}
		c \|\vphi_h^\ell\|_{L^q(\Gat)} &\leq \|\vphi_h\|_{L^q(\Ga_h \t)} \leq C \|\vphi_h^\ell\|_{L^q(\Gat)} , \\
		c \|\nb_{\Ga \t} \vphi_h^\ell\|_{L^q(\Gat)} &\leq \|\nb_{\Ga_h \t} \vphi_h\|_{L^q(\Ga_h \t} \leq C \|\nb_{\Ga \t}\vphi_h^\ell\|_{L^q(\Gat)} .
	\end{aligned}
\end{equation}

\subsection{Semi-discrete problem}
\label{section:semidiscrete problem}

The semi-discrete problem corresponding to the Cahn--Hilliard equation \eqref{eq:CH WeakSolution} reads: Find a solution $u_h(\cdot,t) \in S_h\t$ with continuous-in-time discrete material derivative $\mat_h u_h(\cdot,t) \in S_h\t$ and $w_h(\cdot,t) \in S_h\t$ such that for all test functions $\vphi_h^u(\cdot,t) \in S_h\t$ and $\vphi_h^w(\cdot,t) \in S_h\t$ 
% with integrals
%\begin{subequations}
%	\label{eq:semidiscrete problem}
%	\begin{align}
%	\label{eq:semidiscrete problem - u}
%	\int_{\Ga_h\t} \Big(\mat_h u_h \vphi_h^u + u_h \, \vphi_h^u (\nb_{\Ga_h\t} \cdot V_h)\Big) + \int_{\Ga_h\t} \nb_{\Ga_h\t} w_h \cdot \nb_{\Ga_h\t} \vphi_h^u = &\ 0, \\
%	\label{eq:semidiscrete problem - w}
%	\int_{\Ga_h\t} \eps \, \nb_{\Ga_h\t} u_h \cdot \nb_{\Ga_h\t} \vphi_h^w + \int_{\Ga_h\t} \eps\inv\,  f'(u_h) \, \vphi_h^w - \int_{\Ga_h\t} w_h \vphi_h^w = &\ 0,
%	\end{align}
%\end{subequations}
 \begin{subequations}
	\label{eq:semi-discrete problem - pre}
	\begin{align}
	\label{eq:semi-discrete problem - pre - u}
	m_h(\mat_h u_h , \vphi_h^u) + a_h(w_h , \vphi_h^u ) = &\ m_h( f(u_h, \nb_{\Ga_h \t} u_h), \vphi_h^u)\\
	&- r_h(V_h; u_h , \vphi_h^u) , \nonumber\\
	\label{eq:semi-discrete problem - pre - w}
	m_h(w_h , \vphi_h^w) -  a_h ( u_h , \vphi_h^w) = &\ m_h( g(u_h, \nb_{\Ga_h \t} u_h) , \vphi_h^w) ,
	\end{align}
\end{subequations} 
with given initial data $u_h(\cdot,0) = u_h^0$ on $\Ga_h^0$.

Equivalently, the semi-discrete problem corresponding to the weak form \eqref{eq:CH WeakSolution - diff}, using the discrete version of the transport formula \eqref{eq:transport formula - m} for \eqref{eq:semi-discrete problem - pre - u}, reads: Find a solution $u_h(\cdot,t) \in S_h\t$ with continuous-in-time discrete material derivative $\mat_h u_h(\cdot,t) \in S_h\t$ and $w_h(\cdot,t) \in S_h\t$ such that for all test functions $\vphi_h^u(\cdot,t) \in S_h\t$ with $\mat_h \vphi_h^u =0$ and $\vphi_h^w(\cdot,t) \in S_h\t$ 
% with integrals
%\begin{subequations}
%\label{eq:semi-discrete problem - pre - diff}
%	\begin{align} 
%	\label{eq:semi-discrete problem - pre - diff - u}
%	\diff \bigg( \int_{\Ga_h\t}  u_h \vphi_h^u \bigg) + \int_{\Ga_h\t} \nb_{\Ga_h\t} W_h \cdot \nb_{\Ga_h\t} \vphi_h^u = &\ \int_{\Ga_h\t} u_h \, \partial_h^{\bullet} \vphi_h^u ,
%	\\ 
%	\label{eq:semi-discrete problem - pre - diff - w}
%	\int_{\Ga_h\t} \eps \nb_{\Ga_h\t} u_h \cdot \nb_{\Ga_h\t} \vphi_h^w + \eps\inv W'(u_h) \vphi_h^w  = &\ \int_{\Ga_h\t} w_h \vphi_h^w ,
%	\end{align}
%\end{subequations}
 \begin{subequations}
	\label{eq:semi-discrete problem - pre - diff}
	\begin{align} 
	\label{eq:semi-discrete problem - pre - diff - u}
	\diff m_h(u_h , \vphi_h^u) + a_h(w_h , \vphi_h^u) = &\ m_h( f(u_h,  \nb_{\Ga_h \t} u_h) , \vphi_h^u) ,
	\\ 
	\label{eq:semi-discrete problem - pre - diff - w}
	m_h(w_h , \vphi_h^w) - a_h ( u_h , \vphi_h^w) = &\ m_h( g(u_h,  \nb_{\Ga_h \t} u_h) , \vphi_h^w) ,
	\end{align}
\end{subequations} 
again, with given initial data $u_h(\cdot,0) = u_h^0$ on $\Ga_h^0$.

%\redon 
%\sout{According to the following result, proved in \cite{ElliottRanner_CH}, the semi-discrete problem is well-posed, and the discrete material derivatives of both solution components are continuous in time. }
%\begin{proposition}[{\cite[Theorem~3.1]{ElliottRanner_CH}}]
%\label{proposition:semi-discrete well-posedness}
%	In the above semi-discrete setting, including the assumptions on the mesh, initial data, and non-linearities, the semi-discrete problem \eqref{eq:semi-discrete problem - pre} has a unique solution $u_h(\cdot,t), w_h(\cdot,t) \in S_h(t)$ with continuous-in-time discrete material derivatives $\mat_h u_h(\cdot,t), \mat_h w_h(\cdot,t) \in S_h(t)$, i.e.~their nodal values are $C^1$ in time.
%	
%	Furthermore, there exists $h_0 > 0$ and $C_E(\|u_h^0\|_{H^1(\Ga_h^0)}) > 0$ (also depending on $T$) such that for all $h \leq h_0$ the discrete energy $E_h(u_h) := \int_{\Ga_h\t} \tfrac{1}{2} |\nbgh u_h|^2 + W(u_h)$ satisfies
%	\begin{align*}
%		\sup_{t \in (0,T)} E_h(u_h(\cdot,t)) + \frac12 \int_0^T \|\nbgh w_h(\cdot,t)\|_{L^2(\Ga_h\t)}^2 \d t \leq  C_E(\|u_h^0\|_{H^1(\Ga_h^0)}) .
%	\end{align*}
%\end{proposition}

By a direct modification of the proof of Theorem~3.1 in \cite{ElliottRanner_CH} (based on standard ODE theory), we obtain that the above semi-discrete problem is well-posed, and the discrete material derivatives of both solution components are continuous in time, i.e.~the nodal values of the semi-discrete solution are both $C^1$ in time.
Therefore, for a given $u_h(\cdot,0) = u_h^0$, the initial value $w_h(\cdot,0) = w_h^0$ is obtained by solving the elliptic problem \eqref{eq:semi-discrete problem - pre - w} (or \eqref{eq:semi-discrete problem - pre - diff - w}) at time $t = 0$.

%\redoff 

\subsection{Matrix--vector formulation}

We collect the nodal values of $u_h(\cdot,t) = \sum_{j=1}^N u_j\t \phi_j(\cdot,t) \in S_h\t$ and $w_h(\cdot,t) = \sum_{j=1}^N w_j\t \phi_j(\cdot,t) \in S_h\t$, the solution pair of the semi-discrete problem \eqref{eq:semi-discrete problem - pre}, into the vectors $\bfu \t = (u_1\t,\ldots,u_N \t) \in \R^N$ and $\bfw \t = (w_1\t,\ldots,w_N \t) \in \R^N$. We define the time-dependent matrices, the mass and stiffness matrix, corresponding to the bilinear forms $m_h$ and $a_h$, respectively, and the non-linear terms involving $f$ and $g$:
\begin{equation}
\label{eq:FEM matrices}
\begin{aligned}
	\bfM\t|_{kj} = &\ m_h\big( \phi_j(\cdot,t) , \phi_k(\cdot,t) \big) , \\
	\bfA\t|_{kj} = &\ a_h\big( \phi_j(\cdot,t) , \phi_k(\cdot,t) \big) , \\
	 \bff(\bfu\t)|_{k}  = &\ m_h\big( f(u_h(\cdot,t), \nb_{\Ga_h \t} u_h(\cdot, t) ) , \phi_k(\cdot,t) \big) , \\
	 \bfg(\bfu\t)|_{k}  = &\ m_h\big( g(u_h(\cdot,t), \nb_{\Ga_h \t} u_h(\cdot, t) ) , \phi_k(\cdot,t) \big) , \\
\end{aligned}
\qquad j,k = 1,\dotsc,N .
\end{equation}
We further define the matrix corresponding to the bilinear form $a_h^*$:
\begin{equation*}
	\bfK\t = \bfM\t + \bfA\t .
\end{equation*}
We also note that, via the transport property \eqref{eq:transport property}, the time derivative of the mass matrix is given by
\begin{equation*}
	\bfMd\t|_{kj} = r_h(V_h(\cdot,t) ; \phi_j(\cdot,t) , \phi_k(\cdot,t)) .
\end{equation*}

The discrete material derivative of any surface finite element function $u_h(\cdot,t) \in S_h\t$, with nodal values $\bfu\t$, again by using the transport property \eqref{eq:transport property} of the basis functions and the product rule, is given by
\begin{equation}
\label{eq:discrete mat of an ESFEM func}
	\mat_h u_h(\cdot,t) = \mat_h \bigg( \sum_{j=1}^N u_j\t \phi_j(\cdot,t) \bigg) = \sum_{j=1}^N \dot u_j\t \phi_j(\cdot,t) .
\end{equation}
Thus, the nodal values of $\mat_h u_h$ are given by the vector $\dot \bfu\t$.

The finite element semi-discretisation of the Cahn--Hilliard equation \eqref{eq:semi-discrete problem - pre} then reads:
\begin{subequations}
\label{eq:matrix-vector form - pre}
	\begin{align}
		\label{eq:matrix-vector form - pre - u}
		\bfM\t \dot\bfu\t + \bfA\t \bfw\t = &\ \bff(\bfu\t) -\bfMd\t \bfu\t , \\
		\label{eq:matrix-vector form - pre - w}
		\bfM\t \bfw\t - \bfA\t \bfu\t = &\ \bfg(\bfu\t) .
	\end{align}
\end{subequations}
The anti-symmetric structure of \eqref{eq:matrix-vector form - pre}, which is shared with \eqref{eq:CH system} and \eqref{eq:semi-discrete problem - pre}, is recognised best in the rewritten form:
\begin{equation*}
\begin{bmatrix}
\bfM\t \dfrac{\d}{\d t} & \bfA\t \\
-\bfA\t \vphantom{\dfrac{\d}{\d t}} & \bfM\t
\end{bmatrix}
\begin{bmatrix}
\vphantom{\dfrac{\d}{\d t}} \bfu\t \\
\vphantom{\dfrac{\d}{\d t}} \bfw\t
\end{bmatrix}
= 
\begin{bmatrix}
\vphantom{\dfrac{\d}{\d t}} \bff(\bfu\t) -\bfMd\t \bfu\t \\
\vphantom{\dfrac{\d}{\d t}} \bfg(\bfu\t)
\end{bmatrix} .
\end{equation*}

In order to exploit this favourable structure, the stability analysis will use the matrix--vector system \eqref{eq:matrix-vector form - pre}.

For computations, it is however more advantageous to use the equivalent matrix--vector formulation
\begin{subequations}
	\label{eq:matrix-vector form - pre - diff}
	\begin{align}
	\label{eq:matrix-vector form - pre - diff - u}
	\diff \Big( \bfM\t \bfu\t \Big) + \bfA\t \bfw\t = &\ \bff(\bfu\t) , \\
	\label{eq:matrix-vector form - pre - diff - w}
	\bfM\t \bfw\t - \bfA\t \bfu\t = &\  \bfg(\bfu\t) ,
	\end{align}
\end{subequations}
where the surface velocity $V_h$ does not appear directly, as compared to the term with $\dot \bfM\t$ in \eqref{eq:matrix-vector form - pre}.

The $C^1$-regularity results stated after \eqref{eq:semi-discrete problem - pre - diff} translate to the modified system as well: the solutions $\bfu\t$ and $\bfw\t$ are both in $C^1(0,T;\R^N)$.

\subsection{A modified problem}
\label{section:modified problem}

The initial value $\bfu(0)$ is chosen suitably, on the other hand the initial value $\bfw(0)$ is obtained, from the second equation of the system \eqref{eq:matrix-vector form - pre}, or equivalently \eqref{eq:matrix-vector form - pre - diff}. Our error analysis requires the errors in both initial values to be $O(h^{k+1})$ in the $H^1(\Ga_h)$ norm. For $\bfu$ this is achieved using the Ritz map of $u^0$ (in which case the initial error in $\bfu$ will \emph{vanish}), however, such an error estimate is still not feasible for $\bfw$. Instead we transform the second equation such that the initial error in $\bfw$ also vanishes, in exchange for a time-independent (and small) inhomogeneity. 

To obtain optimal-order error estimates we modify the equation \eqref{eq:matrix-vector form - pre - w} (and equivalently \eqref{eq:matrix-vector form - pre - diff - w} as well) using a time-independent correction term. Let $\bar{\bfw}(0) \in \R^N$ denote the solution obtained from \eqref{eq:matrix-vector form - pre - w} at time $t = 0$, and let $\ws(0) \in \R^N$ contain the nodal values of the Ritz map of $w(0)$, and set
\begin{equation}
\label{eq:vartheta}
	\bfvartheta = \bfM(0) \big( \ws(0) - \bar{\bfw}(0) \big) \in \R^N .
\end{equation}
The second equation is then modified, such that the system \eqref{eq:matrix-vector form - pre} reads:
\begin{subequations}
\label{eq:matrix-vector form}
	\begin{align}
	\label{eq:matrix-vector form - u}
	\bfM\t \dot\bfu\t + \bfA\t \bfw\t = &\ \bff(\bfu\t) -\bfMd\t \bfu\t , \\
	\label{eq:matrix-vector form - w}
	\bfM\t \bfw\t - \bfA\t \bfu\t = &\ \bfg(\bfu\t) + \bfvartheta .
	\end{align}
\end{subequations}
Similarly, the equivalent system \eqref{eq:matrix-vector form - pre - diff} is modified to:
\begin{subequations}
	\label{eq:matrix-vector form - diff}
	\begin{align}
	\label{eq:matrix-vector form - diff - u}
	\diff \Big( \bfM\t \bfu\t \Big) + \bfA\t \bfw\t = &\ \bff(\bfu\t) , \\
	\label{eq:matrix-vector form - diff - w}
	\bfM\t \bfw\t - \bfA\t \bfu\t = &\ \bfg(\bfu\t) + \bfvartheta .
	\end{align}
\end{subequations}
The semi-discrete finite element formulations \eqref{eq:semi-discrete problem - pre} and \eqref{eq:semi-discrete problem - pre - diff} are modified accordingly. 

We recall that the solutions of the modified semi-discrete problems \eqref{eq:matrix-vector form} and \eqref{eq:matrix-vector form - diff} are both $C^1$ in time.

The initial value $\bfw(0)$ is obtained by solving the elliptic problem \eqref{eq:matrix-vector form - w} at $t = 0$, which, via \eqref{eq:vartheta} and \eqref{eq:matrix-vector form - pre - w}, yields
\begin{equation}
\label{eq:initial data for w}
\begin{aligned}
	\bfM(0) \bfw(0) = &\ \bfA(0) \bfu(0) + \bfg(\bfu(0)) + \bfvartheta \\
	= &\ \bfM(0) \bar{\bfw}(0) + \bfvartheta \\
%	\\
%	= &\ \bfM(0) \bar{\bfw}(0) + \bfM(0) \big( \ws(0) - \bar{\bfw}(0) \big) \\
	= &\ \bfM(0) \ws(0) .
\end{aligned}
\end{equation}

The advantage of the modified system is, that the errors in the initial data for $\bfw$ are included into the problem similarly to a residual term, which allows for a feasible weaker norm estimate of this term (in fact we will show later, that it is a defect term). Note that for the linear case, this is nothing else but shifting the solutions to a particular initial value using a constant inhomogeneity.

\section{Error estimates}
\label{section:error estimates}

We next state a new convergence result for the evolving surface finite element semi-discretisation  of polynomial degree $k \geq 1$ if the nonlinearities only depend on $u$, and of degree $k \geq 2$ if they also depend on $\nbg u$. In the theorem below, and in the remainder of this work, these two cases will be referred to as (a) and (b), respectively.  

\begin{theorem}
\label{theorem:semi-discrete convergence}
	Let $u$ and $w$ be the weak solutions of the Cahn--Hilliard equation on an evolving surface \eqref{eq:CH system}, and assume that they satisfy the regularity conditions \eqref{eq:regularity conditions}.
	
	Then, there exists an $h_0 > 0$ such that for all $h \leq h_0$ the errors between the solutions $u$ and $w$ and the evolving surface finite element solutions $u_h$ and $w_h$ of degree $k$, with nodal vectors solving the modified system \eqref{eq:matrix-vector form - diff}, and choosing the Ritz map of $u^0$ for the initial value $u_h(\cdot,0)$, satisfy the optimal-order uniform-in-time error estimates in both variables, for $0 \leq t \leq T$:
	
	 (a) For general nonlinearities $f$ and $g$ depending on $(u,\nbg u)$, for at least quadratic finite elements $k \geq 2$:
	\begin{equation*}
		\|u_h^\ell(\cdot,t) - u(\cdot,t)\|_{H^1(\Ga\t)} \leq C h^{k} ,
		\andquad
		\|w_h^\ell(\cdot,t) - w(\cdot,t)\|_{H^1(\Ga\t)} \leq C h^{k} ,
	\end{equation*}
	whereas the material derivative of the error in $u$ satisfies
	\begin{equation*}
	\begin{aligned}
		&\ \bigg( \int_0^t \|\mat (u_h^\ell(\cdot,s) - u(\cdot,s))\|_{H^1(\Ga(s))}^2  \d s \bigg)^{1/2} \leq C h^{k} .
	\end{aligned}
	\end{equation*}
	
	(b) If the nonlinearities are both \emph{independent} of $\nbg u$, then for any $k\geq 1$:
	\begin{alignat*}{3}
		\|u_h^\ell(\cdot,t) - u(\cdot,t)\|_{L^2(\Ga\t)} + h \|u_h^\ell(\cdot,t) - u(\cdot,t)\|_{H^1(\Ga\t)} \leq &\ C h^{k+1} , \\
		%		\intertext{and} 
		\|w_h^\ell(\cdot,t) - w(\cdot,t)\|_{L^2(\Ga\t)} + h \|w_h^\ell(\cdot,t) - w(\cdot,t)\|_{H^1(\Ga\t)} \leq &\ C h^{k+1} ,
	\end{alignat*}
	whereas the material derivative of the error in $u$ satisfies
	\begin{equation*}
		\begin{aligned}
			&\ \bigg( \int_0^t \|\mat (u_h^\ell(\cdot,s) - u(\cdot,s))\|_{L^2(\Ga(s))}^2 \\
			&\ \qquad + h \|\mat (u_h^\ell(\cdot,s) - u(\cdot,s))\|_{H^1(\Ga(s))}^2  \d s \bigg)^{1/2} \leq C h^{k+1} .
		\end{aligned}
	\end{equation*}

	The constant $C > 0$ is independent of $h$ and $t$, but depends on the bounds of the Sobolev norms of the solution $u$ and $w$, on the surface evolution, and on the length of the time interval $T$. 
%	We note, that the bound depends on an arbitrary $0 <\eps < 1$.
\end{theorem}

Sufficient regularity conditions on $u=u(\cdot,t
)$ and $w=w(\cdot,t)$ required by Theorem~\ref{theorem:semi-discrete convergence} are:
\begin{equation}
\label{eq:regularity conditions}
	\begin{aligned}
		&\ u , \mat u , (\mat)^{(2)} u \in H^{k+1}(\Ga\t) , \quad
		w , \mat w \in H^{k+1}(\Ga\t)  \quad \text{$L^2$-in-time} , \\
		&\ u\in W^{2,\infty}(\Ga\t) \cap H^{k+1}(\Ga\t) , \quad w \in H^{k+1}(\Ga\t) \quad \text{uniformly in time} , \\
		&\ \hspace*{-0.5cm}\text{and for the surface velocity:} \\
		&\ v , \mat v \in W^{k+1,\infty}(\Ga\t)  \quad \text{uniformly in time} .
	\end{aligned}
\end{equation} 
 Our result proves uniform-in-time error estimates in the $H^1$ and $L^2$ norms (in both cases (a) and (b)) for the error in $u$ and $w$ and for the errors in the material derivatives of $u$ (only sub-optimal in (a)). 

The classical Cahn--Hilliard equation (with a double-well potential) is naturally recovered in case (b), and slightly improves the result of \cite[Theorem~5.1]{ElliottRanner_CH}, proving a new time uniform estimate for the chemical potential.

Comparing our regularity assumptions to \cite[Theorem~5.1]{ElliottRanner_CH}: The spatial $H^{k+1}(\Ga\t)$ regularity assumptions \eqref{eq:regularity conditions} are required since we are using isoparametric evolving surface finite elements of degree $k$, whereas the assumptions on (further) material derivatives and the $L^\infty$-type and regularity assumptions on $u$ and $w$, and $\mat v$ \eqref{eq:regularity conditions} are required to obtain the uniform-in-time error estimates, via the new stability proof presented below.

\medskip
Theorem~\ref{theorem:semi-discrete convergence} is proved by studying the questions of stability and consistency. 
The consistency of the algorithm is shown by proving high-order estimates for the defects (the error obtained by inserting the Ritz map of the exact solutions into the method), which are obtained by using geometric and approximation error estimates for high-order evolving surface finite elements from \cite{highorderESFEM}, which combines techniques of \cite{DziukElliott_ESFEM,DziukElliott_L2} and \cite{Demlow}.

The main issue in the proof is stability, i.e.~a mesh independent, uniform-in-time bound of the errors in terms of the defects. The main idea of the stability proof was originally developed for Willmore flow \cite{Willmore}, and it relies on \emph{energy estimates} that exploit the \emph{anti-symmetric structure} of the Cahn--Hilliard equation, see \eqref{eq:CH system}, \eqref{eq:semi-discrete problem - pre}, and \eqref{eq:matrix-vector form}. The basic idea of the stability proof is concisely sketched in Figure~\ref{fig:energy estimates}.
In order to estimate the non-linear terms, a key issue in the stability proof is to ensure that the $W^{1,\infty}$ norm of the error in $u$ remains bounded. The uniform-in-time $H^1$ norm error bounds together with an inverse estimate provide a bound in the $W^{1,\infty}$ norm. Similarly, it is also possible to show such a $W^{1,\infty}$ norm bound for the error in $w$, provided by our uniform-in-time $H^1$ norm bounds in both $u$ and $w$.

%with a possibly non-linear operator $\mathcal{A}$. Note that $\mathcal{A} = \text{Id}$ yields the Laplace--Beltrami operator. Now, we are able to use the uniform-in-time convergence result for $w$ to bound $w_h$ for a $t^* > 0$ by
%%
%\begin{equation*}
%\begin{aligned}
%\|w_h(\cdot,t)\|_{L^\infty(\Ga_h \t)} &= \|w_h^*(\cdot,t) - e_{w_h}(\cdot,t)\|_{L^\infty(\Ga_h \t)} \\
%&\leq \|w_h^*(\cdot,t)\|_{L^\infty(\Ga_h \t)} + \|e_{w_h}(\cdot,t)\|_{L^\infty(\Ga_h \t)} \leq C,
%\end{aligned}
%\end{equation*}
%for all $ 0 \leq t \leq t^*$, and then extend the bound to the whole time domain $t \in [0,T]$, as will be done in the proof of the stability result in estimate $\eqref{eq: InverseInequality}$.
%
%In addition, we obtain an optimal-order $L^2$-in-time convergence estimate for the full $H^1$ norm error for the material derivative $\partial_h^{\bullet} u$.

\section{Stability}
\label{section:stability}

\subsection{Preliminaries}

This section is dedicated to the definition of a few concepts, such as the comparison of various quantities on different discrete surfaces and a generalised Ritz map, which are all used throughout the stability analysis.

The finite element matrices $\bfM\t$, $\bfA\t$, and $\bfK\t$ induce (semi-)norms which correspond to discrete Sobolev (semi-)norms:
\begin{equation}
\label{eq: MatrixNorms}
	\begin{aligned}
		\|\bfw\|_{\bfM\t}^2 &=  \bfwt \bfM\t \bfw = \|w_h\|_{L^2{(\Ga_h\t)}}^2,\\
		\|\bfw\|_{\bfA\t}^2 &=  \bfwt \bfA\t \bfw = \|\nb_{\Ga_h\t} w_h\|_{L^2{(\Ga_h\t)}}^2, \quad \text{ and } \\
		\|\bfw\|_{\bfK\t}^2 &= \|\bfw\|_{\bfM\t}^2 + \|\bfw\|_{\bfA\t}^2 = \|w_h\|_{H^1{(\Ga_h\t)}}^2,
	\end{aligned}
\end{equation}
for any vector $\bfw \in \R^N$ corresponding to the finite element function $w_h \in S_h\t$.

%The following result, proved in \cite[Lemma~4.1]{DziukLubichMansour_rksurf}, compares quantities on surfaces at different times, We refer to  for basic estimates, that are used for the proof of the lemma thereafter.
%%
%\begin{lemma}
%	\label{lemma:MatrixEstimate}
%	There are constants $\alpha$ and $\beta$, that are independent of $h$ and the final time $T$ such that
%	%
%	\begin{equation*}
%	\begin{aligned}
%	\bfwt \big(\bfM(s) - \bfM\t\big) \bfz &\leq (e^{\alpha (s - t)} - 1) \,\|\bfw\|_{\bfM\t}\,\|\bfz\|_{\bfM\t}, \\
%	\bfwt \big( \bfA(s) - \bfA\t\big) \bfz &\leq (e^{\beta (s - t)} - 1) \,\|\bfw\|_{\bfA\t}\,\|\bfz\|_{\bfA\t},
%	\end{aligned}
%	%\label{eq: MatrixEstimate}
%	\end{equation*}
%	for all vectors $\bfw, \bfz \in \R^N$ and $0 \leq t \leq s \leq T$.
%\end{lemma}
%We will often use this lemma when $s$ is close to $t$, therefore we have $e^{c (s - t)} - 1 \leq 2 c (s-t)$.
%, and obtain that
%\begin{equation}
%\label{eq:matrix difference}
%	\text{the norms $\|\cdot\|_{\bfM\t}$ and $\|\cdot\|_{\bfA\t}$ are $h$ uniformly equivalent for $0 \leq t \leq T$.}
%\end{equation}

From \cite[Lemma~4.6]{KLLP2017} we recall the following estimates for the time derivatives of the mass and stiffness matrix, and, additionally, we prove that they also hold for the second order time derivatives. 
\begin{lemma}
\label{lemma:MatrixDerivativeEstimate}
	For all vectors $\bfw, \bfz \in \R^N$ we have
	\begin{subequations}
		\label{eq: MatrixDerivativeEstimate}
		\begin{align}
		\label{eq: MatrixDerivativeEstimate1}
		\bfwt \bfMd\t \bfz &\leq c \,\|\bfw\|_{\bfM\t} \,\|\bfz\|_{\bfM\t}, \\
		\label{eq: MatrixDerivativeEstimate2}
		\bfwt \bfAd\t \bfz &\leq c \,\|\bfw\|_{\bfA\t} \,\|\bfz\|_{\bfA\t},\\
		\label{eq: MatrixDerivativeEstimate3}
		\bfwt \bfMdd\t \bfz &\leq c \,\|\bfw\|_{\bfM\t} \,\|\bfz\|_{\bfM\t}, \\
		\label{eq: MatrixDerivativeEstimate4}
		\bfwt \bfAdd\t \bfz &\leq c \,\|\bfw\|_{\bfA\t} \,\|\bfz\|_{\bfA\t},
		\end{align}
	\end{subequations}
	where the constant $c > 0$ is independent of $h$, but depends on the surface velocity $v$.
\end{lemma}
\begin{proof}
	The first two estimates were shown in Lemma~4.6 of \cite{KLLP2017}.
	
	We prove the estimate \eqref{eq: MatrixDerivativeEstimate3} for the second derivative of the mass matrix. For fixed vectors $\bfw, \bfz \in \R^N$ corresponding to discrete functions $w_h(\cdot,t), z_h(\cdot,t) \in S_h\t$ (for $0 \leq t \leq T$), we have $\mat_h w_h(\cdot,t) = \mat_h z_h(\cdot,t) = 0$ by the transport property \eqref{eq:transport property}, see \eqref{eq:discrete mat of an ESFEM func}. Using the discrete version of the Leibniz formula \cite[Lemma~2.2]{DziukElliott_ESFEM} or \cite[Lemma~4.2]{DziukElliott_L2} twice, we obtain
	\begin{align*}
		&\ \bfwt \bfMdd\t \bfz 
		= \frac{\d^2}{\d t^2} \int_{\Ga_h\t}{ w_h\, z_h} 
%		\\ =&\ 
		= \diff \int_{\Ga_h\t}{ w_h\, z_h (\nbgh \cdot V_h)} \\
		=&\ \int_{\Ga_h\t}{ w_h \,z_h \,\mat_h (\nbgh \cdot V_h)} + \int_{\Ga_h\t}{w_h \,z_h  \,(\nbgh \cdot V_h)^2} .
	\end{align*}	
	We remind here that the discrete spatial differential operators, and hence the integrals, are understood in an element-wise sense. 
	
	To estimate the first integral we recall how to interchange surface differential operators with the material derivative \cite[Lemma~2.6]{DziukKronerMuller}. For discrete differential operators they read:
	\begin{equation}
	\label{eq:interchange formulas}
		\begin{aligned}
			\mat_h(\nbgh w_h) = &\ \nbgh (\mat_h w_h)  - (I - \nu_h \nu_h^T) \nbgh V_h \cdot \nbgh w_h , \\
			\mat_h(\nbgh \cdot w_h) = &\ \nbgh \cdot \mat_h w_h  - (I - \nu_h \nu_h^T) \nbgh V_h : \nbgh w_h ,
		\end{aligned}
	\end{equation}
	understood element-wise, for $w_h:\Ga_h\t \to \R$ and $w_h:\Ga_h\t \to \R^3$, respectively. 
	Then, the second formula from \eqref{eq:interchange formulas} is used to estimate the first integral, together with the bounds on the discrete velocity $V_h$. 
	The boundedness of $V_h$ is implied by the sufficient regularity of the velocity $v$, and recalling that $V_h$ is the interpolation of $v$, cf.~\eqref{eq:discrete velocity}, see Lemma~\ref{lemma:VhBhBound} or \cite[Lemma~3.1.6]{Beschle_thesis}. We altogether obtain
	\begin{align*}
		\int_{\Ga_h\t}{ w_h \,z_h \,\mat_h (\nbgh \cdot V_h)} \leq &\ \|w_h\|_{L^2(\Ga_h\t} \,\|z_h\|_{L^2(\Ga_h\t)} \,\|\mat_h (\nbgh \cdot V_h)\|_{L^{\infty}(\Ga_h\t)} \\ 
		\leq &\ c \,\|\bfw\|_{\bfM\t} \,\|\bfz\|_{\bfM\t}.
	\end{align*}
	The second integral is directly bounded by
	\begin{align*}
	\int_{\Ga_h\t}{w_h \,z_h  \,(\nbgh \cdot V_h)^2} \leq &\ \|w_h\|_{L^2(\Ga_h\t} \,\|z_h\|_{L^2(\Ga_h\t)} \,\|\nbgh \cdot V_h\|_{L^{\infty}(\Ga_h\t)}^2 \\
	 \leq &\ c \,\|\bfw\|_{\bfM\t} \,\|\bfz\|_{\bfM\t} .
	\end{align*}
	
	The estimate for the stiffness matrix is shown by analogous arguments, now using the interchange formula \cite[Equation~$(7.27)$]{MCF}, and the analogous version of \eqref{eq:interchange formulas} for the first order differential operator appearing in the transport formula for the stiffness matrix \cite[Lemma~4.2,~$(4.18)$]{DziukElliott_L2}.
	
\qed\end{proof}

\subsection{Error equations and defects}
\label{section:error equations}

Before turning to the stability analysis, 
let us define a Ritz map of the exact solution onto the evolving surface finite element space, from \cite{LubichMansour_wave,highorderESFEM} we recall the definition of a time-dependent Ritz map on evolving surfaces: $R_h : H^1(\Gat) \to S_h^\ell\t$, (here we do not include the velocity term of \cite{LubichMansour_wave}). 

Let $u(\cdot,t) \in H^1(\Ga\t)$ for $0 \leq t \leq T$ be arbitrary. Then, the Ritz map is defined through $\widetilde R_h \t u \in S_h\t$ which satisfies, for all $\vphi_h \in S_h \t$,
\begin{equation}
\label{eq:definition Ritz map}
	\begin{aligned}
		a_h^*( \widetilde R_h\t u, \vphi_h) = a^*( u, \vphi_h^\ell) . 
%		+ m( \zeta , (v - v_h) \cdot \nbg \vphi_h^\ell ) .
	\end{aligned}
\end{equation}
The Ritz map is then defined as the lift of $\widetilde{R}_h\t$, i.e.~$R_h\t u  = (\widetilde R_h\t u)^\ell \in S_h^\ell\t$. We will often suppress the omnipresent time-dependency of the Ritz map. 
%In the mass term above, $v_h$ denotes the surface velocity of the lifted material points. It is defined by the velocity of the edges of the curved simplices in the lifted triangulation $\Ga_h^\ell\t = \Ga\t$, using the discrete surface velocity $V_h$. For more details we refer to \cite[Definition~4.3]{DziukElliott_L2}, or to \eqref{eq:definition v_h}.
%
%The motion of the edges of the simplices in the lifted triangulation defines a surface velocity for the smooth surface $\Ga \t$. We denote it by $v_h$ and through it we obtain a further discrete material derivative for functions living on the lifted evolving surface finite element space $S_h^\ell \t$. It is defined as
%\begin{equation*}
%\mat_h \vphi_h := \partial_t \vphi_h + v_h \cdot \nb \vphi_h,
%\end{equation*}
%it is the lift of the discrete material derivative such that for all $\vphi_h \in S_h \t$ with lift $\vphi_h^\ell \in S_h^\ell \t$ we have
%\begin{equation*}
%	\mat_h \vphi_h^\ell = (\mat_h \vphi_h)^\ell
%\end{equation*}
%and hence it has the transport property of basis functions \eqref{eq:transport property} as well.
%
In \cite{LubichMansour_wave} it was shown that the above Ritz map is well-defined, %high-order 
error estimates for the high-order evolving surface FEM were shown in \cite{highorderESFEM}, and are recalled in Lemma~\ref{lemma:Ritz map error}. We note, that the Ritz map used here differs from the one used by Elliott and Ranner in \cite{ElliottRanner_CH}, and the references therein, as it involves the bilinear form $a^*$ instead of $a$ together with the average condition. 

%\blueon Since we will use Ritz maps with different $\zeta$ functions, we will use the notation $R_h^\zeta$ to clearly  denote our particular choice. In case of $\zeta = 0$ we will simply write $R_h := R_h^0$.

Let us consider now the (unlifted) Ritz map of the exact solutions $u$ and $w$ of \eqref{eq:CH system}, which are denoted by 
\begin{equation*}
u_h^*(\cdot,t) = \widetilde R_h\t u(\cdot,t) \in S_h\t \andquad w_h^*(\cdot,t) = \widetilde R_h\t w(\cdot,t) \in S_h\t ,
\end{equation*} \blueoff 
whose nodal values are collected into the vectors 
\begin{equation*}
\us\t \in \R^N \andquad \ws\t \in \R^N.
\end{equation*}
The nodal vectors of the Ritz maps of the exact solutions satisfy the system \eqref{eq:matrix-vector form - pre} only up to some defects $\du\t$ and $\dw\t$ in $\R^N$, corresponding to the finite element functions $d_h^u(\cdot,t)$ and $d_h^w(\cdot,t)$ in $S_h\t$:
\begin{subequations}
	\label{eq:defect definition}
	\begin{align}
	\bfM\t \dotus\t + \bfA\t \ws\t = &\ \bff(\us\t) - \bfMd\t \us\t + \bfM\t \du\t , \\
	\bfM\t \ws\t - \bfA\t \us\t = &\ \bfg(\us\t) + \bfM\t \dw\t .
	\end{align}
\end{subequations}

The errors between the nodal values of the semi-discrete solutions and of the Ritz maps of the exact solutions are denoted by $\eu\t = \bfu\t - \us\t$ and $\ew\t = \bfw\t - \ws\t$ in $\R^N$. 
By subtracting \eqref{eq:defect definition} from \eqref{eq:matrix-vector form} we obtain that the errors $ \eu$ and $\ew$ (corresponding to the functions $e_{u_h}$ and $e_{w_h} \in S_h\t$) satisfy the following error equations:
\begin{subequations}
	\label{eq:error equations}
	\begin{align}
	\label{eq:error equations - u}
	\bfM\t \bfeud\t + \bfA\t \ew\t = &\ \Big( \bff(\bfu\t) - \bff(\us\t) \Big) \\
	&-\bfMd\t \eu\t - \bfM\t \du\t ,\nonumber \\
	\label{eq:error equations - w}
	\bfM\t \ew\t - \bfA\t \eu\t = &\ \Big( \bfg(\bfu\t) - \bfg(\us\t) \Big) \nonumber \\
	&\ - \bfM\t \dw\t  + \bfvartheta ,
	\end{align}
\end{subequations}
with zero initial values $\eu(0) = 0$ and $\ew(0) = 0$. 
Both initial values indeed vanish by construction: for $\eu(0)$ recall that we choose $\bfu(0)$ to be the nodal values of the Ritz map of $u^0$ and $\us(t)$ contains the nodal values of the Ritz map of $u$ for all $t$, while for $\ew(0)$ we have $\bfw(0) = \ws(0)$ by the construction \eqref{eq:initial data for w}.  

Since the initial values also satisfy \eqref{eq:error equations - w} at $t = 0$, we obtain the useful expression
\begin{equation}
\label{eq:vartheta equals defect d_w}
	\bfvartheta = \bfM(0) \dw(0) .
\end{equation}

%For a defect $\bfd \in \R^N$, corresponding to $d_h \in S_h\t$, we use the dual norm, cf.~\cite[proof of Theorem~5.1]{LubichMansourVenkataraman_bdsurf}, [cb: the norm is only mentioned in the citation as a side note, maybe there is a better source] defined by
%%
%\begin{equation}
%\label{eq:dualnorm1}
%\begin{aligned}
%	\|\bfd\|_{*,t}^2 := \|d_h \|_{{H_h\inv}(\Ga_h\t)}^2 := \sup_{0 \neq \vphi_h\t \in S_h\t} \frac{\int_{\Ga_h\t}{ d_h \, \vphi_h\t}}{\|\vphi_h\t\|_{H^1(\Ga_h\t)}} 
%	\\
%	&= \sup_{0 \neq \bfz \in \R^N} \frac{\bfdut \bfM\t \bfz }{\bfz \big(\bfM\t + \bfA\t\big) \bfz^{\frac{1}{2}}} \\
%	&= \sup_{0 \neq \bfw \in \R^N} \frac{\bfdut \bfM\t \big(\bfM\t + \bfA\t\big)^{-\frac{1}{2}} \bfw}{\big(\bfw , \bfw\big)^{\frac{1}{2}}} \\
%	&= \|\big(\bfM\t + \bfA\t\big)^{-\frac{1}{2}} \bfM\t \bfdu\|_2 \\
%	&=  \bfdut \bfM\t \big(\bfM\t + \bfA\t\big)\inv \bfM\t \bfdu^{\frac{1}{2}}.
%\end{aligned}
%\end{equation}
%Which, moreover, satisfies
%%
%\begin{equation*}
%	\|\bfd\|_{*,t}^2 =  \bfd^T \bfM\t \bfK\t \bfM\t \bfd .
%\end{equation*}

\subsection{Stability bounds}
\label{section:stability proof}

\begin{proposition}
\label{proposition:stability}
	Suppose there exists a constant $c > 0$ independent of $h$ and $t$ such that the defects are bounded for a $\kappa \geq 2$ by
	\begin{equation}
	\label{eq:defect bounds - assumed}
		\begin{alignedat}{3}
			\|\bfdu\t\|_{\bfM\t} \leq &\ c h^{\kappa} ,  \qquad & 
			\|\bfdud\t\|_{\bfM\t} \leq &\ c h^{\kappa} , \\
			\|\bfdw\t\|_{\bfM\t} \leq &\ c h^{\kappa} , \qquad & 
			\|\bfdwd\t\|_{\bfM\t} \leq &\ c h^{\kappa} ,
		\end{alignedat}
		\qquad t \in [0,T] .
	\end{equation}
	Furthermore, suppose that for all $0\leq t \leq T$ the Ritz maps $u_h^* = \widetilde R_h u$ and $w_h^* = \widetilde R_h w$ satisfy the bounds $\|u_h^*(\cdot,t)\|_{W^{1,\infty}(\Ga_h\t)} \leq M$ and $\|w_h^*(\cdot,t)\|_{W^{1,\infty}(\Ga_h\t)} \leq M$.
	
	Then, there exists $h_0 > 0$ such that the following error bound holds for $h \leq h_0$ and $0 \leq t \leq T$:
	\begin{equation}
	\label{eq:stability bound}
	\begin{aligned}
	&\ \|\bfeu\t\|_{\bfK\t}^2 + \|\bfew\t\|_{\bfK\t}^2 + \int_0^t \!\! {\|\bfeud (s) \|_{\bfK\s}^2} \d s\\
	\leq &\  C \int_0^t \!\! \|\bfdu (s)\|_{\bfM(s)}^2 +  \|\bfdud (s)\|_{\bfM(s)}^2 + \|\bfdw (s)\|_{\bfM(s)}^2 + \|\bfdwd (s)\|_{\bfM(s)}^2 \d s \\
	&\ + C\,\|\bfdu\t\|_{\bfM\t}^2 + C t \|\dw(0)\|_{\bfM(0)}^2 .
	\end{aligned}
	\end{equation}
	The constant $C > 0$ is independent of $t$ and $h$, but depends exponentially on the final time $T$. 
%	We note, that for the more general case the bound also depends on $0 <\eps < 1$.
	%
\end{proposition}

In Section~\ref{section:consistency}, Proposition~\ref{proposition:consistency}, we show that the defects are in fact bounded as $O(h^{k+1})$. 

\begin{proof}	
The proof is based on energy estimates, and its basic idea is very similar to that of \cite{Willmore}. 
Proving uniform-in-time $H^1$ norm error estimates is essential for handling the non-linear term, which is done by deriving a $W^{1,\infty}$ norm bound for the errors using an inverse estimate.

In order to achieve a uniform-in-time stability bound, two sets of energy estimates are required. These energy estimates strongly exploit the anti-symmetric structure of \eqref{eq:CH system}. (i) In the first one, an energy estimate is proved for $\eu$, but comes with a critical term involving $\doteu$. (ii) The second estimate uses the time derivative of \eqref{eq:error equations - w}, leads to a bound of this critical term and also to a uniform-in-time bound for $\ew$. The combination of these two energy estimates gives the above stability bound.
The structure and basic idea of the proof is sketched in Figure~\ref{fig:energy estimates}.
\begin{figure}[htbp]
	\begin{center}
		\includegraphics[width=\textwidth]{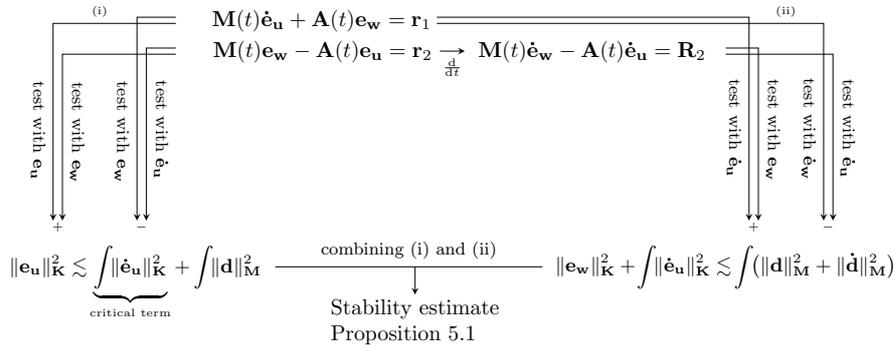}
		\caption{Sketch of the structure of the energy estimates for the stability proof. In the diagram $\bfr_1$ and $\bfr_2$ denote the right-hand sides of \eqref{eq:error equations - u} and \eqref{eq:error equations - w}. (Note that, after time differentiation, the term $\bfR_2$ not only contains the time derivative of $\bfr_2$, but other terms involving derivatives of matrices as well.)}
		\label{fig:energy estimates}
	\end{center}
\end{figure}
%[cb: norms on defects in figure have to be corrected, before $\bfA \t$ is $\eps$ and not $\eps\inv$. kb: Corrected!] 
In order to handle the non-linear terms we first prove the stability bound on a time interval where the $W^{1,\infty}$ norm of $e_h^u$ is small enough, and then show that this time interval can be enlarged up to $T$.

In the following $c$ and $C$ are generic constants that take different values on different occurrences. Whenever it is possible, without confusion, we omit the argument $t$ of time-dependent vectors but not of time-dependent matrices. By $\varrho_j> 0$ we will denote small numbers, used in Young's inequalities for different absorptions, and hence we will often incorporate $h$ independent multiplicative constants into those, yet unchosen, factors.

We start by stating that there exists a maximal time $0 < t^* \leq T$ such that, for all $t \leq t^*$,
\begin{equation}
\label{eq: tStarBound}
	\|e_{u_h}(\cdot,t)\|_{W^{1,\infty}(\Ga_h\t)} \leq h^\frac{\kappa - d/2}{2} ,
	\quadfora  0 \leq t \leq t^*.
\end{equation}
Since $e_{u_h}(\cdot,0) = 0$ and since $u_h$ and $u_h^*$,  respectively their spatial derivatives $\nb_{\Ga_h} u_h$ and $\nb_{\Ga_h} u_h^*$  are continuous in time, we directly infer that $t^* > 0$.

Thus, by the assumption that the Ritz maps of the exact solutions satisfy $\|u_h^*\t\|_{W^{1,\infty}(\Ga_h\t)}, \|w_h^*\t\|_{W^{1,\infty}(\Ga_h\t)} \leq M$, with a finite constant $M>0$, we obtain the following bound for the numerical solution:
\begin{equation}
\label{eq:UhBound2}
	\begin{aligned}
		\|u_h(\cdot,t)\|_{W^{1,\infty}(\Ga_h\t)} &= \|u_h^*(\cdot,t) - e_{u_h}(\cdot,t)\|_{W^{1,\infty}(\Ga_h\t)} \\
		&\leq \|u_h^*(\cdot,t)\|_{W^{1,\infty}(\Ga_h\t)} + \|e_{u_h}(\cdot,t)\|_{W^{1,\infty}(\Ga_h\t)} \leq 2M,
	\end{aligned}
\end{equation}
for all $0 \leq t \leq t^*$ and for $h \leq h_0$ sufficiently small, and similarly for $w_h$. Thus, for  $f \in C(\R \times \R^d)$
\begin{equation}
 \|f\big(u_h(\cdot,t), \nb_{\Ga_h \t} u_h(\cdot,t)\big)\|_{L^{\infty}(\Ga_h\t)} \leq C ,
\label{eq:fBound}
\end{equation}
for all $0 \leq t \leq t^*$ and  $h \leq h_0$ sufficiently small.
We first prove the stated stability bound for $0 \leq t \leq t^*$, and then show that indeed $t^*$ coincides with $T$. 

\textit{Energy estimate (i):} We take the first error equation \eqref{eq:error equations - u} and test it with $\bfeu$, while the second one \eqref{eq:error equations - w} is tested by $\bfew$, to obtain
\begin{equation*}
	\begin{aligned}
		\bfeut \bfM\t \bfeud + \bfeut \bfA\t \bfew = &\ \bfeut \big(\bff(\bfu\t) - \bff(\bfu^*\t)\big) \\
		&\ - \bfeut \bfMd\t  \bfeu - \bfeut \bfM\t  \bfdu, \\
		\bfewt \bfM\t \bfew - \bfewt \bfA\t \bfeu = &\ \bfewt \big(\bfg(\bfu\t) - \bfg(\bfu^*\t)\big) \\
		&\ - \bfewt  \bfM\t \bfdw + \bfewt \bfvartheta .
	\end{aligned} 
\end{equation*}
By adding the two equations, and by the symmetry of $\bfA$, we eliminate the mixed term $\bfeut \bfA\t \bfew$, and obtain
\begin{align*}
	\bfeut \bfM\t \bfeud +  \, \bfewt \bfM\t \bfew  = &\ - \bfeut \bfMd\t \bfeu   \\
	&\ + \bfeut \big(\bff(\bfu\t) - \bff(\bfu^*\t)\big) \\
	&\ + \,  \bfewt \big(\bfg(\bfu\t) - \bfg(\bfu^*\t) \big) \\
	&\ - \bfeut \bfM\t \bfdu  -  \bfewt \bfM\t \bfdw  + \bfewt \bfvartheta .
\end{align*}
Using the product rule and symmetry of $\bfM$ we rewrite the first term as
\begin{equation*}
\bfeut \bfM\t \bfeud  = \frac{1}{2} \,\diff \big( \bfeut \bfM\t \bfeu \big) - \frac{1}{2} \, \bfeut \bfMd\t \bfeu ,
\end{equation*}
which altogether yields
\begin{equation*}
	\begin{aligned}
		\frac{1}{2} \diff \|\bfeu\|_{\bfM\t}^2 + \|\bfew\|_{\bfM\t}^2 
		= &\ -\frac{1}{2} \bfeut \bfMd\t \bfeu  \\
		&\ + \bfeut \big(\bff(\bfu\t) - \bff(\bfu^*\t) \big) \\
		&\ + \bfewt \big(\bfg(\bfu\t) - \bfg(\bfu^*\t) \big) \\
		&\ -  \bfeut \bfM\t \bfdu  - \bfewt \bfM\t \bfdw  + \bfewt \bfvartheta .
	\end{aligned}
\end{equation*}
Similarly, we test \eqref{eq:error equations - u} by $\bfew$ and \eqref{eq:error equations - w} by $\bfeud$,
% to obtain
%%
%\begin{equation*}
%	\begin{aligned}
%		\bfewt \bfM\t \bfeud  +  \bfewt \bfA\t \bfew  = &\ - \bfewt \bfMd\t \bfeu - \bfewt \bfM\t \bfdu , \\
%		\bfeudt \bfM\t \bfew  - \eps \,  \bfeudt \bfA\t \bfeu   = &\ \eps\inv \,  \bfeudt \big( \bfW(\bfu\t) - \bfW(\bfu^*\t) \big) \\
%		&\ + \bfeudt \bfM\t \bfdw .
%	\end{aligned}
%\end{equation*}
%%
now a subtraction leads to cancelling the mixed term $\bfewt \bfM\t \bfeud$, and again by the product rule and the symmetry of $\bfA$, we obtain
\begin{align*}
	\frac{1}{2} \diff \|\bfeu\|_{\bfA\t}^2 + \|\bfew\|_{\bfA\t}^2 
	= &\ - \bfewt \bfMd\t \bfeu +  \frac{1}{2} \, \bfeut \bfAd\t \bfeu  \\
	&\ + \bfewt \big(\bff(\bfu\t) - \bff(\bfu^*\t) \big) \\
	&\ - \,  \bfeudt \big(\bfg(\bfu\t) - \bfg(\bfu^*\t)\big)  \\
	&\ -  \bfewt \bfM\t \bfdu  +  \bfeudt \bfM\t \bfdw  - \bfeudt \bfvartheta.
\end{align*}
Taking the linear combination of the above equalities yields
\begin{equation}
\label{eq:energy est - pre bounds}
	\begin{aligned}
		\frac{1}{2} \diff \|\bfeu\|_{\bfK\t}^2 + \|\bfew\|_{\bfK\t}^2 
		= &\ - \bfewt \bfMd\t \bfeu - \frac{1}{2} \bfeut \bfMd\t \bfeu + \frac{1}{2} \bfeut \bfAd\t \bfeu \\
		&\ + \bfeut \big(\bff(\bfu\t) - \bff(\bfu^*\t) \big) \\
		&\ + \bfewt \big(\bff(\bfu\t) - \bff(\bfu^*\t) \big) \\
		&\ + \bfewt \big(\bfg(\bfu\t) - \bfg(\bfu^*\t) \big) \\
		&\ -\bfeudt \big(\bfg(\bfu\t) - \bfg(\bfu^*\t)\big) \\
		&\ - \bfeut \bfM\t \bfdu  - \bfewt \bfM\t \bfdw \\
		&\ - \bfewt \bfM\t \bfdu  +  \bfeudt \bfM\t \bfdw \\
		&\ + \bfewt \bfvartheta - \bfeudt \bfvartheta .
	\end{aligned}
\end{equation}
The terms on the right-hand side are now estimated separately.

The terms involving time derivatives of matrices are estimated using Lemma~\ref{lemma:MatrixDerivativeEstimate}, by
\begin{equation}
\label{eq:energy est - i - diff terms}
\begin{aligned}
	&\ - \bfewt \bfMd\t \bfeu - \frac{1}{2} \bfeut \bfMd\t \bfeu + \frac{1}{2} \bfeut \bfAd\t \bfeu 
	\leq  \|\ew\|_{\bfM\t} \|\eu\|_{\bfM\t} + c \|\eu\|_{\bfK\t}^2 .
\end{aligned}
\end{equation}
 For the non-linear terms, using \eqref{eq:UhBound2} and the local-Lipschitz property of $f$, we obtain 
\begin{equation}
\label{eq:energy est - i - non-linear term 1}
	\begin{aligned}
		 &\bfeut \big(\bff(\bfu\t) - \bff(\bfu^*\t) \big) \\
		= &\ \int_{\Ga_h\t} \!\! e_{u_h}(\cdot,t) \big( f(u_h(\cdot,t), \nb_{\Ga_h \t} u_h(\cdot,t)) - f(u_h^*(\cdot,t), \nb_{\Ga_h \t} u_h^*(\cdot,t))\big)  \\
		\leq &\ L \, \|e_{u_h}(\cdot,t)\|_{L^2(\Ga_h\t)} \|u_h(\cdot,t) - u_h^*(\cdot,t)\|_{H^1(\Ga_h\t)}\\
		= &\ c \, \|\bfeu\|_{\bfM\t} \|\bfeu\|_{\bfK\t} , 
	\end{aligned}
\end{equation} 
% \begin{equation}
% \label{eq:energy est - i - non-linear term 1}
% 	\begin{aligned}
% 		 \bfewt \big(\bfW(\bfu\t) - \bfW(\bfu^*\t) \big) 
% 		= &\ \int_{\Ga_h\t} \!\! e_{w_h}(\cdot,t) \big(W'(u_h(\cdot,t)) - W'(u_h^*(\cdot,t))\big)  \\
% 		\leq &\ L \, \|e_{w_h}(\cdot,t)\|_{L^2(\Ga_h\t)} \|u_h(\cdot,t) - u_h^*(\cdot,t)\|_{L^2(\Ga_h\t)}\\
% 		= &\ c \, \|\bfew\|_{\bfM\t} \|\bfeu\|_{\bfM\t} , 
% 	\end{aligned}
% \end{equation}
where $L$ is the local Lipschitz constant of $f$, and we similarly obtain  
\begin{subequations}
\begin{align}
\label{eq:energy est - i - non-linear term 2}
\bfewt \big(\bff(\bfu\t) - \bff(\bfu^*\t)\big) \leq c \,\|\bfew\|_{\bfM\t} \|\bfeu\|_{\bfK\t},\\
\label{eq:energy est - i - non-linear term 3}
\bfewt \big(\bfg(\bfu\t) - \bfg(\bfu^*\t) \big)  \leq c \, \|\bfew\|_{\bfM\t} \|\bfeu\|_{\bfK\t},\\
\label{eq:energy est - i - non-linear term 4}
\bfeudt \big(\bfg(\bfu\t) - \bfg(\bfu^*\t)\big) \leq c \,\|\bfeud\|_{\bfM\t} \|\bfeu\|_{\bfK\t}.
\end{align}
\end{subequations}

The defect terms are estimated by the Cauchy--Schwarz inequality, as
\begin{equation}
\label{eq:energy est - i - defect terms}
\begin{aligned}
	&\ -\bfeut \bfM\t \bfdu  - \bfewt \bfM\t \bfdw 
	-  \bfewt \bfM\t \bfdu  +  \bfeudt \bfM\t \bfdw \\
	\leq &\ \|\eu\|_{\bfM\t} \|\du\|_{\bfM\t} + \|\ew\|_{\bfM\t} \|\dw\|_{\bfM\t} \\
	&\ + \|\ew\|_{\bfM\t} \|\du\|_{\bfM\t} + \|\bfeud\|_{\bfM\t} \|\dw\|_{\bfM\t} .
\end{aligned}
\end{equation}
The terms involving the correction term $\bfvartheta$ are bounded similarly as the defect terms. Using equality \eqref{eq:vartheta equals defect d_w} and the norm equivalence in time \cite[Lemma~4.1]{DziukLubichMansour_rksurf} (to change the time from $0$ to $t$), we obtain 
\begin{equation}
\label{eq:energy est - i - vartheta terms}
\begin{aligned}
	\bfewt \bfvartheta -  \bfeudt \bfvartheta 
	\leq &\ \|\ew\|_{\bfM(0)} \|\dw(0)\|_{\bfM(0)} + \|\bfeud\|_{\bfM(0)} \|\dw(0)\|_{\bfM(0)} \\
	\leq &\ c\inv \|\ew\|_{\bfM\t} \|\dw(0)\|_{\bfM(0)} + c\inv \|\bfeud\|_{\bfM\t} \|\dw(0)\|_{\bfM(0)} .
\end{aligned}
\end{equation}
Altogether, by the combination of the estimates \eqref{eq:energy est - i - diff terms}--\eqref{eq:energy est - i - defect terms} with \eqref{eq:energy est - pre bounds}, by multiple Young's inequalities (with $\varrho_0 > 0$ chosen later on) and by absorptions to the left-hand side, we obtain
\begin{equation}
\label{eq:energy estimate - i - pre int}
	\begin{aligned}
		\frac{1}{2} \diff \|\bfeu\|_{\bfK\t}^2 + \|\bfew\|_{\bfK\t}^2 
		\leq &\ \varrho_0 \frac{1}{2} \|\bfeud\|_{\bfK\t}^2 + c \|\eu\|_{\bfK\t}^2 \\
		&\ + c \|\bfdu\|_{\bfM\t}^2 + c \|\bfdw\|_{\bfM\t}^2 + c \|\dw(0)\|_{\bfM(0)}^2 .
	\end{aligned}
\end{equation}
Integrating from $0$ to $t \in (0,t^*]$, and using that $\bfeu(0) = 0$, we obtain the first energy estimate:
\begin{equation}
\label{eq:energy estimate - i}
	\begin{aligned}
		\|\bfeu\t\|_{\bfK\t}^2 +  \int_0^t \!\! {\|\ew\s\|_{\bfK\s}^2} \d s 
		\leq &\ \varrho_0 \int_0^t \!\! {\|\bfeud\s\|_{\bfK\s}^2} \d s 
%		\\ &\ 
		+ c \int_0^t \!\! {\|\bfeu\s\|_{\bfK\s}^2} \d s \\
		&\ + c \int_0^t \!\! \big( \|\bfdu\s\|_{\bfM(s)}^2 + \|\bfdw\s\|_{\bfM(s)}^2 \big) \d s \\
		&\ + c t \|\dw(0)\|_{\bfM(0)}^2 .
	\end{aligned}
\end{equation}
Note that if we do not use the Ritz map for the initial value for $u_h$, the error $\|\eu(0)\|_{\bfK(0)}^2$ would not vanish on the right-hand side. This $H^1$ norm error however cannot be bounded with the sufficient order. 
Furthermore, note the critical term, with $\|\bfeud\s\|_{\bfK\s}$, on the right-hand side, which cannot be bounded or absorbed in any direct way.

\textit{Energy estimates (ii)} To control the critical term on the right-hand side of \eqref{eq:energy estimate - i} we will now derive an energy estimate, which includes this term on the left-hand side. To this end, we first differentiate the second equation of \eqref{eq:error equations} with respect to time (note that the time-independent $\bfvartheta$ vanishes), and, after rearranging the terms, we obtain the following system:
\begin{subequations}
\label{eq:diff error equations}
	\begin{align}
		\label{eq:diff error equations - u}
		\bfM\t\bfeud + \bfA\t\bfew = &\ \bff(\bfu\t) - \bff(\bfu^*\t) \\
		&\ -\bfMd\t\bfeu - \bfM\t\bfdu, \nonumber\\
		\label{eq:diff error equations - w}
		\bfM\t\bfewd - \bfA\t\bfeud = &\ - \bfMd\t\bfew + \bfAd\t\bfeu \nonumber \\
		&\ + \diff\big(\bfg(\bfu\t) - \bfg(\bfu^*\t)\big) \\
		&\ - \bfMd\t\bfdw - \bfM\t\bfdwd. \nonumber
	\end{align}
\end{subequations}
Testing the error equation system \eqref{eq:diff error equations} twice, similarly as before in Part (i), would not lead to a feasible energy estimate, but to a bound which includes a new critical term $\bfeud$. The issue is avoided by separating the two estimates for the error equations, (ii.a) and (ii.b), and then taking their \emph{weighted} combination in (ii.c),
(ii.a). We test \eqref{eq:diff error equations - u} by $\bfeud$ and \eqref{eq:diff error equations - w} by $\bfew$,
% and reorder to obtain
%%
%\begin{equation}
%\begin{aligned}
%\label{eq: differror}
%\bfeudt  \bfM\t \bfeud  +  \bfeudt \bfA\t \bfew  = &-  \bfeudt \bfMd\t \bfeu - \bfeudt \bfM\t \bfdu ,\\
%\eps\inv \bfewt \bfM\t \bfewd  - \bfewt \bfA\t \bfeud  =&\ \bfewt \bfAd\t \bfeu - \eps\inv \bfewt \bfMd\t \bfew   \\
%&+\eps^{-2}  \bfewt  \diff\big(\bfW(\bfu\t) - \bfW(\bfu^*\t)\big) \\
%& - \eps\inv \bfewt \bfMd\t \bfdw  \\
%& - \eps\inv \bfewt \bfM\t \bfdwd\t.
%\end{aligned}
%\end{equation}
%%
adding the two equations together to cancel the mixed term $\bfeudt \bfA\t \bfew$, and using the product rule as before, we obtain
\begin{equation}
\label{eq:energy est - ii.a - pre est}
    \begin{aligned}
        \|\bfeud\|_{\bfM\t}^2 + \frac{1}{2} \diff \|\bfew\|_{\bfM\t}^2 
        = &\ - \bfeudt \bfMd\t \bfeu - \frac{1}{2} \bfewt \bfMd\t \bfew + \bfewt \bfAd\t \bfeu \\
        &\ + \bfeudt  \big(\bff(\bfu\t) - \bff(\bfu^*\t)\big) \\
        &\ +\bfewt  \diff\big(\bfg(\bfu\t) - \bfg(\bfu^*\t)\big) \\
        &\ - \bfeudt \bfM\t \bfdu - \bfewt \bfMd\t \bfdw - \bfewt \bfM\t \bfdwd\t.
    \end{aligned}
\end{equation}
The right-hand side terms are again estimated separately. The ones in the first line are bounded, using Lemma~\ref{lemma:MatrixDerivativeEstimate}, by
\begin{equation}
\label{eq:energy est - ii.a - diff terms}
	\begin{aligned}
		&\ - \bfeudt \bfMd\t \bfeu - \frac{1}{2} \bfewt \bfMd\t \bfew + \bfewt \bfAd\t \bfeu \\
		\leq &\ c \|\bfeud\|_{\bfM\t} \|\eu\|_{\bfM\t} + c \, \|\ew\|_{\bfM\t}^2 + c \|\ew\|_{\bfA\t} \|\eu\|_{\bfA\t} .
	\end{aligned}
\end{equation}
The first non-linear term is estimated as in \eqref{eq:energy est - i - non-linear term 1} -- \eqref{eq:energy est - i - non-linear term 4} whereas the second non-linear term occurs differentiated with respect to time. Therefore, with the help of the transport formula \eqref{eq:transport formula - m} we compute, omitting the omnipresent argument $t$,  
\begin{equation*}
\begin{aligned}
 &\ \ew^T \diff \big(\bfg(\bfu) -\bfg(\bfu^*)\big)\\
	= &\   \int_{\Ga_h\t} \!\!\!\! {\partial_1 g\big(u_h, \nb_{\Ga_h \t} u_h\big) \, \mat_h u_h  \, e_{w_h}} - \int_{\Ga_h\t} \!\!\!\! {\partial_1 g\big(u_h^*, \nb_{\Ga_h \t} u_h^*\big) \, \mat_h u_h^* \, e_{w_h}} \\
    &\ + \int_{\Ga_h\t} \!\!\!\!\!\! {\partial_2 g\big(u_h, \nb_{\Ga_h \t} u_h\big) \, \mat_h (\nb_{\Ga_h \t} u_h  ) \, e_{w_h}} 
    \\
    &\ - \int_{\Ga_h\t} \!\!\!\!\!\! {\partial_2 g\big(u_h^*, \nb_{\Ga_h \t} u_h^*\big) \, \mat_h (\nb_{\Ga_h \t} u_h^*) \, e_{w_h}} \\
	&\ + \int_{\Ga_h\t} \!\!\!\!\!\!\! {(\nbgh \cdot V_h) \, g\big(u_h, \nb_{\Ga_h \t} u_h\big) \, e_{w_h}} 
	- \int_{\Ga_h\t} \!\!\!\!\!\!\! {(\nbgh \cdot V_h) \, g\big(u_h^*, \nb_{\Ga_h \t} u_h^*\big) \, e_{w_h}} \\
	=: & \ I + II + III.
	\end{aligned}
\end{equation*}
Let us first estimate the most challenging second term. Inserting \linebreak $\mp \int_{\Ga_h\t}{\partial_2 g\big(u_h, \nb_{\Ga_h \t} u_h\big) \,\mat_h\, \nb_{\Ga_h \t} u_h^* \, e_{w_h}}$ we bound $II$ by
\begin{equation*}
\begin{aligned}
II = & \ \int_{\Ga_h\t}{\partial_2 g\big(u_h, \nb_{\Ga_h \t} u_h\big)} 
\, \Big( \mat_h (\nb_{\Ga_h \t} u_h)   - \mat_h (\nb_{\Ga_h \t} u_h^*) \Big) \, e_{w_h} \\
&\ - \int_{\Ga_h\t}{\Big( \partial_2 g\big(u_h^*, \nb_{\Ga_h \t} u_h^*\big) - \partial_2 g\big(u_h, \nb_{\Ga_h \t} u_h\big)\Big)} \, \mat_h (\nb_{\Ga_h \t} u_h^*) \, e_{w_h} \\
\leq & \ c \, \|\bfew\|_{\bfM\t} \big( \|\bfeu\|_{\bfK\t} + \|\bfeud\|_{\bfK\t} \big),
\end{aligned}
\end{equation*}
using the first interchange formula from \eqref{eq:interchange formulas}, the local Lipschitz property of $\partial_2 g$ together with \eqref{eq:fBound}, and the bounds on $V_h$ obtained by interpolation error estimates (for details, see \cite[Lemma~3.1.6]{Beschle_thesis}).

The second term is now estimated analogously, by adding and subtracting, but not requiring the interchange steps, these yield
%Inserting $\mp \int_{\Ga_h\t}{\partial_1 g\big(u_h, \nb_{\Ga_h \t} u_h\big) \,\mat_h\, u_h^* \, e_{w_h}}$ we bound $I$ by
\begin{equation*}
 \begin{aligned}
   I 
%   = & \ \int_{\Ga_h\t}{\partial_1 g\big(u_h, \nb_{\Ga_h \t} u_h)\big) \, \big[\mat_h u_h  - \mat_h u_h^*\big] \, e_{w_h}} \\
%	&\ - \int_{\Ga_h\t}{\Big( \partial_1 g\big(u_h^*, \nb_{\Ga_h \t} u_h^*\big)  - \partial_1 f\big(u_h, \nb_{\Ga_h \t} u_h\big) \Big) } \, \mat_h u_h^* \, e_{w_h} \\
	 \leq  & \ c \, \|\bfew\|_{\bfM\t} \big( \|\bfeu\|_{\bfM\t} + \|\bfeud\|_{\bfM\t} \big),
 \end{aligned}
\end{equation*}
using the local Lipschitz property of $\partial_1 g$ together with \eqref{eq:fBound}.  
Furthermore, for the third term we directly obtain
\begin{equation*}
 \begin{aligned}
    III = & \ \int_{\Ga_h\t} \!\!\!\!\!\!\! {(\nbgh \cdot V_h) \, g\big(u_h, \nb_{\Ga_h \t} u_h\big) \, e_{w_h}} 
    - \int_{\Ga_h\t} \!\!\!\!\!\!\! {(\nbgh \cdot V_h) \, g\big(u_h^*, \nb_{\Ga_h \t} u_h^*\big) \, e_{w_h}} \\
	\leq & \ c \, \|\bfew\|_{\bfM\t} \|\bfeu\|_{\bfK\t},
 \end{aligned}
\end{equation*}
using the local Lipschitz property of $g$ together with \eqref{eq:fBound}.
Altogether, the estimates for $I$--$III$ yield
\begin{equation}
\label{eq:energy est - ii.a - non-linear term}
	\ew^T \diff \big(\bfg(\bfu) -\bfg(\bfu^*)\big) \leq \ c \, \|\bfew\|_{\bfM\t} \big( \|\bfeu\|_{\bfK\t} + \|\bfeud\|_{\bfK\t} \big).
\end{equation}
 
The defect terms are bounded, similarly as before, by
\begin{equation}
\label{eq:energy est - ii.a - defect terms}
	\begin{aligned}
		&\ - \bfeudt \bfM\t \bfdu - \bfewt \bfMd\t \bfdw - \bfewt \bfM\t \bfdwd\t \\
		\leq &\ c \|\bfeud\|_{\bfM\t} \|\du\|_{\bfM\t} + c \,\|\ew\|_{\bfM\t} \|\dw\|_{\bfM\t} + c \, \|\ew\|_{\bfM\t} \|\bfdwd\|_{\bfM\t} .
	\end{aligned}
\end{equation}
Altogether, by plugging in \eqref{eq:energy est - ii.a - diff terms}--\eqref{eq:energy est - ii.a - defect terms} into \eqref{eq:energy est - ii.a - pre est}, then using Young's inequalities (with a small number $\varrho_1>0$), we obtain the first energy estimate of this part:
% \begin{equation}
% \label{eq:energy est - ii.a}
% \begin{aligned}
% 	\frac{1}{2} \|\bfeud\|_{\bfM\t}^2 + \frac{1}{2} \diff \|\bfew\|_{\bfM\t}^2 
% 	\leq &\ c \|\bfeu\|_{\bfK\t}^2	+ c \, \|\bfew\|_{\bfK\t}^2 
% 	\\
% %	+ \varrho \|\bfeud\|_{\bfA\t}^2 \\
% 	&\ + c \big( \|\bfdu\|_{\bfM\t}^2 + \|\bfdw\|_{\bfM\t}^2 + \|\bfdwd\|_{\bfM\t}^2 \big) .
% 	\end{aligned}
% \end{equation}
\begin{equation}
\label{eq:energy est - ii.a}
\begin{aligned}
	\|\bfeud\|_{\bfM\t}^2 + \frac{1}{2} \diff \|\bfew\|_{\bfM\t}^2 
	\leq &\ c \|\bfeu\|_{\bfK\t}^2	+ c \varrho_1 \|\bfeud\|_{\bfK\t}^2 + c \, \|\bfew\|_{\bfK\t}^2 
	\\
%	+ \varrho \|\bfeud\|_{\bfA\t}^2 \\
	&\ + c \big( \|\bfdu\|_{\bfM\t}^2 + \|\bfdw\|_{\bfM\t}^2 + \|\bfdwd\|_{\bfM\t}^2 \big) .
	\end{aligned}
\end{equation}
%[kb: At first I could not find where this term comes from. It is due to the old version (using dual norm estimates) of the defect bounds \eqref{eq:energy est - ii.a - defect terms}! This needs to be settled! Maybe it's useful for non-linear mobility terms.] [cb: I still cannot figure out where it comes from.] 
(ii.b) We now test \eqref{eq:diff error equations - u} by $\bfewdt$ and \eqref{eq:diff error equations - w} by $\bfeudt$, then
% We reorder to arrive at
%%
%\begin{align*}
%\bfewdt  \bfM\t \bfeud + \bfewdt \bfA\t \bfew = &- \bfewdt \bfMd\t \bfeu + \bfewdt \bfM\t \bfdu, \\
%\bfeudt \bfM\t \bfewd\t - \eps \,\bfeudt \bfA\t \bfeud  = &\ \eps\bfeudt \bfAd\t \bfeu  -  \bfeudt \bfMd\t \bfew \\
%& +\eps\inv\bfeudt  \diff\big(\bfW(\bfu\t) - \bfW(\bfu^*\t)\big)\\
%& + \bfeudt \bfMd\t \bfdw + \bfeudt \bfM\t \bfdwd.
%\end{align*}
%%
subtracting the second from the first equation to cancel the mixed term $\bfewdt \bfM\t \bfeud$ and using the product rule again we obtain
\begin{equation}
\label{eq:energy est - ii.b - pre est}
\begin{aligned}
	\|\bfeud\|_{\bfA\t}^2 + \frac{1}{2}\diff \|\bfew\|_{\bfA\t}^2 = &\ - \bfewdt \bfMd\t \bfeu - \bfewdt \bfM\t \bfdu  \\
	&\ - \bfeudt \bfAd\t \bfeu +  \bfeudt \bfMd\t \bfew + \frac{1}{2} \bfewt \bfAd\t \bfew \\
	&\ +  \bfewdt \big(\bff(\bfu\t) - \bff(\bfu^*\t)\big) \\
	&\ -  \bfeudt \diff\big(\bfg(\bfu\t) - \bfg(\bfu^*\t)\big) \\
	&\ + \bfeudt \bfMd\t \bfdw  + \bfeudt \bfM\t \bfdwd.
\end{aligned}
\end{equation}
The terms are again estimated separately. The terms with time derivatives of matrices on the right-hand sides of \eqref{eq:energy est - ii.b - pre est} are bounded, using Lemma~\ref{lemma:MatrixDerivativeEstimate}, by
\begin{equation}
\label{eq:energy est - ii.b - diff terms}
	\begin{aligned}
		&\ - \bfeudt \bfAd\t \bfeu + \bfeudt \bfMd\t \bfew + \frac{1}{2} \bfewt \bfAd\t \bfew \\
		\leq &\ c \|\bfeud\|_{\bfA\t} \|\eu\|_{\bfA\t} + c \, \|\bfeud\|_{\bfM\t} \|\ew\|_{\bfM\t} + c \,  \|\ew\|_{\bfA\t}^2 .
	\end{aligned}
\end{equation}
The differentiated non-linear term is bounded, similarly to \eqref{eq:energy est - ii.a - non-linear term}, by
\begin{equation}
\label{eq:energy est - ii.b - non-linear term}
\begin{aligned}
	 \bfeudt \diff\big(\bfg(\bfu\t) - \bfg(\bfu^*\t)\big) \leq & \ c \,\|\bfeud\|_{\bfM\t} \big( \|\eu\|_{\bfK\t} + \|\bfeud\|_{\bfK\t} \big) \\
	 \leq & \ c \varrho_2 \|\bfeud\|_{\bfM\t}^2 + c \, \|\bfeu\|_{\bfK\t}^2 \\
	 & \ + \frac{c_0}{4 \varrho_3} \|\bfeud\|_{\bfM\t}^2 + c_0 \varrho_3 \|\bfeud\|_{\bfK\t}^2,
\end{aligned}
\end{equation}
with a particular constant $c_0 > 0$ (independent of $h$, but depending on $F''$, viz.~on the constant in \eqref{eq:energy est - ii.a - non-linear term}).
%separating $\|\bfeud\|_{\bfM\t}$ with a particular constant $c_0 > 0$ (independent of $h$, but depending on $F''$, viz.~on the constant in \eqref{eq:energy est - ii.a - non-linear term}) from $\|\bfeud\|_{\bfK\t}$ with a small factor $\varrho_3 > 0$.
The defect terms are bounded, similarly as before, by
\begin{equation}
\label{eq:energy est - ii.b - defect terms}
\begin{aligned}
&\ - \bfeudt \bfM\t \bfdu - \bfewt \bfMd\t \bfdw -\bfewt \bfM\t \bfdwd\t \\
\leq &\ c \|\bfeud\|_{\bfM\t} \|\du\|_{\bfM\t} + c \, \|\ew\|_{\bfM\t} \|\dw\|_{\bfM\t} + c \, \|\ew\|_{\bfM\t} \|\bfdwd\|_{\bfM\t} .
\end{aligned}
\end{equation}
Let us highlight that it is not possible to directly estimate the terms containing $\bfewd\t$ in their current form, because there is no term on the left-hand side to absorb them. Therefore, we first rewrite them using the product rule, and estimate them using Lemma~\ref{lemma:MatrixDerivativeEstimate}, to obtain
\begin{equation}
\label{eq:energy est - ii.b - dot e_w terms}
\begin{aligned}
	 \bfewdt \big(\bff(\bfu\t) - \bff(\bfu^*\t)\big)  = &\ \diff \big( \bfewt \big(\bff(\bfu\t) - \bff(\bfu^*\t)\big) \big) \\
	& \ - \bfewt \diff\big(\bff(\bfu\t) - \bff(\bfu^*\t)\big)\\
	\leq & \ \diff \big( \bfewt \big(\bff(\bfu\t) - \bff(\bfu^*\t)\big) \big) \\
	& \ + c \|\ew\|_{\bfM\t}\big(\|\eu\|_{\bfK\t} + \|\bfeud\|_{\bfK\t} \big) , \\
	\bfewdt \bfMd\t \bfeu  = &\ \diff \big( \bfewt \bfMd\t \bfeu \big) - \bfewt  \bfMdd\t \bfeu  - \bfewt \bfMd\t \bfeud \\
	\leq &\  \diff \big( \bfewt \bfMd\t \bfeu \big) + c \|\ew\|_{\bfM\t} \big( \|\eu\|_{\bfM\t} + \|\bfeud\|_{\bfM\t} \big) , \\
	\bfewdt \bfM\t\bfdu = &\ \diff \big( \bfewt \bfM\t \bfdu \big) - \bfewt \bfMd\t \bfdu - \bfewt \bfM\t \bfdud \\
	\leq &\ \diff \big( \bfewt \bfM\t \bfdu \big) + c \|\ew\|_{\bfM\t} \big( \|\du\|_{\bfM\t} + \|\bfdud\|_{\bfM\t} \big) .
\end{aligned}
\end{equation}
Altogether, by plugging in \eqref{eq:energy est - ii.b - dot e_w terms}--\eqref{eq:energy est - ii.b - defect terms} into \eqref{eq:energy est - ii.b - pre est}, then using Young's inequalities (with a small number $\varrho_2>0$), we obtain the second energy estimate of this part:
% \begin{equation}
% \label{eq:energy est - ii.b}
% 	\begin{aligned}
% 		\frac{1}{2}\|\bfeud\|_{\bfA\t}^2 + \frac{1}{2}\diff \|\bfew\|_{\bfA\t}^2 
% 		\leq &\ \redon \varrho_1 \|\bfeud\|_{\bfK\t}^2 \redoff + c_0 \|\bfeud\|_{\bfM\t}^2 + c \|\bfeu\|_{\bfK\t}^2 	+ c \, \|\bfew\|_{\bfK\t}^2 \\
% 		&\ + c  \big( \|\du\|_{\bfM\t}^2 + \|\bfdud\|_{\bfM\t}^2 + \|\dw\|_{\bfM\t}^2 + \|\bfdwd\|_{\bfM\t}^2 \big)  \\
% 		&\ + \diff \big( \bfewt \big(\bff(\bfu\t) - \bff(\bfu^*\t)\big) \big) \\
% 		&\ - \diff \bfewt \bfMd\t \bfeu - \diff \bfewt \bfM\t \bfdu,
% 	\end{aligned}
% \end{equation}
\begin{equation}
\label{eq:energy est - ii.b}
	\begin{aligned}
		\|\bfeud\|_{\bfA\t}^2 + \frac{1}{2}\diff \|\bfew\|_{\bfA\t}^2 
		\leq &\ \frac{c_0}{4 \varrho_3} \|\bfeud\|_{\bfM\t}^2 + (c \varrho_2 + c_0 \varrho_3) \|\bfeud\|_{\bfK\t}^2 \\
		&\ + c \|\bfeu\|_{\bfK\t}^2 	+ c \, \|\bfew\|_{\bfK\t}^2 \\
		&\ + c  \big( \|\du\|_{\bfM\t}^2 + \|\bfdud\|_{\bfM\t}^2 + \|\dw\|_{\bfM\t}^2 + \|\bfdwd\|_{\bfM\t}^2 \big)  \\
		&\ + \diff \big( \bfewt \big(\bff(\bfu\t) - \bff(\bfu^*\t)\big) \big) \\
		&\ - \diff \bfewt \bfMd\t \bfeu - \diff \bfewt \bfM\t \bfdu.
	\end{aligned}
\end{equation}
%\begin{equation}
%\label{eq: energy4}
%\begin{aligned}
%	&\eps\|\bfeud\|_{\bfA\t}^2 + \frac{1}{2}\diff \|\bfew\|_{\bfA\t}^2 \\
%	\leq &c \|\bfeu\|_{\bfK\t}^2 	+ c \|\bfew\|_{\bfK\t}^2 + C_{W''}\,\|\bfeud\|_{\bfM\t}^2 + 7 \,\varrho\,\|\bfeud\|_{\bfK\t}^2\\
%	&+ c \|\bfdu\|_{*,t}^2 + c \|\bfdud\|_{*,t}^2 + c\, \|\bfdw\|_{\bfM\t}^2 \\
%	&+c \|\bfdwd\|_{*,t}^2 - \diff \bfewt \bfMd\t \bfeu + \diff \bfewt \bfM\t \bfdu,
%	\end{aligned}
%\end{equation}
(ii.c) 
We now take the weighted combination of the energy estimates from (ii.a) and (ii.b): multiplying the estimate \eqref{eq:energy est - ii.a} by $\frac{3 c_0}{4 \varrho_3}$ and adding it to the estimate \eqref{eq:energy est - ii.b}. Collecting the terms and directly absorbing the term $c_0 \|\bfeud\|_{\bfM\t}^2$ on the right-hand side of \eqref{eq:energy est - ii.b} to the left-hand side, (and choosing $\varrho_1,\varrho_2,\varrho_3>0$ small enough for absorption of the $\|\bfeud\|_{\bfK\t}^2$ terms from the left-hand side to the right-hand side), we obtain 
\begin{equation}
	\begin{aligned}
	\label{eq:energy estimate - ii - pre estimates}
		&\  \min\bigg\{\frac{c_0}{2 \varrho_3},\frac{1}{2} \bigg\} \|\bfeud\|_{\bfK\t}^2  + \, \frac{3c_0}{2 \varrho_3} \diff \|\bfew\|_{\bfM\t}^2 + \frac{1}{2}\diff \|\bfew\|_{\bfA\t}^2 \\
		\leq &\ c \|\bfeu\|_{\bfK\t}^2 + c \,  \|\bfew\|_{\bfK\t}^2 \\
		&\ + c  \, \big( \|\du\|_{\bfM\t}^2 + \|\bfdud\|_{\bfM\t}^2 + \|\dw\|_{\bfM\t}^2 + \|\bfdwd\|_{\bfM\t}^2 \big)  \\
		&\ + \diff \big( \bfewt \big(\bff(\bfu\t) - \bff(\bfu^*\t)\big) \big) - \diff \bfewt \bfMd\t \bfeu - \diff \bfewt \bfM\t \bfdu .
	\end{aligned}
\end{equation}
Integrating the above inequality \eqref{eq:energy estimate - ii - pre estimates} from $0$ to $t \leq t^*$, and then dividing by $\min\big\{\frac{c_0}{2 \varrho_3},\frac{1}{2} \big\}$, yields
\begin{equation*}
	\begin{aligned}
		&\ \|\bfew\t\|_{\bfK\t}^2 + \int_0^t{\|\bfeud\s\|_{\bfK\s}^2} \d s \\
		\leq &\ c \int_0^t{\|\bfeu\s\|_{\bfK\s}^2} \d s + c \, \int_0^t{\|\bfew\s\|_{\bfK\s}^2} \d s \\
		&\ + c \int_0^t \big( \|\bfdu\s\|_{\bfM(s)}^2 + \|\bfdud\s\|_{\bfM(s)}^2 + \|\bfdw\s\|_{\bfM(s)}^2 + \|\bfdwd\s\|_{\bfM(s)}^2 \big) \d s \\
		&\  + \bfewt \big(\bff(\bfu\t) - \bff(\bfu^*\t)\big) - \bfew (0) \big(\bff(\bfu(0)) - \bff(\bfu^*(0))\big) \\
		&\ - c \, \bfewt\t \bfMd\t \bfeu\t + c\, \bfewt(0)\bfMd(0) \bfeu(0)\\
		&\ - c \,  \bfewt\t \bfM\t \bfdu\t + c\, \bfewt(0) \bfM(0) \bfdu(0)\\
		&\ + \|\bfew(0)\|_{\bfK(0)}^2.
	\end{aligned}
\end{equation*}
We estimate the newly obtained non-integrated terms on the right-hand side using Lemma~\ref{lemma:MatrixDerivativeEstimate}, Cauchy--Schwarz and Young's inequalities, the estimate for the non-linear term \eqref{eq:energy est - i - non-linear term 2}, a further absorption, and using that $\eu(0)$ and $\ew(0)$ are zero, we then obtain
\begin{equation}
\label{eq:energy estimate - ii}
	\begin{aligned}
		&\ \|\bfew\t\|_{\bfK\t}^2 + \int_0^t{\|\bfeud\s\|_{\bfK\s}^2} \d s\\
		\leq &\ c \int_0^t{\|\bfeu\s\|_{\bfK\s}^2} \d s + c \,  \int_0^t{\|\bfew\s\|_{\bfK\s}^2} \d s \\
		&\ + c \int_0^t \big( \|\bfdu\s\|_{\bfM(s)}^2 + \|\bfdud\s\|_{\bfM(s)}^2 + \|\bfdw\s\|_{\bfM(s)}^2 + \|\bfdwd\s\|_{\bfM(s)}^2 \big) \d s \\
		&\  + c_1 \|\bfeu\t\|_{\bfK\t}^2 %+ \varrho \, c \|\bfew\|_{\bfM\t}^2 
		+ c \|\bfdu\t\|_{\bfM\t}^2 ,
	\end{aligned}
\end{equation}
with a $c_1 > 0$.
This energy estimate now contains the (previously) critical term $\bfeud$ on the left-hand side.
Without the construction in Section~\ref{section:modified problem} the initial values  for $\bfw$  would not vanish and a term $\|\bfew(0)\|_{\bfK(0)}^2$ would remain on the right-hand side. This $H^1$ norm error however cannot be bounded with the sufficient order. 

\textit{Combining the energy estimates:}
We now take again a $c_1$-weighted linear combination (in order to absorb the term  $c_1 \|\bfeu\|_{\bfK\t}^2$)  of the two energy estimates \eqref{eq:energy estimate - i} and \eqref{eq:energy estimate - ii}, to obtain
\begin{equation}
\label{eq:pre Gronwall}
	\begin{aligned}
		&\ \|\bfeu\t\|_{\bfK\t}^2 + \|\bfew\t\|_{\bfK\t}^2
		 + \int_0^t{\|\bfeud\s\|_{\bfK\s}^2} \d s +  \int_0^t{\|\ew\s\|_{\bfK\s}^2} \d s \\
		\leq &\ \varrho_0 \int_0^t{\|\bfeud\s\|_{\bfM\s}^2} \d s \\
		&\ + c \int_0^t{\|\bfeu\s\|_{\bfK\s}^2} \d s+ c \, \int_0^t{\|\bfew\s\|_{\bfK\s}^2} \d s\\
		&\ + c \int_0^t \big( \|\bfdu\s\|_{\bfM(s)}^2 + \|\bfdud\s\|_{\bfM(s)}^2 + \|\bfdw\s\|_{\bfM(s)}^2 + \|\bfdwd\s\|_{\bfM(s)}^2 \big) \d s \\
		&\ + c \,\|\bfdu\t\|_{\bfM\t}^2 .
	\end{aligned}
\end{equation}
By choosing $\varrho_0$ small enough, the first term (previously the critical term) on the left-hand side is now absorbed.
This enables us to use Gronwall's inequality, which then yields the stated stability estimate on $[0,t^*]$.

Now, it only remains to show that, in fact, $t^* = T$, for $h$ sufficiently small. The proved stability bound (for $0 \leq t \leq t^*$) together with the assumed defect bounds \eqref{eq:defect bounds - assumed} imply
\begin{equation*}
	 \|\bfeu\t\|_{\bfK\t}^2 + \|\bfew\t\|_{\bfK\t}^2 \leq c h^\kappa , \qquad \text{with} \quad \kappa \geq 2 .
\end{equation*}
By an inverse estimate, see, e.g.~\cite[Theorem~4.5.11]{BreS08}, we have, for $0 \leq t \leq t^*$,
\begin{equation}
\label{eq: InverseInequality}
	\begin{aligned}
		 \|e_{u_h}(\cdot,t)\|_{W^{1,\infty}(\Ga_h\t)}  \leq &\  c h^{-d/2} \|e_{u_h}(\cdot,t)\|_{H^1(\Ga_h\t)} \\
		\leq &\ c h^{-d/2} \|\bfeu \t\|_{\bfK\t} \leq c \, C h^{\kappa - d/2} \leq \frac{1}{2} \, h^{\frac{\kappa - d/2}{2}},
	\end{aligned}
\end{equation}
for sufficiently small $h$. Therefore, the bound \eqref{eq: tStarBound} is extended beyond $t^*$, which contradicts the maximality of $t^*$ unless we already have $t^* = T$. We hence proved the stability bound \eqref{eq:stability bound} over $[0,T]$, and completed the proof.
\qed\end{proof}
 
\begin{remark}
	The dimensional assumptions $\Ga\t \subset \R^{d+1}$ for $d = 1,2$ are not entirely restrictive. For a higher dimensional surface, the argument \eqref{eq: InverseInequality} can be repeated for a $\kappa$ sufficiently large, that is requiring a finite element basis of sufficiently high order, depending on the dimension $d$. 
\end{remark} 
 
\section{Consistency}

\label{section:consistency}
Before we turn to proving consistency of the spatial semi-discretisation and to the proof of Theorem~\ref{theorem:semi-discrete convergence}, we collect some preparatory results: error estimates of the nodal interpolations on the surface, for the Ritz map, and some results which estimate various geometric errors. Most of these results were shown in \cite{DziukElliott_L2,Demlow,highorderESFEM}.

Let us briefly recall our assumptions on the evolving surface and on its discrete counterpart, from Section~\ref{section:CH} and \ref{subsection:ESFEM}: $\Ga\t$ is a closed  smooth (at least $C^2$) surface in $\R^{d+1}$ with $ d\leq 3$, evolving with the surface velocity $v$, with regularity $v(\cdot,t),\mat v(\cdot,t) \in W^{k+1,\infty}(\Ga\t)$ uniformly in time. The discrete surface $\Ga_h\t$ is a $k$-order interpolation of $\Ga\t$ at each time, and therefore its velocity $V_h$ is the nodal interpolation of $v$ on $\Ga_h\t$, see \eqref{eq:discrete velocity} and Section~\ref{subsection:ESFEM}. 

\subsection{Geometric errors}

\subsubsection{Interpolation error estimates}

The following result gives estimates for the error in the interpolation. Our setting follows that of Section 2.5 of \cite{Demlow}.

Let us assume that the surface $\Ga\t$ is approximated by the interpolation surface $\Ga_h\t$ of order $k$. Then for any $u \in H^{k+1}(\Ga\t)$, there is a unique $k$-order surface finite element interpolation $\widetilde I_h u \in S_h\t$, furthermore we set $(\widetilde I_h u)^\ell = I_h u$.
\begin{lemma}
	\label{lemma:interpolation error estimate}
	For any $u\ct \in H^{k+1}(\Ga\t)$ for all $0 \leq t \leq T$. The surface interpolation operator $I_h$ of order $k$ satisfies the following error estimates, for $u = u\ct$ and for $0 \leq t \leq T$,
	\begin{align*}
	\|u - I_h u\|_{L^2(\Ga\t)} + h\|\nbg(u - I_h u)\|_{L^2(\Ga\t)} &\leq c  h^{k+1} \|u\|_{H^{k+1}(\Ga \t)} ,\\
	\|u - I_h u\|_{L^\infty(\Ga\t)} + h\|\nbg(u - I_h u)\|_{L^\infty(\Ga\t)} &\leq c  h^{k+1} \|u\|_{W^{k+1,\infty}(\Ga \t)} ,
	\end{align*}
	with a constant $c > 0$ independent of $h$ and $t$, but depending on $v$ and $\GT$.
\end{lemma}

\subsubsection{Discrete surface velocities}

This section gives a definition of a discrete velocity on the exact surface $\Ga\t$ associated to $V_h$, and explores approximation results for the discrete velocities.
The following result, recalled from \cite[Lemma~3.1.6]{Beschle_thesis}, shows boundedness of the discrete velocity $V_h$, using the fact that it is the interpolation of $v$. The proof is based on the interpolation error estimate Lemma~\ref{lemma:interpolation error estimate} and the interchange formulas \eqref{eq:interchange formulas}.
\begin{lemma}
	\label{lemma:VhBhBound}
	Assume that $v$ and $\mat v$ are in $W^{{k+1},\infty}(\Ga\t)$. Then, for $h \leq h_0$ sufficiently small, the following bounds hold:
	\begin{equation*}
	\begin{aligned}
	\|V_h\|_{W^{1,\infty}(\Ga_h\t)} 
%	 + \| \mathcal{B}_h (V_h)\|_{L^{\infty}(\Ga_h\t)} 
	\leq &\ c \|v\|_{W^{{k+1},\infty}(\Ga\t)}, \\
	\|\mat_h V_h\|_{W^{1,\infty}(\Ga_h\t)} \leq &\ c \|\mat v\|_{W^{{k+1},\infty}(\Ga\t)}, \\
	\|\mat_h (\nb_{\Ga_h \t} \cdot V_h)\|_{L^{\infty}(\Ga_h\t)} \leq &\ c \big( \|\mat v\|_{W^{{k+1},\infty}(\Ga\t)} + \|v\|_{W^{{k+1},\infty}(\Ga\t)}^2 \big),
	\end{aligned}
	%\label{eq: VhBhBound}
	\end{equation*}
	where the constant $c > 0$ is independent of $h$ and $t$, but depends on $\GT$.
\end{lemma}
To $V_h$ we associate a discrete surface (or material) velocity of $\Ga\t$, denoted by $v_h$. It is the surface velocity of the lifted material points $y\t = (x\t)^\ell \in \Ga\t$. The edges of a lifted element evolve with this velocity $v_h$, which is not the interpolation of $v$ in $S_h^\ell\t$. For more details we refer to \cite[Definition~4.3]{DziukElliott_L2} and \cite[Section~5.4]{DziukElliott_acta}. 

Here we recall an explicit formula for $v_h$: for $x\t \in \Ga_h\t$ with $y\t = x^\ell\t$,
\begin{equation}
\label{eq:definition v_h}
v_h(y\t,t) = \pa_t y(x\t,t) + V_h(x\t,t) \cdot \nb y(x\t,t) ,
\end{equation}
with $y\t = y(x\t,t) \in \Ga\t$ denoting the lift of $x\t \in \Ga_h\t$, cf.~Section~\ref{section:lift}, i.e.~the unique solution to $x\t = y(x\t,t) + d(x\t,t) \, \nu(y(x\t,t),t)$. For an explicit formula using $V_h$ and a distance function we refer to \cite[equation~$(4.7)$]{DziukElliott_L2}. 

Apart from the original material derivative $\mat$ on $\Ga\t$, a discrete material derivative associated to the velocity $v_h$ is also defined on $\Ga\t$, see~\cite[equation~$(4.9)$]{DziukElliott_L2}, for $\vphi(\cdot,t): \Ga \t \to \mathbb{R}$ (element-wise) by
\begin{equation}
\label{eq:discrete material derivative on Ga}
\mat_h \vphi(\cdot,t) = \pa_t \bar \vphi(\cdot,t) + v_h(\cdot,t) \cdot \nb \bar \vphi(\cdot,t)  \quadfor 0 \leq t \leq T,
\end{equation}
where $\bar \vphi(\cdot,t)$ is an extension into a small neighbourhood of $\Ga\t$. 
That is we have the following three different material derivatives:
\begin{alignat*}{3}
	\text{for } \vphi = \vphi(\cdot,t) : \Ga\t \to \R: &\ \qquad \mat \vphi = & &\ \pa_t \bar \vphi + v \cdot \nb \bar \vphi , \\
	&\ \qquad \mat_h \vphi = & &\ \pa_t \bar \vphi + v_h \cdot \nb \bar \vphi , \\
	\text{for } \vphi_h = \vphi_h(\cdot,t) : \Ga_h\t \to \R: &\ \qquad \mat_h \vphi_h = & &\ \pa_t \bar \vphi_h + V_h \cdot \nb \bar \vphi_h .
\end{alignat*}
We note here that it will be always clear from the context whether the discrete material derivative $\mat_h$ is meant on $\Ga\t$ associated to $v_h$, or on $\Ga_h\t$ associated to $V_h$. 
%Furthermore, we note that $\mat_h (\vphi_h^\ell) = (\mat_h \vphi_h)^\ell$. More details in $L^2$ estimates. 

From \cite[Lemma~5.4]{highorderESFEM} we recall high-order error bounds between the velocity $v_h$ of the lifted material points and the surface velocity $v$ (for the case $k=1$, and without material derivative, $l = 0$, we refer to \cite{DziukElliott_L2}).
\begin{lemma}
	\label{lemma:velocity error estimate}
	The difference between the continuous velocity $v$ and the discrete velocity $v_h$ on $\Ga\t$ is estimated by
	\begin{equation*}
	\label{eq:velocity error estimate}
	\|(\mat_h)^{(l)}(v - v_h)\|_{L^{\infty}(\Ga\t)} + h\|\nb_{\Ga \t} (\mat_h)^{(l)}(v - v_h)\|_{L^{\infty}(\Ga\t)} \leq c_lh^{k+1},
	\end{equation*}
	for $l \geq 0$, with a constant $c_l > 0$ independent of $h$ and $t$, but depending on the surface velocity $v$. 
\end{lemma}
Since we need to establish a bound for the discrete material derivatives of both defects $d_u$ and $d_w$, we recall some \emph{transport formulas} from \cite[Lemma~4.2]{DziukElliott_L2} (for any sufficiently regular functions): 
\begin{subequations}
	\label{eq:transport formula - T2 and T3}
	\begin{align}
	\label{eq:transport formula - T2}
	\diff m(u,\vphi) = &\ m(\mat_h u,\vphi) + m(u,\mat_h \vphi) + r(v_h;u,\vphi) , \\
	\label{eq:transport formula - T3}
	\diff m_h(u_h,\vphi_h) = &\ m_h(\mat_h u_h,\vphi_h) + m(u_h,\mat_h \vphi_h) + r_h(V_h;u_h,\vphi_h) .
	\end{align}
\end{subequations}
These formulas will help us to derive equations for $\mat_h d_u$ and $\mat_h d_w$ and are often used in the proofs in Section~\ref{section:geometric approx errors}.
The two transport formulae on $\Ga\t$, \eqref{eq:transport formula - m} and \eqref{eq:transport formula - T2}, arise by interpreting $\Ga\t$ as a continuous surface with velocity $v$, and as the union of curved elements (the lifted elements of $\Ga_h\t$) with velocity $v_h$, see \eqref{eq:definition v_h}, respectively. We will use them analogously to \cite[Section~7]{DziukElliott_L2}.

\subsubsection{Error estimates for the generalised Ritz map}

From \cite[Theorem~6.3 and 6.4]{highorderESFEM} we recall that the generalised Ritz map \eqref{eq:definition Ritz map} satisfies the following optimal high-order error estimates.
\begin{lemma}
	\label{lemma:Ritz map error}
	Let $u : \GT \to \R$ such that $u\ct$ and $(\mat)^{(j)} u\ct \in H^{k+1}(\Ga\t)$ for all $0 \leq t \leq T$ and $j=1,\dotsc,l$, for some $l \in \N$. Then, the error in the generalised Ritz map \eqref{eq:definition Ritz map} satisfies the bounds, for $0 \leq t \leq T$ and for $h \leq h_0$ with sufficiently small $h_0$,
	\begin{align*}
	&\ \|u - R_h\t u\|_{L^2(\Gat)} + h \|u - R_h\t u\|_{H^1(\Gat)} \leq c h^{k+1} \|u\|_{H^{k+1}(\Gat)} \\
%	+ \|\zeta\|_{L^2(\Gat)} \big) 	, \\
	&\ \|(\mat_h)^{(l)} (u - R_h\t u)\|_{L^2(\Gat)} + h \|(\mat_h)^{(l)} (u - R_h\t u)\|_{H^1(\Gat)} \\
	&\ \qquad\qquad  \leq c h^{k+1} \sum_{j=0}^l \| (\mat)^{(j)} u\|_{H^{k+1}(\Gat)} ,
%	\big( \|u\|_{H^{k+1}(\Gat)} + \|\mat u\|_{H^{k+1}(\Gat)} + \|\zeta\|_{L^2(\Gat)} + \|\mat \zeta\|_{L^2(\Gat)} \big) , 
	\end{align*}
	where the constant $c>0$ is independent of $h$ and $t$, but depends on $\GT$.
\end{lemma}

%[cb: Ritz map defined above, error estimates here] 
%
%For the Ritz map we have the following error estimate from [Lubich and Mansour, Wave equation on evolving surfaces, Theorem $8.2$]
%\begin{theorem}
%	Let $z \in \GT \to \R$ with $z(\cdot,t) \in H^{k+1}(\Ga\t)$ and $\zeta(\cdot,t) \in L^2(\Ga\t)$ for $ 0 \leq t\leq T$. The error in the Ritz map satisfies the bounds
%	%
%	\begin{equation*}
%	\begin{aligned}
%	\|z\t - (\widetilde{R}_h z) \t \|_{L^2(\Ga\t)} +& h \|\nb_\Ga \big(z\t - (\widetilde{R}_h z) \t\big) \|_{L^2(\Ga\t)} \\
%	&\leq c h^{k+1} \big( \|z\t \|_{H^{k+1}(\Ga\t)} + \|\zeta\t \|_{L^2(\Ga\t)}\big),
%	\end{aligned}
%	\end{equation*}
%	%
%	for $0 \leq t \leq T$ and $h \leq h_0$ with $h_0$ sufficiently small and the constant $c$ independent of $h$ and $t$.
%\end{theorem}
%The errors in the material derivative of the Ritz map is bounded according to [Lubich and Mansour, Wave equation on evolving surfaces, Theorem $8.3$] with $\zeta = \mat z$ [cb: With $\zeta = \mat z$ theorem still correct?] 
%\begin{theorem}
%	The error in the material derivatives of the Ritz map are bounded for $l \geq 1, 0 \leq t \leq T$ and $h \leq h_0$ for $h_0$ sufficiently small, by
%	%
%	\begin{equation*}
%	\begin{aligned}
%	\| \partial_h^{(l)} (z -  R z ) \t\|_{L^2(\Ga \t)} + h \| \nb_\Ga \big(\partial_h^{(l)}(z -  R z ) \t \big)\| \\
%	\leq c_l h^{k+1} \sum_{i=0}^{l} \Big( \| \partial^{(i)} z \t \|_{H^{k+1}(\Ga \t)} + \| \partial^{(i+1)} z \t \|_{L^2(\Ga \t)} \Big),
%	\end{aligned}
%	\end{equation*}
%	with $c_l$ independent of $h$ and $t$.
%	%
%\end{theorem}

\subsubsection{Geometric approximation errors}
\label{section:geometric approx errors}

The time dependent bilinear forms $m,r$ and their discrete counterparts $m_h,r_h$, from \eqref{eq:bilinear forms} and \eqref{eq:discrete bilinear forms}, respectively, satisfy the following high-order \emph{geometric approximation estimates}, see \cite[Lemma 5.6]{highorderESFEM}.
%[cb: we do not directly use the formulas for $a,a_h$ and $b,b_h$ in the paper kb: We should probably do this.] 
\begin{lemma}
\label{lemma:GeometricBilinear}
	Let $z_h, \vphi_h \in S_h\t$ arbitrary with lifts $z_h^\ell, \vphi_h^\ell \in S_h^\ell\t$. Then, for all $h \leq h_0$ with $h_0$ sufficiently small, the following estimates hold 
	\begin{align*}
	%\label{eq: GeometricBilinearM}
	|m(z_h^\ell,\vphi_h^\ell) - m_h(z_h,\vphi_h)| &\leq ch^{k+1} \|z_h^\ell\|_{L^2(\Ga\t)}\|\vphi_h^\ell\|_{L^2(\Ga\t)}, \\
	%\label{eq: GeometricBilinearA}
%	|a(z_h^\ell,\vphi_h^\ell) - a_h(z_h,\vphi_h)| &\leq ch^{k+1} \|\nbg z_h^\ell\|_{L^2(\Ga\t)}\|\nbg\vphi_h^\ell\|_{L^2(\Ga\t)}, \\
%	\end{align*}
%	and
%	%
%	\begin{align*}
	%\label{eq: GeometricBilinearG}
	|r(v_h;z_h^\ell,\vphi_h^\ell) - r_h(V_h;z_h,\vphi_h)| &\leq ch^{k+1} \|z_h^\ell\|_{L^2(\Ga\t)}\|\vphi_h^\ell\|_{L^2(\Ga\t)}, 
%	\\
	%\label{eq: GeometricBilinearB}
%	|b(v_h;z_h^\ell,\vphi_h^\ell) - b_h(V_h;z_h,\vphi_h)| &\leq ch^{k+1} \|\nbg z_h^\ell\|_{L^2(\Ga\t)}\|\nbg\vphi_h^\ell\|_{L^2(\Ga\t)},
	\end{align*}
	where the constant $c > 0$ is independent of $h$ and $t$, but depends on $\GT$.
\end{lemma}
Similar results hold for the errors in the bilinear form $a$, cf.~\cite[Lemma 5.6]{highorderESFEM}, but these are not used herein. The previous estimates also hold for any functions in $L^2(\Ga_h\t)$. Therefore, the proof of the previous lemma implies 
\begin{equation}
\label{eq:GeometricBilinearPsi}
\begin{aligned}
	&|m(f(z_h^\ell,\nb_\Ga z_h^\ell),\vphi_h^\ell) - m_h(f(z_h,\nb_{\Ga_h} z_h),\vphi_h)| \\
	\leq &ch^{k+1} \|f(z_h^\ell,\nb_\Ga z_h^\ell)\|_{L^2(\Ga\t)}\|\vphi_h^\ell\|_{L^2(\Ga\t)},\\
	&|m(\mat_h f(z_h^\ell,\nb_\Ga z_h^\ell), \vphi_h^\ell) - m_h(\mat_h f(z_h,\nb_{\Ga_h} z_h), \vphi_h)|\\ \leq &ch^{k+1} \|\mat_h f(z_h^\ell,\nb_\Ga z_h^\ell)\|_{L^2(\Ga\t)} \|\vphi_h^\ell\|_{L^2(\Ga\t)},
\end{aligned}
\end{equation}
respectively for $g$. 
Let $\mu_h$ denote the quotient of the measures on $\Ga\t$ and $\Ga_h\t$. In \cite[Lemma~5.2]{highorderESFEM} it is shown that the following estimates hold:
\begin{align}
	\label{eq: GeometricSurfaceMeasure_3}
	\|1 - \mu_h\|_{L^{\infty}(\Ga_h\t)} &\leq c h^{k+1}, \\
    \label{eq: GeometricSurfaceMeasure_1}
    \|\mat_h \mu_h\|_{L^\infty(\Ga_h\t)} &\leq c h^{k+1},\\
    \label{eq: GeometricSurfaceMeasure_2}
    \|(\mat_h)^{(2)} \mu_h\|_{L^\infty(\Ga_h\t)} &\leq c h^{k+1}.
\end{align}
Below we present and prove a new geometric approximation estimate which relates time derivatives of $r$ and $r_h$. 
\begin{lemma}
\label{lemma:GeometricNew}
	Let $z_h, \vphi_h \in S_h\t$ be arbitrary with $\mat_h z_h, \mat_h \vphi_h \in S_h\t$, with their corresponding lifts in $S_h^\ell\t$. Then, for all $h \leq h_0$ with $h_0$ sufficiently small, the following estimate holds
	\begin{equation*}
	\begin{aligned}
	&\ \Big| m((\nb_{\Ga \t} \cdot v_h)^2 \,z_h^\ell, \vphi_h^\ell) + m(\mat_h (\nb_{\Ga \t} \cdot v_h) \,z_h^\ell, \vphi_h^\ell) \\
	&\ - m_h((\nb_{\Ga_h \t} \cdot V_h)^2 \,z_h, \vphi_h) - m_h(\mat_h (\nb_{\Ga_h \t} \cdot V_h)\, z_h, \vphi_h ) \Big| \\
	\leq &\ c h^{k+1}\, \Big( \|z_h^\ell\|_{L^2(\Ga\t)}\,\|\vphi_h^\ell\|_{L^2(\Ga\t)} + \|z_h^\ell\|_{L^2(\Ga\t)} \,\|\mat_h \vphi_h^\ell\|_{L^2(\Ga\t)} \\
	&\ \phantom{c h^{k+1}\, \Big( } + \|\mat_h z_h^\ell\|_{L^2(\Ga\t)} \,\|\vphi_h^\ell\|_{L^2(\Ga\t)} \Big),
	\end{aligned}
	%\label{eq: GeometricBilinearNew}
	\end{equation*}
	where the constant $c > 0$ is independent of $h$ and $t$, but depends on the surface velocity $v$.
\end{lemma}
\begin{proof}
	Although, this lemma was first proved in \cite[Lemma~3.1.8]{Beschle_thesis}, due to its importance we present it here in full detail.
	
	We start by differentiating the integral transformation 
	\begin{equation*}
		m(z_h^\ell,\vphi_h^\ell) = m_h(z_h,\vphi_h \mu_h),
	\end{equation*}
	with respect to time using the transport formulae \eqref{eq:transport formula - T2 and T3}, to obtain
	\begin{align*}
		\diff m(z_h^\ell,\vphi_h^\ell) = &\ m(\mat_h z_h^\ell, \vphi_h^\ell) + m(z_h^\ell, \mat_h \vphi_h^\ell) + r(v_h;z_h^\ell,\vphi_h^\ell) \\
		=\diff m_h(z_h,\vphi_h \mu_h) = &\ m_h(\mat_h z_h, \vphi_h\mu_h) + m_h(z_h, (\mat_h \vphi_h)\mu_h) \\
		&\ + r_h(V_h;z_h,\vphi_h\mu_h) + m_h(z_h,(\mat_h \mu_h)\vphi_h).
	\end{align*}
	Using $\mat_h (z_h^\ell) = (\mat_h z_h)^\ell$, see \cite[Lemma~4.1]{DziukElliott_L2}, we obtain
	\begin{equation}
	\begin{aligned}
	r(v_h;z_h^\ell,\vphi_h^\ell) - r_h(V_h;z_h,\vphi_h\mu_h) = &\ m_h(\mat_h z_h, \vphi_h\mu_h) - m((\mat_h z_h)^\ell, \vphi_h^\ell)\\
	&\ +m_h(z_h, (\mat_h \vphi_h)\mu_h) - m(z_h^\ell, (\mat_h \vphi_h)^\ell)\\
	&\ +m_h(z_h,(\mat_h \mu_h)\vphi_h) \\
	=  &\ m_h(z_h,(\mat_h \mu_h)\vphi_h).
	\end{aligned}																				
	\label{eq:gdifference}
	\end{equation}
	In particular, for $\mat_h z_h$ in the role of $z_h$, and with the use of the geometric estimate for the surface measure $\|\mat_h \mu_h\|_{L^\infty} \leq c h^{k+1}$ \eqref{eq: GeometricSurfaceMeasure_1} we obtain the estimate
	\begin{align*}
		r(v_h;\mat_h z_h^\ell,\vphi_h^\ell) - r_h(V_h;\mat_h z_h,\vphi_h\mu_h) 
		= &\ m_h(\mat_hz_h,(\mat_h\mu_h) \vphi_h) \\
%		\leq&\  c \|\mat_h \mu_h\|_{L^{\infty}(\Ga_h\t)} \|\mat_h z_h\|_{L^2(\Ga_h\t)} \|\vphi_h\|_{L^2(\Ga_h\t)}\\
		\leq&\ c h^{k+1} \|\mat_h z_h\|_{L^2(\Ga_h\t)} \|\vphi_h\|_{L^2(\Ga_h\t)},
	\end{align*}
	and with $\mat_h \vphi_h$ in the role of $\vphi_h$,
	\begin{align*}
		r(v_h;z_h^\ell,\mat_h \vphi_h^\ell) - r_h(V_h;z_h,(\mat_h \vphi_h)\mu_h) 
		= &\ m_h(z_h,\mat_h \mu_h\, \mat_h \vphi_h) \\
%		\leq&\ c\|\mat_h \mu_h\|_{L^{\infty}(\Ga_h\t)} \|z_h\|_{L^2(\Ga_h\t)} \|\mat_h \vphi_h\|_{L^2(\Ga_h\t)}\\
		\leq &\ c h^{k+1} \|z_h\|_{L^2(\Ga_h\t)} \|\mat_h \vphi_h\|_{L^2(\Ga_h\t)}.
	\end{align*} 
	Differentiating equation $\eqref{eq:gdifference}$ with respect to time, using \eqref{eq:transport formula - T2 and T3}, yields
	\begin{align*}
		\diff m((\nb_\Ga \cdot v_h) z_h^\ell, \vphi_h^\ell) - \diff m_h((\nbgh \cdot V_h) z_h, \vphi_h\, \mu_h)= \diff m_h(z_h,(\mat_h \mu_h)\, \vphi_h).
	\end{align*}
	Computing the derivatives on the left-hand side then leads to
	\begin{align*}
		&\ m((\nb_\Ga \cdot v_h)^2 z_h^\ell, \vphi_h^\ell) - m_h((\nbgh \cdot V_h)^2 z_h, \vphi_h \,\mu_h) \\
		&\ + m(\mat_h (\nb_\Ga \cdot v_h) z_h^\ell, \vphi_h^\ell) - m_h(\mat_h (\nbgh \cdot V_h) z_h, \vphi_h \,\mu_h)  \\
		= &\  r_h(V_h; (\mat_h z_h, \vphi_h) \,\mu_h) - r(v_h;\mat_h z_h^\ell, \vphi_h^\ell) \\
		&\ + r_h(V_h; z_h, \mat_h \vphi_h \,\mu_h) - r(v_h; z_h^\ell, \mat_h \vphi_h^\ell)  \\	
		&\ + \diff m_h(z_h,(\mat_h \mu_h) \,\vphi_h) + r_h(V_h; z_h, \vphi_h(\mat_h \mu_h)).
	\end{align*}
	The pairs in the first two lines on the right-hand side are already estimated above, while the last term is estimated by the geometric estimate \\
	$\|\mat_h \mu_h\|_{L^\infty} \leq c h^{k+1}$ \eqref{eq: GeometricSurfaceMeasure_1}. To estimate the remaining derivative term, we first compute the time derivative by \eqref{eq:transport formula - T3} and then estimate each term to obtain
	\begin{align*}
	\diff m_h(z_h,(\mat_h \mu_h)\,\vphi_h) =&\ m_h(\mat_h z_h,(\mat_h \mu_h)\, \vphi_h) + m_h(z_h,(\mat_h \mat_h \mu_h) \,\vphi_h) \\
	&\ + m_h(z_h,(\mat_h \mu_h) \,\mat_h \vphi_h) + r_h(V_h; z_h,(\mat_h \mu_h)\, \vphi_h) \\
%	\leq&\  c\, \|(\mat_h \mu_h)\|_{L^{\infty}(\Ga_h\t)} \|\mat_h z_h\|_{L^2(\Ga_h\t)} \|\vphi_h\|_{L^2(\Ga_h\t)} \\
%	&\ + c\, \|\mat_h \mat_h \mu_h\|_{L^{\infty}(\Ga_h\t)} \|z_h\|_{L^2(\Ga\t)} \|\vphi_h\|_{L^2(\Ga\t)} \\
%	&\ + c\, \|(\mat_h \mu_h)\|_{L^{\infty}(\Ga_h\t)} \|z_h\|_{L^2(\Ga_h\t)} \|\mat_h \vphi_h\|_{L^2(\Ga_h\t)} \\
%	&\ + c \|(\mat_h \mu_h)\|_{L^{\infty}(\Ga_h\t)} \|z_h\|_{L^2(\Ga_h\t)} \|\vphi_h\|_{L^2(\Ga_h\t)} \\
	\leq &\ c h^{k+1} \Big(\|z_h\|_{L^2(\Ga_h\t)}\|\vphi_h\|_{L^2(\Ga_h\t)} \\
	&\ \phantom{c h^{k+1} \Big(} + \|z_h\|_{L^2(\Ga_h\t)} \|\mat_h \vphi_h\|_{L^2(\Ga_h\t)} \\
	&\ \phantom{c h^{k+1} \Big(} + \|\mat_h z_h\|_{L^2(\Ga_h\t)} \|\vphi_h\|_{L^2(\Ga_h\t)} \Big),
	\end{align*}
	using the geometric error estimate $\|(\mat_h)^{(2)} \mu_h\|_{L^\infty} \leq c h^{k+1}$  \eqref{eq: GeometricSurfaceMeasure_2}.
	
	Altogether, by triangle inequalities and by combining the above estimates, we obtain
	\begin{equation*}
	\begin{aligned}
	&\ |m((\nb_\Ga \cdot v_h)^2 z_h^\ell, \vphi_h^\ell) - m_h((\nbgh \cdot V_h)^2 z_h, \vphi_h) \\
	&\  + m(\mat_h (\nb_\Ga \cdot v_h) z_h^\ell, \vphi_h^\ell) - m_h(\mat_h (\nbgh \cdot V_h) z_h, \vphi_h )| \\
	= &\ |m((\nb_\Ga \cdot v_h)^2 z_h^\ell, \vphi_h^\ell) - m_h((\nbgh \cdot V_h)^2 z_h, \vphi_h\mu_h) \\
	&\ + m_h((\nbgh \cdot V_h)^2 z_h, \vphi_h\mu_h) - m_h((\nbgh \cdot V_h)^2 z_h, \vphi_h) \\
	&\  + m(\mat_h (\nb_\Ga \cdot v_h) z_h^\ell, \vphi_h^\ell) 	- m_h(\mat_h (\nbgh \cdot V_h) z_h, \vphi_h\mu_h) \\
	&\  +m_h(\mat_h (\nbgh \cdot V_h) z_h, \vphi_h\mu_h) -  m_h(\mat_h (\nbgh \cdot V_h) z_h, \vphi_h )|\\
	\leq &\ |m((\nb_\Ga \cdot v_h)^2 z_h^\ell, \vphi_h^\ell) - m_h((\nbgh \cdot V_h)^2 z_h, \vphi_h\mu_h) \\
	&\  + m(\mat_h (\nb_\Ga \cdot v_h) z_h^\ell, \vphi_h^\ell) 	- m_h(\mat_h (\nbgh \cdot V_h) z_h, \vphi_h\mu_h)| \\
	&\ + |m_h((\nbgh \cdot V_h)^2 z_h, \vphi_h(\mu_h - 1))|\\
	&\  +|m_h(\mat_h (\nbgh \cdot V_h) z_h, \vphi_h(\mu_h - 1))| \\
	\leq &\  c h^{k+1} \Big(\|z_h\|_{L^2(\Ga_h\t)}\|\vphi_h\|_{L^2(\Ga_h\t)} + \|z_h\|_{L^2(\Ga_h\t)} \|\mat_h \vphi_h\|_{L^2(\Ga_h\t)} \\
	&\ + \|\mat_h z_h\|_{L^2(\Ga_h\t)} \|\vphi_h\|_{L^2(\Ga_h\t)} \Big) \\
	&\ + c \|(\mu_h - 1)\|_{L^{\infty}(\Ga_h\t)} \|z_h\|_{L^2(\Ga_h\t)}\|\vphi_h\|_{L^2(\Ga_h\t)}\\
	\leq &\ c h^{k+1} \Big(\|z_h\|_{L^2(\Ga_h\t)}\|\vphi_h\|_{L^2(\Ga_h\t)} + \|z_h\|_{L^2(\Ga_h\t)} \|\mat_h \vphi_h\|_{L^2(\Ga_h\t)} \\
	&\ + \|\mat_h z_h\|_{L^2(\Ga_h\t)} \|\vphi_h\|_{L^2(\Ga_h\t)} \Big),
	\end{aligned}
	\end{equation*}
	where we have used the bounds on the discrete velocity from Lemma \ref{lemma:VhBhBound}, and the geometric estimate $\|1 - \mu_h\|_{L^{\infty}} \leq c h^{k+1}$ \eqref{eq: GeometricSurfaceMeasure_3}.
\qed\end{proof}

\subsection{Defect bounds}

In this section we prove bounds for the defects and for their time derivatives, i.e.~we prove that condition \eqref{eq:defect bounds - assumed} of Proposition~\ref{proposition:stability} is indeed satisfied.
\begin{proposition}
\label{proposition:consistency}
	Let $u,w$ solve the Cahn--Hilliard equation on an evolving surface \eqref{eq:CH system}. Furthermore, let $u,w$ and the continuous surface velocity $v$ be sufficiently smooth, e.g.~satisfying \eqref{eq:regularity conditions}. Then, for all $h \leq h_0$ sufficiently small, and for all $t \in [0,T]$: 
	
	(a) For general nonlinearities $f$ and $g$ the defects are bounded as
	\begin{equation}
	\label{eq:defect bounds}
		\begin{aligned}
			&\ \|\bfdu\t\|_{\bfM\t} = \|d_u\|_{L^2(\Ga_h\t)} \leq ch^{k},\\
			&\ \|\bfdud\t\|_{\bfM\t} = \|\mat_h d_u\|_{L^2(\Ga_h\t)}\leq ch^{k} ,\\
			&\ \|\bfdw\t\|_{\bfM\t} = \|d_w\|_{L^2(\Ga_h\t)}\leq ch^{k},\\
			&\ \|\bfdwd\t\|_{\bfM\t} = \|\mat_h d_w\|_{L^2(\Ga_h\t)}\leq ch^{k} .
		\end{aligned}
	\end{equation}
	
	(b) If $f$ and $g$ are both \emph{independent} of $\nbg u$, then the above estimates in \eqref{eq:defect bounds} are improved to $O(h^{k+1})$.
%	\begin{equation}
%	\label{eq:defect bounds - improved}
%		\begin{aligned}
%			&\ \|\bfdu\t\|_{\bfM\t} = \|d_u\|_{L^2(\Ga_h\t)} \leq ch^{k+1},\\
%			&\ \|\bfdud\t\|_{\bfM\t} = \|\mat_h d_u\|_{L^2(\Ga_h\t)}\leq ch^{k+1} ,\\
%			&\ \|\bfdw\t\|_{\bfM\t} = \|d_w\|_{L^2(\Ga_h\t)}\leq ch^{k+1},\\
%			&\ \|\bfdwd\t\|_{\bfM\t} = \|\mat_h d_w\|_{L^2(\Ga_h\t)}\leq ch^{k+1} .
%		\end{aligned}
%	\end{equation}
	
	The constant $c > 0$ is independent of $h$ and $t$, but depends on the bounds on Sobolev norms of $u,w$ and the surface velocity $v$.
\end{proposition}

\begin{proof}	
The Ritz map \eqref{eq:definition Ritz map} of the exact solutions $u$ and $w$ satisfies the discrete problem only up to some defects, $d_u\ct \in S_h\t$ and $d_w\ct \in S_h\t$, defined in \eqref{eq:defect definition}.
Rewriting these equations using the bilinear form notation from \eqref{eq:bilinear forms}, we thus have, for an arbitrary $\vphi_h \in S_h\t$, 
\begin{equation}
\label{eq: defect}
\begin{aligned}
	m_h(d_u, \vphi_h) =\ & m_h(\mat_h\widetilde{R}_h u, \vphi_h) + a_h(\widetilde{R}_h w, \vphi_h) \\
	\ & - m_h(f(\widetilde{R}_h u,\nb_{\Ga_h}\widetilde{R}_h u), \vphi_h) + r_h(V_h;\widetilde{R}_h u, \vphi_h) , \\
	m_h(d_w, \vphi_h) =\ & a_h(\widetilde{R}_h u, \vphi_h) + m_h(g(\widetilde{R}_h u,\nb_{\Ga_h}\widetilde{R}_h u), \vphi_h) - m_h(\widetilde{R}_h w, \vphi_h).
	\end{aligned}
\end{equation}

Upon subtracting the corresponding equations for the exact solution \eqref{eq:CH WeakSolution - diff} with $\vphi=\vphi_h^\ell$ and applying the transport formula \eqref{eq:transport formula - T2} (with $\mat_h \vphi_h^\ell = 0$), from the equations in \eqref{eq: defect}, and then adding and subtracting some terms in order to apply the definition of the Ritz map $\widetilde{R}_h$ \eqref{eq:definition Ritz map}, we obtain the following two equations satisfied by the defects $d_u$ and $d_w$:
\begin{subequations}
\label{eq:defect expressions}
	\begin{align}
	\label{eq:defect expression - d_u}
		m_h(d_u, \vphi_h) = &\ \Big( m_h(\mat_h\widetilde{R}_h u, \vphi_h) - m(\mat_h u, \vphi_h^\ell) \Big) \nonumber \\
%		&\ + \Big( a_h^*(\widetilde{R}_h w, \vphi_h) - a^*(w, \vphi_h^\ell) \Big) \nonumber \\
		&\ - \Big( m_h(\widetilde{R}_h w, \vphi_h) - m(w, \vphi_h^\ell) \Big) \nonumber \\
		&\ + \Big( r_h(V_h;\widetilde{R}_h u, \vphi_h) - r(v_h; u, \vphi_h^\ell) \Big) \nonumber \\
		&\  - \Big(m_h(f(\widetilde{R}_h u,\nb_{\Ga_h}\widetilde{R}_h u), \vphi_h) - m(f(u,\nb_{\Ga} u), \vphi_h^\ell)\Big) \nonumber \\
		= &\ I_u + II_u + III_u + IV_u, \\
		\label{eq:defect expression - d_w}
		m_h(d_w, \vphi_h) = &\ - \Big( m_h(\widetilde{R}_h u, \vphi_h) - m(u, \vphi_h^\ell) \Big) \nonumber\\
		&\ - \Big( m_h(\widetilde{R}_h w, \vphi_h) - m(w, \vphi_h^\ell) \Big)  \nonumber \\
		&\  + \Big(m_h(g(\widetilde{R}_h u,\nb_{\Ga_h}\widetilde{R}_h u), \vphi_h) - m(g(u,\nb_{\Ga} u), \vphi_h^\ell)\Big) \nonumber  \\
		= &\ I_w + II_w + III_w.
	\end{align}
\end{subequations}
We now estimate the defects and their material derivatives in the $L^2(\Ga\t)$ norm by bounding each pair on the right-hand sides of the above equations separately, using the geometric estimates from the previous subsection and using similar techniques as in \cite{DziukElliott_L2,highorderESFEM}. 
Since throughout the proofs most norms are on $\Ga\t$, we will omit these below and write $L^2$, $H^{k+1}$ instead of  $L^2(\Ga\t)$, $H^{k+1}(\Ga\t)$, etc.

%
%%%%%%%%%%%%%%%%%%%%%%%%%%%%%%%%%%%%%%%%%%%%%%%%%%%%%%%%%%%%%%%%%%%%%%%%%%%%%%%%%%%%%%%%%%%%%%%
%%%%%%%%%%%%%%%%%%%%%%%%%%%%%%%%%%% DEFECT du  %%%%%%%%%%%%%%%%%%%%%%%%%%%%%%%%%%%%%%%%%%%%%%%%
%%%%%%%%%%%%%%%%%%%%%%%%%%%%%%%%%%%%%%%%%%%%%%%%%%%%%%%%%%%%%%%%%%%%%%%%%%%%%%%%%%%%%%%%%%%%%%%
%
\emph{Bound for $d_u$:}	
For the pair in the first line, we add and subtract terms to obtain
\begin{equation}
\label{eq:est for d_u terms - mass term with material derivative}
	\begin{aligned}
		I_u 
		= &\ \Big( m_h(\mat_h\widetilde{R}_h u, \vphi_h) - m(\mat_h R_h u, \vphi_h^\ell) \Big) 
%		\\ &\ 
		+ m(\mat_h (R_h u - u), \vphi_h^\ell) \\
		\leq &\ c h^{k+1} \|\mat_h R_h u\|_{L^2} \|\vphi_h^\ell\|_{L^2} 
%		\\ &\ 
		+ c h^{k+1} \Big(\|u\|_{H^{k+1}} + \|\mat u\|_{H^{k+1}} \Big) \|\vphi_h^\ell\|_{L^2} \\
		\leq &\ c h^{k+1} \Big( \|u\|_{H^{k+1}} + \|\mat u\|_{H^{k+1}} \Big) \|\vphi_h^\ell\|_{L^2} ,
	\end{aligned}	
\end{equation}
where we have used Lemma~\ref{lemma:GeometricBilinear} together with the fact that $\mat_h (z_h^\ell) = (\mat_h z_h)^\ell$ (\cite[Lemma~4.1]{DziukElliott_L2}) and the Ritz map error bound Lemma~\ref{lemma:Ritz map error}. The Ritz map error estimate is again used to show the bound $\|\mat_h R_h u\|_{L^2} \leq c ( \|u\|_{H^{k+1}} + \|\mat u\|_{H^{k+1}} )$.

By the same techniques, we prove the following bound for $II_u$:
\begin{equation}
\label{eq:est for d_u terms - mass term}
	\begin{aligned}
		II_u \leq &\
		- \Big( m_h(\widetilde{R}_h w, \vphi_h) - m(R_h w, \vphi_h^\ell) \Big) - m(R_h w - w, \vphi_h^\ell) \\
		%\leq &\ c h^{k+1} \Big( \|u\|_{H^{k+1}} + \|\mat u\|_{H^{k+1}} + \|w\|_{H^{k+1}} + \|\mat w\|_{H^{k+1}} \Big) \|\vphi_h^\ell\|_{L^2} \\
		\leq &\ c h^{k+1} \|w\|_{H^{k+1}} \|\vphi_h^\ell\|_{L^2} .
	\end{aligned}
\end{equation}
The third term $III_u$ is estimated using similar arguments as before, by Lemma~\ref{lemma:GeometricBilinear}, Lemma~\ref{lemma:Ritz map error}, and the boundedness of $v_h$ (proved using Lemma~\ref{lemma:velocity error estimate}), 
\begin{equation}
\label{eq:est for d_u terms - r term}
	\begin{alignedat}{3}
		III_u 
		= &\ r_h(V_h;\widetilde{R}_h u, \vphi_h) - r(v_h; u, \vphi_h^\ell) \\
		= &\ \Big( r_h(V_h;\widetilde{R}_h u, \vphi_h) - r(v_h; R_h u, \vphi_h^\ell) \Big) 
		+ r(v_h; R_h u - u, \vphi_h^\ell) \\
%%		\\	&\ 
%		+ r(v_h - v; u, \vphi_h^\ell) 
%		\\	&\ 
%		+ m(u , (v_h - v) \cdot \nbg \vphi_h^\ell) && \quad \text{(by \eqref{eq:defect pair - pre estimate form})} \\
%		= &\ \Big( r_h(V_h;\widetilde{R}_h u, \vphi_h) - r(v_h; R_h u, \vphi_h^\ell) \Big) \\
%		&\ + r(v_h; R_h u - u, \vphi_h^\ell)
%		\\	&\ 
%		- m((v_h - v) \cdot \nbg u,  \vphi_h^\ell) 
%		- m(u , (v_h - v) \cdot \nbg \vphi_h^\ell)
%		\\	&\ 
%		+ m(u , (v_h - v) \cdot \nbg \vphi_h^\ell) 
%		&& \quad \text{(by \eqref{eq:Green's formula for g term})} \\
%		\leq &\ c h^{k+1} \|R_h u\|_{L^2} \|\vphi_h^\ell\|_{L^2} 
%		+ c h^{k+1} \|u\|_{H^{k+1}} \|\vphi_h^\ell\|_{L^2} 
%		+ c h^{k+1} \|u\|_{H^1} \|\vphi_h^\ell\|_{L^2} \\
		\leq &\ c h^{k+1} \|u\|_{H^{k+1}} \|\vphi_h^\ell\|_{L^2} .
	\end{alignedat}
\end{equation}
 
The fourth term $IV_u$ including the non-linearity is estimated using the above techniques, and in addition, due to the (locally Lipschitz continuous) non-linear terms $f$ and $g$, requires a $W^{1,\infty}$ bound on the Ritz map, which we obtain by 
%\begin{equation}
%\label{eq:Ritz map L^infty bound}
%\begin{aligned}
%	\|R_h w\|_{L^\infty} \! 
%	\leq &\ \|R_h w - I_h w\|_{L^\infty} + \|I_h w\|_{L^\infty} \\
%	\leq &\ c h^{-d/2} \|R_h w - I_h w\|_{L^2} + \|I_h w\|_{L^\infty} \\
%	\leq &\ c h^{-d/2} \big( \|R_h w - w\|_{L^2} + \|w - I_h w\|_{L^2} \big) \! + \|I_h w - w\|_{L^\infty} \! + \|w\|_{L^\infty} \\
%	\leq &\ c h^{k+1-d/2} \big(\|u\|_{H^{k+1}} + \|w\|_{L^{2}}\big) %+ c h^{k+1-d/2} \|w\|_{H^{k+1}} 
%	+ (c h^2 + 1) \|w\|_{W^{2,\infty}} , 
%\end{aligned}
%\end{equation}
\begin{equation}
\label{eq:Ritz map W^1,infty bound}
\begin{aligned}
\|R_h u\|_{W^{1,\infty}} \! 
\leq &\ \|R_h u - I_h u\|_{W^{1,\infty}} + \|I_h u\|_{W^{1,\infty}} \\
\leq &\ c h^{-d/2} \|R_h u - I_h u\|_{H^1} + \|I_h u\|_{W^{1,\infty}} \\
\leq &\ c h^{-d/2} \big( \|R_h u - u\|_{H^1} + \|u - I_h u\|_{H^1} \big) \! + \|I_h u - u\|_{W^{1,\infty}} \! + \|u\|_{W^{1,\infty}} \\
\leq &\ c h^{k-d/2} \|u\|_{H^{k+1}} %+ c h^{k+1-d/2} \|w\|_{H^{k+1}} 
+ (c h + 1) \|u\|_{W^{2,\infty}} , 
\end{aligned}
\end{equation}
with $k - d/2\geq 0$, using an inverse estimate \cite[Theorem~4.5.11]{BreS08}, interpolation error bounds Lemma~\ref{lemma:interpolation error estimate}, and for the last term the (sub-optimal) interpolation error estimate of \cite[Proposition~2.7]{Demlow} (with $p=\infty$).
We then estimate, using 
\begin{equation}
\label{eq:est for d_u terms - g term}
	\begin{aligned}
		IV_u \leq &\  m_h(f(\widetilde{R}_h u,\nb_{\Ga_h}\widetilde{R}_h u), \vphi_h) - m(f(R_h u, \nb_\Ga R_h u), \vphi_h^\ell) \\
		&\ + m(f(R_h u,\nb_{\Ga} R_h u) - f(u,\nb_{\Ga} u), \vphi_h^\ell) \\
		\leq &\ ch^{k+1} \|f(R_h u, \nb_\Ga R_h u)\|_{L^2} \|\vphi_h^\ell\|_{L^2} 
		%	\\ &\
		+ c \|R_h u - u\|_{H^1} \|\vphi_h^\ell\|_{L^2} \\
		\leq &\ \Big( c h^{k+1} \|f(R_h u, \nb_\Ga R_h u)\|_{L^2} + c h^k \|u\|_{H^{k+1}} \Big) \|\vphi_h^\ell\|_{L^2} \\
%		\leq &\ \Big( c h^{k+1} \|f(R_h u, \nb_\Ga R_h u) - f(u, \nb_\Ga u)\|_{L^2} + \|f(u, \nb_\Ga u)\|_{L^2} + c h^k \|u\|_{H^{k+1}} \Big)  \|\vphi_h^\ell\|_{L^2} \\
		\leq &\ \Big( c h^{k+1} \big( c \|u\|_{H^{k+1}} + \|f(u, \nb_\Ga u)\|_{L^2} \big) + c h^k \|u\|_{H^{k+1}} \Big) \|\vphi_h^\ell\|_{L^2} .
	\end{aligned}
\end{equation}
Note in particular that the only term in all of the above consistency estimates which is of order $O(h^k)$ is the last term in \eqref{eq:est for d_u terms - g term}, which is due to the presence of $\nbg u$ in the nonlinearity.

The estimates \eqref{eq:est for d_u terms - mass term with material derivative}--\eqref{eq:est for d_u terms - g term} together, using the norm equivalence \eqref{eq:norm equivalence}, and the definition of the $L^2$ norm, in general for $f(u,\nbg u)$, yields
\begin{subequations}
\label{eq:est for d_u}
\begin{equation}
	\begin{aligned}
	\|d_u\|_{L^2} = &\ \sup_{0 \neq \vphi_h \in S_h} \frac{m_h(d_u,\vphi_h)}{\|\vphi_h\|_{L^2}} \\
%	\leq &\ c h^{k+1} \Big( \|u\|_{H^{k+1}} + \|\mat u\|_{H^{k+1}} + \|w\|_{H^{k+1}} + \|\mat w\|_{H^{k+1}} \Big) .
	\leq &\ c h^{k} \Big( \|u\|_{H^{k+1}} + \|\mat u\|_{H^{k+1}} + \|w\|_{H^{k+1}}\Big) .
%	\leq &\ c h^{k} \sum_{j=0}^1 \Big( \|(\mat)^{(j)} u\|_{H^{k+1}} + \|(\mat)^{(j)} w\|_{H^{k+1}} \Big) .
	\end{aligned} 
\end{equation}
 
If $f$ is \emph{independent} of $\nbg u$, then by the note after \eqref{eq:est for d_u terms - g term}, the defect estimate improves to 
\begin{equation}
	\|d_u\|_{L^2}
	\leq c h^{k+1} \Big( \|u\|_{H^{k+1}} + \|\mat u\|_{H^{k+1}} + \|w\|_{H^{k+1}}\Big) .
\end{equation}
\end{subequations}

%%%%%%%%%%%%%%%%%%%%%%%%%%%%%%%%%%%%%%%%%%%%%%%%%%%%%%%%%%%%%%%%%%%%%%%%%%%%%%%%%%%%%%%%%%%%%%%
%%%%%%%%%%%%%%%%%%%%%%%%%%%%%%%%%%% DEFECT pdu %%%%%%%%%%%%%%%%%%%%%%%%%%%%%%%%%%%%%%%%%%%%%%%%
%%%%%%%%%%%%%%%%%%%%%%%%%%%%%%%%%%%%%%%%%%%%%%%%%%%%%%%%%%%%%%%%%%%%%%%%%%%%%%%%%%%%%%%%%%%%%%%
\emph{Bound for $\mat_h d_u$:}
We start by differentiating the defect equation for $d_u$ \eqref{eq:defect expression - d_u} with respect to time. Using that  $\mat_h\vphi_h = \mat_h (\vphi_h^\ell) = 0$, we obtain
\begin{equation*}
	\begin{aligned}
		m_h(\mat_h d_u , \vphi_h) = - r_h(V_h;d_u,\vphi_h) + \diff \Big(I_u + II_u + III_u + IV_u \Big) .
	\end{aligned}
\end{equation*}
The first term is immediately bounded, using Lemma~\ref{lemma:VhBhBound}, the Cauchy--Schwarz inequality and \eqref{eq:est for d_u}, by 
\begin{equation}
\label{eq:est for mat d_u terms - g_h}
%	r_h(V_h;d_u,\vphi_h) \leq c h^{k+1} \Big( \|u\|_{H^{k+1}} + \|\mat u\|_{H^{k+1}} + \|w\|_{H^{k+1}} + \|\mat w\|_{H^{k+1}} \Big) .
	r_h(V_h;d_u,\vphi_h) \leq c h^{k+1} \Big( \|u\|_{H^{k+1}} + \|\mat u\|_{H^{k+1}} + \|w\|_{H^{k+1}} \Big) \|\vphi_h^\ell\|_{L^2} .
\end{equation}
%\begin{equation*}
%\begin{aligned}
%	\diff m(w,\vphi_h^\ell) - m_h(\widetilde{R}_h w, \vphi_h) =&\ m_h(\mat_h \widetilde{R}_h w, \vphi_h) - m(\mat_h w, \vphi_h^\ell)\\
%	&+ r_h(V_h; \widetilde{R}_h w, \vphi_h) - r(v; w, \vphi_h^\ell)\\
%	%
%	\diff m_h(\mat_h\widetilde{R}_h u, \vphi_h) - m(\mat u, \vphi_h^\ell) =&\ m_h(\mat_h \mat_h\widetilde{R}_h u, \vphi_h) - m(\mat_h 			\mat u, \vphi_h^\ell)    \\
%	&+ r_h(V_h; \mat_h\widetilde{R}_h u, \vphi_h) - r(v; \mat  u, \vphi_h^\ell)\\
%	%
%	\diff r_h(V_h;\widetilde{R}_h u, \vphi_h) - r(v_h; R_h u, \vphi_h^\ell) =&\ m_h(\mat_h (\nbgh \cdot V_h)\widetilde{R}_h u, \vphi_h) - m(\mat_h (\nbg \cdot v_h) R_h u, \vphi_h^\ell) \\
%	&+ r_h(V_h, \mat_h\widetilde{R}_h u, \vphi_h) - r(v_h, \mat_h R_h u, \vphi_h^\ell)\\
%	&+ m_h((\nbgh \cdot V_h)^2\widetilde{R}_h u, \vphi_h) - m((\nbg \cdot v_h)^2  R_h u, \vphi_h^\ell) \\
%	%
%	\diff r(v; R_h u - u, \vphi_h^\ell) = &\ m(\mat_h (\nbg \cdot v) ( R_h u - u), \vphi_h^\ell) \\
%	&+ m( (\nbg \cdot v) \mat_h ( R_h u - u), \vphi_h^\ell) \\
%	&+ m((\nbg \cdot v_h )(\nbg \cdot v ) ( R_h u - u), \vphi_h^\ell) \\
%	\diff -m(\nbg R_h u, (v_h - v ) \vphi_h^\ell) =& - m(\mat_h\nbg R_h u, (v_h - v ) \vphi_h^\ell) \\
%	&- m(\nbg R_h u, \mat_h(v_h - v ) \vphi_h^\ell) \\
%	&- m((\nbgh \cdot v_h)\nbg R_h u, (v_h - v ) \vphi_h^\ell)
%\end{aligned}	
%\end{equation*}
The terms differentiated in time are estimated separately, using analogous techniques as before.

For the first term, by the transport formulas \eqref{eq:transport formula - T2} and \eqref{eq:transport formula - T3}, we obtain
\begin{equation}
\label{eq:est for mat d_u terms - I}
	\begin{aligned}
		\diff I_u 
%		= &\ \diff \Big( m_h(\mat_h\widetilde{R}_h u, \vphi_h) - m(\mat_h u, \vphi_h^\ell) \Big) \\
		= &\ \Big( m_h((\mat_h)^{(2)} \widetilde{R}_h u, \vphi_h) - m((\mat_h)^{(2)} u, \vphi_h^\ell) \Big) \\
		&\ + \Big( r_h(V_h; \mat_h\widetilde{R}_h u, \vphi_h) - r(v_h; \mat_h u, \vphi_h^\ell) \Big) \\
		\leq &\  c h^{k+1} \sum_{j=0}^2 \|(\mat)^{(j)} u\|_{H^{k+1}} \|\vphi_h^\ell\|_{L^2} 
		+ c h^{k+1} \sum_{j=0}^1 \|(\mat)^{(j)} u\|_{H^{k+1}} \|\vphi_h^\ell\|_{L^2} ,
	\end{aligned}
\end{equation}
where for the inequality we used the arguments used to show \eqref{eq:est for d_u terms - mass term with material derivative} and \eqref{eq:est for d_u terms - r term}.

By the same arguments, for the second term we obtain the bound
\begin{equation}
\label{eq:est for mat d_u terms - II}
	\begin{aligned}
		\diff II_u 
%		= &\ - \diff \Big( m_h(\widetilde{R}_h w, \vphi_h) - m(w, \vphi_h^\ell) \Big) \\
		= &\ - \Big( m_h(\mat_h \widetilde{R}_h w, \vphi_h) - m(\mat_h w, \vphi_h^\ell) \Big) \\
		&\ - \Big( r_h(V_h; \widetilde{R}_h w, \vphi_h) - r(v_h; w, \vphi_h^\ell) \Big) \\
%		\leq &\ c h^{k+1} \sum_{j=0}^1 \Big( \|(\mat)^{(j)} u\|_{H^{k+1}} + \|(\mat)^{(j)} w\|_{H^{k+1}} \Big) \|\vphi_h^\ell\|_{H^1} .
		\leq &\ c h^{k+1} \Big( \|w\|_{H^{k+1}} + \|\mat w\|_{H^{k+1}} \Big) \|\vphi_h^\ell\|_{L^2} .
	\end{aligned}
\end{equation}
By the time differentiation of the third term, using the transport formulas \eqref{eq:transport formula - T2} and \eqref{eq:transport formula - T3}, we obtain
\begin{equation*}
	\begin{aligned}
		\diff III_u  
		= &\ \diff \Big( r_h(V_h;\widetilde{R}_h u, \vphi_h) - r(v_h; u, \vphi_h^\ell) \Big) \Big) \\
		= &\ \Big[ m_h(\mat_h (\nbgh \cdot V_h)\widetilde{R}_h u, \vphi_h) + m_h((\nbgh \cdot V_h)^2\widetilde{R}_h u, \vphi_h)  \\
		&\ - m(\mat_h (\nbg \cdot v_h) u, \vphi_h^\ell) - m((\nbg \cdot v_h)^2 u, \vphi_h^\ell) \Big] \\
		&\ + \Big( \big( r_h(V_h, \mat_h\widetilde{R}_h u, \vphi_h) - r(v_h, \mat_h u, \vphi_h^\ell) \Big) =: \boldsymbol{\dot}{III}_u^1 + \boldsymbol{\dot}{III}_u^2  .
%		&\ + m(\mat u , (v_h - v) \cdot \nbg \vphi_h^\ell) + m(u , \mat(v_h - v) \cdot \nbg \vphi_h^\ell) \\
%		&\ + m(u , (v_h - v) \cdot \mat (\nbg \vphi_h^\ell)) \Big) =: \diff III_u^1 + \diff III_u^2  .
	\end{aligned}
\end{equation*}
The pair in the third line is estimated by previous arguments just as before, by
\begin{equation}
\label{eq:est for mat d_u terms - III - easy ones}
	\begin{aligned}
		\boldsymbol{\dot}{III}_u^2 \leq c h^{k+1} \Big( \|u\|_{H^{k+1}} + \|\mat u\|_{H^{k+1}} \Big) \|\vphi_h^\ell\|_{L^2} .
	\end{aligned}
\end{equation}
The remaining pair in the rectangular brackets is estimated by similar ideas as above, adding and subtracting intermediate terms, using the geometric approximation estimate from Lemma~\ref{lemma:GeometricNew}, Ritz map error estimates Lemma~\ref{lemma:Ritz map error} and bounds on expressions with $v_h$ (shown using Lemma~\ref{lemma:velocity error estimate} with $l = 0$ and $1$), and Lemma~\ref{lemma:VhBhBound}:
\begin{equation}
\label{eq:est for mat d_u terms - III - new}
	\begin{aligned}
		\boldsymbol{\dot}{III}_u^1 = &\ \Big( m_h(\mat_h (\nbgh \cdot V_h)\widetilde{R}_h u, \vphi_h) + m_h((\nbgh \cdot V_h)^2\widetilde{R}_h u, \vphi_h)  \\
		&\ - m(\mat_h (\nbg \cdot v_h) R_h u, \vphi_h^\ell) - m((\nbg \cdot v_h)^2 R_h u, \vphi_h^\ell) \Big) \\
		&\ + m(\mat_h (\nbg \cdot v_h) (R_h u - u) , \vphi_h^\ell) 
		+ m((\nbg \cdot v_h)^2 (R_h u - u) , \vphi_h^\ell) \\
		\leq &\ c h^{k+1} \Big(\| R_h u\|_{L^2} + \|\mat R_h u\|_{L^2} \Big) \|\vphi_h^\ell\|_{L^2} 
		+ c h^{k+1} \|u\|_{H^{k+1}} \|\vphi_h^\ell\|_{L^2} \\
		\leq &\ c h^{k+1} \Big( \|u\|_{H^{k+1}} + \|\mat u\|_{H^{k+1}} \Big) \|\vphi_h^\ell\|_{L^2} .
	\end{aligned}	
\end{equation}
%%\redon 
%%\begin{equation}
%%%\label{eq:est for mat d_u terms - III - new}
%%	\begin{aligned}
%%		\diff III_u^1 = &\ \Big( m_h(\mat_h (\nbgh \cdot V_h)\widetilde{R}_h u, \vphi_h) + m_h((\nbgh \cdot V_h)^2\widetilde{R}_h u, \vphi_h)  \\
%%		&\ - m(\mat (\nbg \cdot v) u, \vphi_h^\ell) - m((\nbg \cdot v)^2 u, \vphi_h^\ell) \Big) \\
%%		= &\ \Big( m_h(\mat_h (\nbgh \cdot V_h)\widetilde{R}_h u, \vphi_h) + m_h((\nbgh \cdot V_h)^2\widetilde{R}_h u, \vphi_h) \\
%%		&\ - m(\mat_h (\nbg \cdot v_h) R_h u, \vphi_h^\ell) - m((\nbg \cdot v_h)^2 R_h u, \vphi_h^\ell) \Big) \\
%%		&\ + m(\mat_h (\nbg \cdot v_h) R_h u - u, \vphi_h^\ell) + m((\nbg \cdot v_h)^2 R_h u - u, \vphi_h^\ell) \\
%%		&\ - m((\mat (\nbg \cdot v) - \mat_h (\nbg \cdot v_h)) u, \vphi_h^\ell) - m((\nbg \cdot v)^2 - (\nbg \cdot v_h)^2 u, \vphi_h^\ell) 
%%	\end{aligned}	
%%\end{equation}
%%
%%\redoff 
%%\blueon 
 
For the time derivative of the fourth term using
% \begin{equation}
% \label{eq:est for mat d_u terms - IV}
% 	\begin{aligned}
% 		\diff IV_u \leq &\ m_h(\partial_1 g(\widetilde{R}_h u,\nb_{\Ga_h}\widetilde{R}_h u) \mat_h \widetilde{R}_h u, \vphi_h) - m_h(\partial_1  g(u,\nb_{\Ga} u) \mat_h  u, \vphi_h^\ell) \\
% 		& \ +  m_h(\partial_2 g(\widetilde{R}_h u,\nb_{\Ga_h}\widetilde{R}_h u) \mat_h \nb_{\Ga_h} \widetilde{R}_h u, \vphi_h) - m_h(\partial_2  g(u,\nb_{\Ga} u) \mat_h \nb_{\Ga} u, \vphi_h^\ell) \\
% 		 & \ +  m_h((\nb_{\Ga_h} \cdot V_h) \,g(\widetilde{R}_h u,\nb_{\Ga_h}\widetilde{R}_h u), \vphi_h) - m( (\nb_{\Ga} \cdot v_h)\, g(u,\nb_{\Ga} u), \vphi_h^\ell) \\
% 		\leq &\ \Big( c h^{k+1} \big( c \|u\|_{H^{k+1}} + \|g(u, \nb_\Ga u)\|_{L^2} + \|\partial_1 g(u, \nb_\Ga u)\|_{L^2} + \|\partial_2 g(u, \nb_\Ga u)\|_{L^2} \big) + c h^k \|u\|_{H^{k+1}} \Big) \|\vphi_h^\ell\|_{L^2}
% 	\end{aligned}
% \end{equation}
\begin{equation*}
 \begin{aligned}
 \mat_h \big( f(\widetilde{R}_h u,\nb_{\Ga_h}\widetilde{R}_h u) \big) = & \ \pa_1  f(\widetilde{R}_h u,\nb_{\Ga_h}\widetilde{R}_h u) \mat_h(\widetilde{R}_h u) \\
 & \ + \pa_2 f(\widetilde{R}_h u,\nb_{\Ga_h}\widetilde{R}_h u) \mat_h (\nb_{\Ga_h} \widetilde{R}_h u), \\
\mat_h \big( f(u,\nb_{\Ga} u) \big) = & \ \pa_1 f(u,\nb_{\Ga} u)\mat_h u + \pa_2 f(u,\nb_{\Ga} u) \mat_h (\nb_{\Ga} u) ,
 \end{aligned}
\end{equation*}
we obtain
\begin{equation}
\label{eq:est for mat d_u terms - IV - new}
	\begin{aligned}
		\diff IV_u = &\ m_h(\pa_1  f(\widetilde{R}_h u,\nb_{\Ga_h}\widetilde{R}_h u) \mat_h\widetilde{R}_h u, \vphi_h) - m(\pa_1 f(u,\nb_{\Ga} u)\mat_h u, \vphi_h^\ell) \\
		 & \ + m_h(\pa_2 f(\widetilde{R}_h u,\nb_{\Ga_h}\widetilde{R}_h u) \mat_h (\nb_{\Ga_h} \widetilde{R}_h u), \vphi_h)\\
		 & \ - m(\pa_2 f(u,\nb_{\Ga} u) \mat_h (\nb_{\Ga} u), \vphi_h^\ell) \\
		 & \ +  r_h( V_h; f(\widetilde{R}_h u,\nb_{\Ga_h}\widetilde{R}_h u), \vphi_h) - r( v_h; f(u,\nb_{\Ga} u), \vphi_h^\ell) \\
		 =: & \ \boldsymbol{\dot}{IV}_u^1 + \boldsymbol{\dot}{IV}_u^2 + \boldsymbol{\dot}{IV}_u^3.
%         = &\ m_h(\pa_1  g(\widetilde{R}_h u,\nb_{\Ga_h}\widetilde{R}_h u) \mat_h\widetilde{R}_h u, \vphi_h) - m(\pa_1 g(u,\nb_{\Ga} u)\mat_h u, \vphi_h^\ell) \\
% 		 & \ + m_h(\pa_2 g(\widetilde{R}_h u,\nb_{\Ga_h}\widetilde{R}_h u) \mat_h \nb_{\Ga_h} \widetilde{R}_h u, \vphi_h) - m(\pa_2 g(u,\nb_{\Ga} u) \mat_h \nb_{\Ga} u, \vphi_h^\ell) \\
% 		 & \ +  r_h(V_h; g(\widetilde{R}_h u,\nb_{\Ga_h}\widetilde{R}_h u), \vphi_h) - r( v_h ; g(R_h u, \nb_\Ga R_h u), \vphi_h^\ell) \\
% 		 & \ + r(v_h ; g(R_h u, \nb_\Ga R_h u) - g(u,\nb_{\Ga} u), \vphi_h^\ell) \\		 
% 		 \leq &\ c h^{k+1} \big(\| u \|_{L^\infty} \|\pa_1 g(R_h u, \nb_\Ga R_h u)\|_{L^2} + \| \nb_\Ga u \|_{L^\infty} \|\pa_2 g(R_h u, \nb_\Ga R_h u)\|_{L^2}+ \|g(R_h u, \nb_\Ga R_h u)\|_{L^2}\big)\|\vphi_h^\ell\|_{L^2} \\
% 		 & \ + m(\mat_h g(R_h u, \nb_\Ga R_h u) -  \mat_h g(u,\nb_{\Ga} u), \vphi_h^\ell) \\
% 		 & \ + r(v_h ; g(R_h u, \nb_\Ga R_h u) - g(u,\nb_{\Ga} u), \vphi_h^\ell)\\
% 		 \leq &\ c h^{k+1} \big(\|\mat_h g(u, \nb_\Ga u)\|_{L^2} + \|g(u, \nb_\Ga u)\|_{L^2}\big)\|\vphi_h^\ell\|_{L^2} \\
% 		 & \ + m(\mat_h g(R_h u, \nb_\Ga R_h u) -  \mat_h g(u,\nb_{\Ga} u), \vphi_h^\ell) \\
% 		 & \ + c \|R_h u - u\|_{H^1} \|\vphi_h^\ell\|_{L^2}
	\end{aligned}
\end{equation}
Similarly to \eqref{eq:Ritz map W^1,infty bound} we obtain a $W^{1,\infty}$ bound of the material derivative of the Ritz map, see also the proof of Proposition~7.1 in \cite{MCF}, which we need for the next two estimates.
The first term is estimated as
\begin{equation}
 \begin{aligned}
    \boldsymbol{\dot}{IV}_u^1 = & \ m_h(\pa_1 f(\widetilde{R}_h u,\nb_{\Ga_h}\widetilde{R}_h u) \mat_h\widetilde{R}_h u, \vphi_h) - m(\pa_1 f(u,\nb_{\Ga} u)\mat_h u, \vphi_h^\ell) \\
    = & \ m_h(\pa_1  f(\widetilde{R}_h u,\nb_{\Ga_h}\widetilde{R}_h u) \mat_h\widetilde{R}_h u, \vphi_h) - m(\pa_1  f(R_h u,\nb_{\Ga_h} R_h u) \mat_h R_h u, \vphi_h^\ell)\\
    & \ + m(\pa_1  f(R_h u,\nb_{\Ga_h} R_h u)\big(\mat_h R_h u - \mat_h u), \vphi_h^\ell) \\
    & \ + m\big(\big(\pa_1  f(R_h u,\nb_{\Ga_h} R_h u) - \pa_1 f(u,\nb_{\Ga} u)\big)\mat_h u, \vphi_h^\ell\big) \\
    \leq & \ \Big( c \, h^{k+1} \big(\|\pa_1  f(R_h u,\nb_{\Ga_h} R_h u) \, \mat_h R_h u\|_{L^2} + \|\mat_h u \|_{H^{k+1}}\big)  
%    \\ & \ \phantom{\Big( } 
    + \|R_h u - u\|_{H^1}\Big) \|\vphi_h^\ell\|_{L^2} \\
    \leq & \ \Big( c \, h^{k+1} \big(\|\pa_1  f(R_h u,\nb_{\Ga_h} R_h u) \, \mat_h R_h u\|_{L^2} + \|\mat_h u \|_{H^{k+1}}\big)   
%    \\ & \ \phantom{\Big( } 
    + c h^k \|u\|_{H^{k+1}}\Big) \|\vphi_h^\ell\|_{L^2} \\
    \leq & \ \Big( c \, h^{k+1} \big(c\|u\|_{H^{k+1}} + \|\pa_1  g(u,\nb_{\Ga} u)\|_{L^2} + \|\mat_h u \|_{H^{k+1}}\big)   
%    \\ & \ \phantom{\Big( } 
    + c h^k \|u\|_{H^{k+1}}\Big) \|\vphi_h^\ell\|_{L^2} 
 \end{aligned}
\end{equation}
using \eqref{eq:GeometricBilinearPsi}.
The second one additionally uses the interchange formulas \eqref{eq:interchange formulas} to obtain
\begin{equation*}
 \begin{aligned}
  \boldsymbol{\dot}{IV}_u^2 \leq & \ \Big( c \, h^{k+1} \|\pa_2 f(R_h u,\nb_{\Ga_h} R_h u) \, \mat_h \nb_{\Ga_h} R_h u\|_{L^2} 
%  \\ & \ \phantom{\Big( } 
  + c \, h^{k} \|\mat_h u \|_{H^{k+1}} + \|R_h u - u\|_{H^1}\Big) \|\vphi_h^\ell\|_{L^2} \\
  \leq & \ \Big( c \, h^{k+1} \|\pa_2 f(R_h u,\nb_{\Ga_h} R_h u) \, \mat_h \nb_{\Ga_h} R_h u\|_{L^2}  
%  \\ & \ \phantom{\Big( } 
  + c \, h^{k} \big(\|\mat_h u \|_{H^{k+1}}  + \|u\|_{H^{k+1}}\big)\Big) \|\vphi_h^\ell\|_{L^2} \\
  \leq & \ \Big( c \, h^{k+1} \big(c\|u\|_{H^{k+1}} + \|\pa_2  f(u,\nb_{\Ga}u)\|_{L^2}  
%  \\ & \ \phantom{\Big( } 
  + c \, h^{k} \big( c \|\mat u \|_{H^{k+1}}  + \|u\|_{H^{k+1}}\big)\Big) \|\vphi_h^\ell\|_{L^2} .
 \end{aligned}
\end{equation*}
The third one is bounded, similarly to \eqref{eq:est for d_u terms - g term}, by
\begin{equation*}
 \begin{aligned}
  \boldsymbol{\dot}{IV}_u^3 \leq &\ \Big( c h^{k+1} \big( c \|u\|_{H^{k+1}} + c \|f(u, \nb_\Ga u)\|_{L^2} \big) + c h^k \|u\|_{H^{k+1}} \Big) \|\vphi_h^\ell\|_{L^2} .
 \end{aligned}
\end{equation*}
 
The combination of the estimates \eqref{eq:est for mat d_u terms - g_h}--\eqref{eq:est for mat d_u terms - IV - new}, using the norm equivalence \eqref{eq:norm equivalence}, yields for a general $f(u,\nbg u)$:
\begin{subequations}
\label{eq:est for mat d_u}
\begin{equation}
	\|\mat_h d_u\|_{L^2} \leq c h^{k} \Big( \sum_{j=0}^2 \|(\mat)^{(j)} u\|_{H^{k+1}} + \sum_{j=0}^1 \|(\mat)^{(j)} w\|_{H^{k+1}} \Big) .
\end{equation}
If $f$ is \emph{independent} of $\nbg u$, then we obtain
\begin{equation}
	\|\mat_h d_u\|_{L^2} \leq c h^{k+1} \Big( \sum_{j=0}^2 \|(\mat)^{(j)} u\|_{H^{k+1}} + \sum_{j=0}^1 \|(\mat)^{(j)} w\|_{H^{k+1}} \Big) .
\end{equation}
\end{subequations}

%%%%%%%%%%%%%%%%%%%%%%%%%%%%%%%%%%%%%%%%%%%%%%%%%%%%%%%%%%%%%%%%%%%%%%%%%%%%%%%%%%%%%%%%%%%%%%%
%%%%%%%%%%%%%%%%%%%%%%%%%%%%%%%%%%% DEFECT dw %%%%%%%%%%%%%%%%%%%%%%%%%%%%%%%%%%%%%%%%%%%%%%%%%
%%%%%%%%%%%%%%%%%%%%%%%%%%%%%%%%%%%%%%%%%%%%%%%%%%%%%%%%%%%%%%%%%%%%%%%%%%%%%%%%%%%%%%%%%%%%%%%
%
\emph{Bound for $d_w$:}	
The $L^2$ norm of the defect $d_w$ \eqref{eq:defect expression - d_w} is estimated by the same techniques by which the bound \eqref{eq:est for d_u terms - mass term} was shown. 

By similar techniques as before, and using \eqref{eq:GeometricBilinearPsi} together with \eqref{eq:Ritz map W^1,infty bound} the pairs for $d_w$ are estimated analogously. The bounds for $I_w$ and $II_w$ are straightforward using the arguments above for $d_u$, while $III_w$ is bounded,  similarly to \eqref{eq:est for d_u terms - g term}, using the local Lipschitz continuity of $g$, by
\begin{equation}
\label{eq:est for d_w - II nonlinear term}
	\begin{aligned}
		III_w 
%		=
%	%	m_h(g(\widetilde{R}_h u,\nb_{\Ga_h}\widetilde{R}_h u), \vphi_h) - m(g(u,\nb_{\Ga} u), \vphi_h) = 
%		&\  m_h(g(\widetilde{R}_h u,\nb_{\Ga_h}\widetilde{R}_h u), \vphi_h) - m(g(R_h u, \nb_\Ga R_h u), \vphi_h^\ell) \\
%		&\ + m(g(R_h u,\nb_{\Ga} R_h u) - g(u,\nb_{\Ga} u), \vphi_h^\ell) \\
%		\leq &\ ch^{k+1} \|g(R_h u, \nb_\Ga R_h u)\|_{L^2} \|\vphi_h^\ell\|_{L^2} 
%	%	\\ &\
%		+ c \|R_h u - u\|_{H^1} \|\vphi_h^\ell\|_{L^2} \\
%		\leq &\ \Big( c h^{k+1} \|g(R_h u, \nb_\Ga R_h u)\|_{L^2} + c h^k \|u\|_{H^{k+1}} \Big) \|\vphi_h^\ell\|_{L^2} \\
%%		\leq &\ \Big( c h^{k+1} \|g(R_h u, \nb_\Ga R_h u) - g(u, \nb_\Ga u)\|_{L^2} + \|f(u, \nb_\Ga u)\|_{L^2} + c h^k \|u\|_{H^{k+1}} \Big)  \|\vphi_h^\ell\|_{L^2} \\
		\leq &\ \Big( c h^{k+1} \big( c \|u\|_{H^{k+1}} + \|g(u, \nb_\Ga u)\|_{L^2} \big) + c h^k \|u\|_{H^{k+1}} \Big) \|\vphi_h^\ell\|_{L^2} .
	\end{aligned}
\end{equation}
Again, note the only $O(h^k)$-term in \eqref{eq:est for d_w - II nonlinear term}.

We altogether obtain the estimate, for the general case $g(u,\nbg u)$:
\begin{subequations}
\label{eq:est for d_w}
\begin{equation}
	\begin{aligned}
		\|d_w\|_{L^2} \leq &\ c h^{k} \Big( \|u\|_{H^{k+1}} + \|w\|_{H^{k+1}} + \|u\|_{W^{2,\infty}} \Big) .
	\end{aligned} 
\end{equation}
Similarly as before, if $g$ is \emph{independent} of $\nbg u$, the above estimate improves to
\begin{equation}
	\|d_w\|_{L^2} \leq c h^{k+1} \Big( \|u\|_{H^{k+1}} + \|w\|_{H^{k+1}} + \|u\|_{W^{2,\infty}} \Big) .
\end{equation}
\end{subequations}

\emph{Bound for $\mat_h d_w$:}	
Just as for $\mat_h d_u$, we differentiate the expression \eqref{eq:defect expression - d_w} with respect to time. Using again $\mat_h\vphi_h = \mat_h (\vphi_h^\ell) = 0$, we obtain
\begin{equation*}
\begin{aligned}
	m_h(\mat_h d_w , \vphi_h) = - r_h(V_h;d_w,\vphi_h) + \diff \Big(I_w + II_w + III_w \Big) .
\end{aligned}
\end{equation*}
The first term is estimated using \eqref{eq:est for d_w}, while the remaining terms are bounded similarly to \eqref{eq:est for mat d_u terms - II} and \eqref{eq:est for mat d_u terms - IV - new} (using \eqref{eq:Ritz map W^1,infty bound}).

Altogether, we obtain, for a general $g(u,\nbg u)$:
\begin{subequations}
\label{eq:est for mat d_w}
\begin{equation}
	\|\mat_h d_w\|_{L^2} \leq c h^{k} \bigg( \sum_{j=0}^1 \Big( \|(\mat)^{(j)} u\|_{H^{k+1}} + \|(\mat)^{(j)} w\|_{H^{k+1}} \Big) + \|u\|_{W^{2,\infty}} \bigg) .
\end{equation}
while, if $g$ is \emph{independent} of $\nbg u$ we obtain
\begin{equation}
	\|\mat_h d_w\|_{L^2} \leq c h^{k+1} \bigg( \sum_{j=0}^1 \Big( \|(\mat)^{(j)} u\|_{H^{k+1}} + \|(\mat)^{(j)} w\|_{H^{k+1}} \Big) + \|u\|_{W^{2,\infty}} \bigg) .
\end{equation}
\end{subequations}

\qed\end{proof}

\blueoff

\begin{remark}
	If the non-linearities are depending only \emph{linearly} on $\nbg u$, e.g.~an advective term $f(u,\nbg u) = \tilde f(u) + \boldsymbol{w} \cdot \nbg u$, then the defects (although do not fall into case (b)) can still be bounded as $O(h^{k+1})$. This requires the use of individually modified Ritz maps, whose definition includes this linear $\nbg u$-depending term. Such Ritz maps have been already used and analysed in \cite[Definition~8.1]{LubichMansour_wave}, and \cite{Willmore}.
\end{remark}

\section{Proof of Theorem~\ref{theorem:semi-discrete convergence}}
\label{section:proof of main theorem}

\begin{proof}[Proof of Theorem~\ref{theorem:semi-discrete convergence}]
We combine the stability bound of Proposition~\ref{proposition:stability}, and the consistency estimates of Proposition~\ref{proposition:consistency}.

The errors are split as follows
\begin{equation*}
\begin{aligned}
	u - u_h^\ell = &\ u - R_h u + \big( u_h^* - u_h \big)^\ell , \\
	w - w_h^\ell = &\ w - R_h w + \big( w_h^* - w_h \big)^\ell , \\
	\mat( u -  u_h^\ell) = &\ \mat (u - R_h u) + \big( \mat_h (u_h^* - u_h) \big)^\ell ,
\end{aligned}
\end{equation*}
upon recalling that $u_h^* = \widetilde{R}_h u$ and $w_h^* = \widetilde{R}_h w$.

The first terms in each error are directly and similarly bounded by error estimates for the Ritz map Lemma~\ref{lemma:Ritz map error} -- uniformly in time -- by
\begin{equation*}
	\|u - R_h u\|_{L^2(\Ga\t)} + h \|u - R_h u\|_{H^1(\Ga\t)} \leq c h^{k+1} \|u\|_{H^{k+1}(\Ga\t)} .
%	\sum_{j=0}^{1} \big(\|(\mat)^{(j)} (u - R_h u)\|_{L^2} + h \|(\mat)^{(j)} (u - R_h u)\|_{H^1}\big) \leq c h^{k+1} ,
\end{equation*}
The second terms are the errors $e_{u_h}$, $e_{w_h}$ and $\mat_h e_{u_h}$, therefore bounded by the combination of the stability estimate \eqref{eq:stability bound} and the consistency estimates Proposition~\ref{proposition:consistency} (a) and (b), for the two respective cases of $\nbg u$ dependency. In Proposition~\ref{proposition:stability} the $W^{1,\infty}$ norm assumption on $u_h^* = \widetilde{R}_h u$ was proved in \eqref{eq:Ritz map W^1,infty bound}. 
Altogether, we obtain
\begin{equation*}
%	\|e_{u_h}\|_{H^1(\Ga\t)}^2 + \|e_{w_h}\|_{H^1(\Ga\t)}^2 + \int_0^t{\|\mat_h e_{u_h}\|_{H^1(\Ga(s))}^2} \d s \leq c h^{k+1} .
	\|e_{u_h}\|_{H^1(\Ga_h\t)}^2 + \|e_{w_h}\|_{H^1(\Ga_h\t)}^2 + \int_0^t{\|\mat_h e_{u_h}\|_{H^1(\Ga_h(s))}^2} \d s \leq c h^{2 j} ,
\end{equation*}
where $j = k$ in case (a), and $j = k+1$ in case (b).

By combining the above estimates we obtain the stated error estimates in parts (a) and (b) of Theorem~\ref{theorem:semi-discrete convergence}.
\qed\end{proof}

\section{Full discretisation via linearly implicit backward difference formulae}
\label{section:BDF}
We recall the matrix--vector formulation from \eqref{eq:matrix-vector form - diff}:
\begin{subequations}
	\begin{align*}
	\diff \Big( \bfM\t \bfu\t \Big) + \bfA\t \bfw\t = &\ \bff(\bfu\t) , \\
	\bfM\t \bfw\t - \bfA\t \bfu\t = &\ \bfg(\bfu\t) .
	\end{align*}
\end{subequations}
As a time discretisation, we consider the linearly implicit $s$-step \emph{backward differentiation formulae} (BDF). 
For a step size $\tau>0$, and with $t_n = n \tau \leq T$, the discretised time derivative is determined by
\begin{equation}
\label{eq:backward differences def}
	\dot \bfu^n = \frac{1}{\tau} \sum_{j=0}^s \delta_j \bfu^{n-j} , \qquad n \geq s ,
\end{equation}
while the non-linear term uses an extrapolated value, and reads as:
\begin{equation*}
	\widetilde{\bfu}^n := \sum_{j=0}^{s-1} \gamma_j \,\bfu^{n - 1 -j} , \qquad n \geq s .
\end{equation*}
We determine the approximations to the variables $\bfu^n$ to $\bfu(t_n)$ and $\bfw^n$ to $\bfw(t_n)$ by the fully discrete system of \emph{linear} equations, for $n \geq s$,
\begin{equation}
\label{eq: MatrixSystemBDFFull}
	\begin{bmatrix}
	\delta_0 \,\bfM(t_n) & \tau\,\bfA(t_n) \\
	-\, \bfA(t_n) & \bfM(t_n)
	\end{bmatrix}
	\begin{bmatrix}
	\bfu^n \\
	\bfw^n
	\end{bmatrix}
	=
%	\begin{bmatrix}
%	\bff (\sum_{j=0}^{s-1}{\gamma_j \,\bfu^{n - 1 -j}}) - \sum_{j=1}^{s} {\delta_j \,\bfM(t_{n-j})\,\bfu^{n-j}} \\
%	\,  \bfg (\sum_{j=0}^{s-1}{\gamma_j \,\bfu^{n - 1 -j}})
%	\end{bmatrix}
	\begin{bmatrix}
	\bff (\widetilde{\bfu}^n) - \sum_{j=1}^{s} {\delta_j \,\bfM(t_{n-j})\,\bfu^{n-j}} \\
	\,  \bfg (\widetilde{\bfu}^n)
	\end{bmatrix},
\end{equation}
which is used for the upcoming numerical experiments.
The starting values $\bfu^i$ and $\bfw^i$ ($i=0,\dotsc,s-1$) are assumed to be given. They can be precomputed using either a lower order method with smaller step sizes, or an implicit Runge--Kutta method.

The method is determined by its coefficients, given by $\delta(\zeta)=\sum_{j=0}^s \delta_j \zeta^j=\sum_{\ell=1}^s \frac{1}{\ell}(1-\zeta)^\ell$ and $\gamma(\zeta) = \sum_{j=0}^{s-1} \gamma_j \zeta^j = (1 - (1-\zeta)^s)/\zeta$. 
The classical BDF method is known to be zero-stable for $s\leq6$ and to have order $s$; see \cite[Chapter~V]{HairerWannerII}.
This order is retained by the linearly implicit variant using the above coefficients $\gamma_j$; 
cf.~\cite{AkrivisLubich_quasilinBDF,AkrivisLiLubich_quasilinBDF}.

The anti-symmetric structure of the system is preserved, and is observed in \eqref{eq: MatrixSystemBDFFull}. Since the idea of energy estimates, using the $G$-stability theory of Dahlquist \cite{Dahlquist} and the multiplier technique of Nevanlinna \& Odeh \cite{NevanlinnaOdeh}, can be transferred to linearly implicit BDF full discretisations (up to order 5), we strongly expect that Proposition~\ref{proposition:stability} translates to the fully discrete case, and so does the convergence result Theorem~\ref{theorem:semi-discrete convergence}. 
This is strengthened by the successful application of these techniques to the analogous linearly implicit backward difference methods applied to evolving surface PDEs: \cite{LubichMansourVenkataraman_bdsurf,ALE2,KovacsPower_quasilinear} showing optimal-order error bounds for various problems on evolving surfaces. 
The method was also analysed for various geometric surface flows, for $H^1$-regularised surface flows \cite{soldrivenBDF}, and for mean curvature flow \cite{MCF}, both proving optimal-order error bounds for full discretisations.

\section{Numerical experiments}
\label{section:numerics}

We performed numerical experiments, using \eqref{eq: MatrixSystemBDFFull}, for the classical non-linear Cahn--Hilliard equation on an evolving surface,  hence our results are easily compared to those in the literature, in particular \cite{ElliottRanner_CH}.  
We report on the following experiments:
\begin{itemize}
	\item[-] We perform a convergence test for the \emph{non-linear} Cahn--Hilliard equation with the linear evolving surface FEM and BDF methods of various order, to illustrate the convergence rates of Theorem~\ref{theorem:semi-discrete convergence}. We would like to note here that \cite{ElliottRanner_CH} only presents errors and EOCs for a \emph{linear} problem (using the linearly implicit Euler method).
	\item[-] We perform the same experiment as Elliott and Ranner in \cite[Section~6.2]{ElliottRanner_CH}, i.e.~we report on the evolution of the Ginzburg--Landau energy along the surface evolution for the non-linear Cahn--Hilliard equation with $\eps = 0.1$ using the first and second order BDF methods.
	\item[-] We perform a numerical experiment that reports on the effects of $\boldsymbol{\vartheta}$ and using the Ritz map as initial value. 
\end{itemize}
In the numerical experiments we use the classical Cahn--Hilliard equation on an evolving surface \eqref{eq:CH system} with the double-well potential, hence the non-linear terms are $f(u) = 0$ and  $g(u) = \frac{1}{4}((u^2 - 1)^2)' = u^3 - u$. With an arbitrary $0 < \eps < 1$, formulated as a system the problem reads:
\begin{equation}
	\label{eq:CH system - inhomogene}
	\begin{alignedat}{3}
		\mat u - \laplace_{\Ga\t} w = &\ - u (\nb_{\Ga\t} \cdot v) + b & \quad & \text{on } \Ga\t , \\
		w + \eps \laplace_{\Ga\t} u  = &\  \eps\inv  g(u) & \quad & \text{on } \Ga\t ,
	\end{alignedat}
\end{equation}
with an extra inhomogeneity $b (\cdot,t) : \Ga\t \to \R$, chosen such that the exact solution is known to be $u(x,t) = e^{-6t} x_1 x_2$, while $w$ is also explicitly known through the second equation of \eqref{eq:CH system - inhomogene}.
The surface $\Ga\t$ evolves time-periodically from a sphere into an ellipsoid and back.
In particular the surface is given as the zero level set of a distance function:
\begin{equation}
\label{eq: evolvingEllipsoid}
	\Gamma\t = \big\{ x \in \mathbb{R}^3 \mid d(x,t) = a\t^{-1} x_1^2 + x_2^2 + x_3^2 - 1 = 0 \big\} ,
\end{equation}
with $a\t = 1 + 0.25\sin(2 \pi t)$. The initial surface $\Ga(0) = \Ga^0$ is the unit sphere. 
The surface evolution is computed using the ODE for the positions \eqref{eq:ODE for positions}, with 
\begin{equation*}
	v = V \nu , \qquad \text{with} \qquad V = - \frac{\pa_t d}{|\nb d|} \ \ \text{and} \ \ \nu = \frac{\nb d}{|\nb d|} .
\end{equation*}
For the numerical experiments the ODE was solved numerically by the classical 4th order Runge--Kutta method with the smallest time step size present in the experiment.

Various numerical experiments have been carried out using the same evolving surface, in particular also for the Cahn--Hilliard equation by Elliott and Ranner \cite{ElliottRanner_CH}, and for other problems as well, see, for instance \cite{DziukElliott_ESFEM,LubichMansourVenkataraman_bdsurf}. 

The initial value $u_h^0$ is the interpolation of the exact initial value $u_0$. For high-order BDF methods the required additional starting values $u_h^i$ (for $i=1,\dotsc,q-1$) are taken as the interpolation of the exact values, if they exist, as well or are otherwise computed using a cascade of steps performed by the preceding lower order method.

\subsection{Convergence experiments}

The following convergence experiments are illustrating the convergence rates stated by Theorem~\ref{theorem:semi-discrete convergence}. 
In these experiments we have used the parameter $\eps = 0.5$. The final time is $T=1$, the time discretisations use a sequence of time step sizes $\tau = 0.2 \cdot 2^{-i}$ for $i=1,\dotsc,7$, and a sequence of initial meshes with (roughly quadrupling) degrees of freedom as reported in the figures.

In Figures~\ref{figure:BDF1_convplot_space}--\ref{figure:BDF3_convplot_time} we report on the $L^\infty(L^2)$ norm errors (left) and $L^\infty(H^1)$ norm errors (right) between the numerical and exact solution for both variables $u$ and $w$, i.e.~the plots show the errors
\begin{equation*}
	\|u-u_h^\ell\|_{L^\infty(L^2)} + \|w-w_h^\ell\|_{L^\infty(L^2)} \quad \text{ and } \quad \|u-u_h^\ell\|_{L^\infty(H^1)} + \|w-w_h^\ell\|_{L^\infty(H^1)} ,
\end{equation*}
where the norms are understood as 
\begin{equation*}
	\|u-u_h^\ell\|_{L^\infty(L^2)} = \max_{0 \leq n \tau \leq T} \| u(\cdot,n \tau)-(u_h^n)^\ell \|_{L^2(\Ga(n \tau))} .
\end{equation*}
For the first order BDF method, Figure~\ref{figure:BDF1_convplot_space} shows logarithmic plots of the errors against the mesh width $h$, the lines marked with different symbols correspond to different time step sizes. We also report on temporal convergence in Figure~\ref{figure:BDF1_convplot_time}, where the roles are reversed, the errors are plotted against the time step size $\tau$, and the lines with different markers correspond to different mesh refinements.

In Figure~\ref{figure:BDF1_convplot_space} we can observe two regions: a region where the spatial discretisation error dominates, matching to the order of convergence of our theoretical results of Theorem~\ref{theorem:semi-discrete convergence} (note the reference lines), and a region, with small mesh widths, where the temporal discretisation error dominates (the error curves flatten out). For the $H^1$ norm we observe better spatial convergence rates as the predicted $O(h^k)$, (probably due to the smoothness of the exact solution). For Figure~\ref{figure:BDF1_convplot_time}, the same description applies, but with reversed roles. Although, we do not study convergence of full discretisations, the classical order of the BDF methods is observed. We note here, that flat error curves, which were completely dominated by a discretisation error, were not plotted.

Figure~\ref{figure:BDF3_convplot_space} and \ref{figure:BDF3_convplot_time} report on the same plots, but for the third order BDF method. Again, both the spatial and temporal convergence, as shown by the figures, are in agreement with the theoretical convergence results of Theorem~\ref{theorem:semi-discrete convergence} and with  the classical orders of the BDF methods (note the reference lines).

The plots for time convergence, Figures~\ref{figure:BDF1_convplot_time} and \ref{figure:BDF3_convplot_time}, are supporting our claim that Theorem~\ref{theorem:semi-discrete convergence} can be extended for full discretisations with linearly implicit BDF methods, which is left to a subsequent work. 

% BDF 1
\begin{figure}[htbp]
	\includegraphics[width=\textwidth]{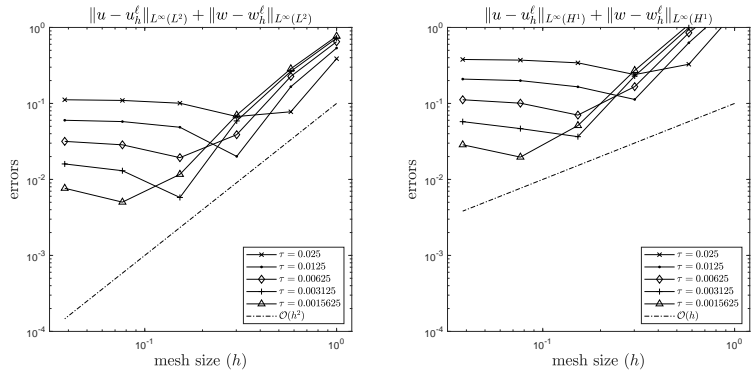}
	\caption{Spatial convergence of the BDF1 / linear ESFEM discretisation for the non-linear Cahn--Hilliard equation on an evolving ellipsoid}
	\label{figure:BDF1_convplot_space}
\end{figure}
\begin{figure}[htbp]
	\includegraphics[width=\textwidth]{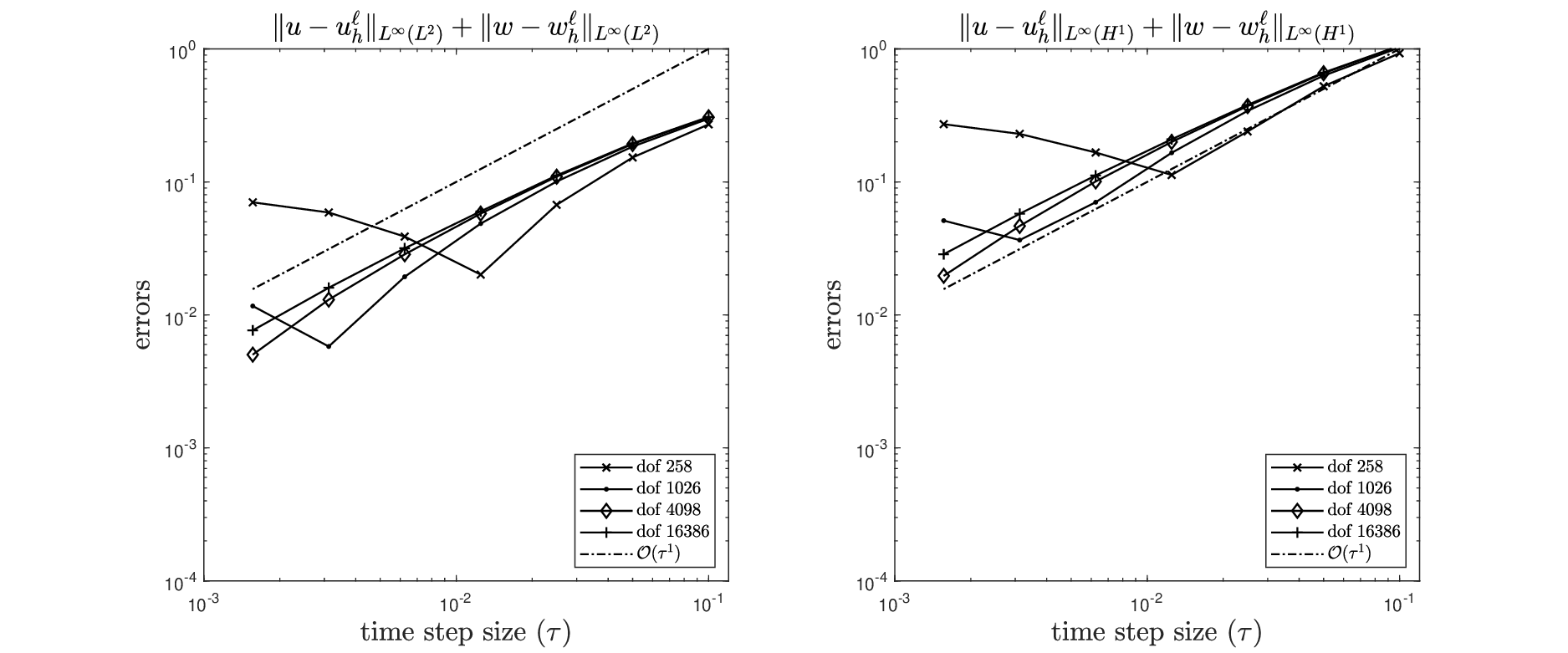}
	\caption{Temporal convergence of the BDF1 / linear ESFEM discretisation for the non-linear Cahn--Hilliard equation on an evolving ellipsoid}
	\label{figure:BDF1_convplot_time}
\end{figure}

% BDF 3
\begin{figure}[htbp]
	\includegraphics[width=\textwidth]{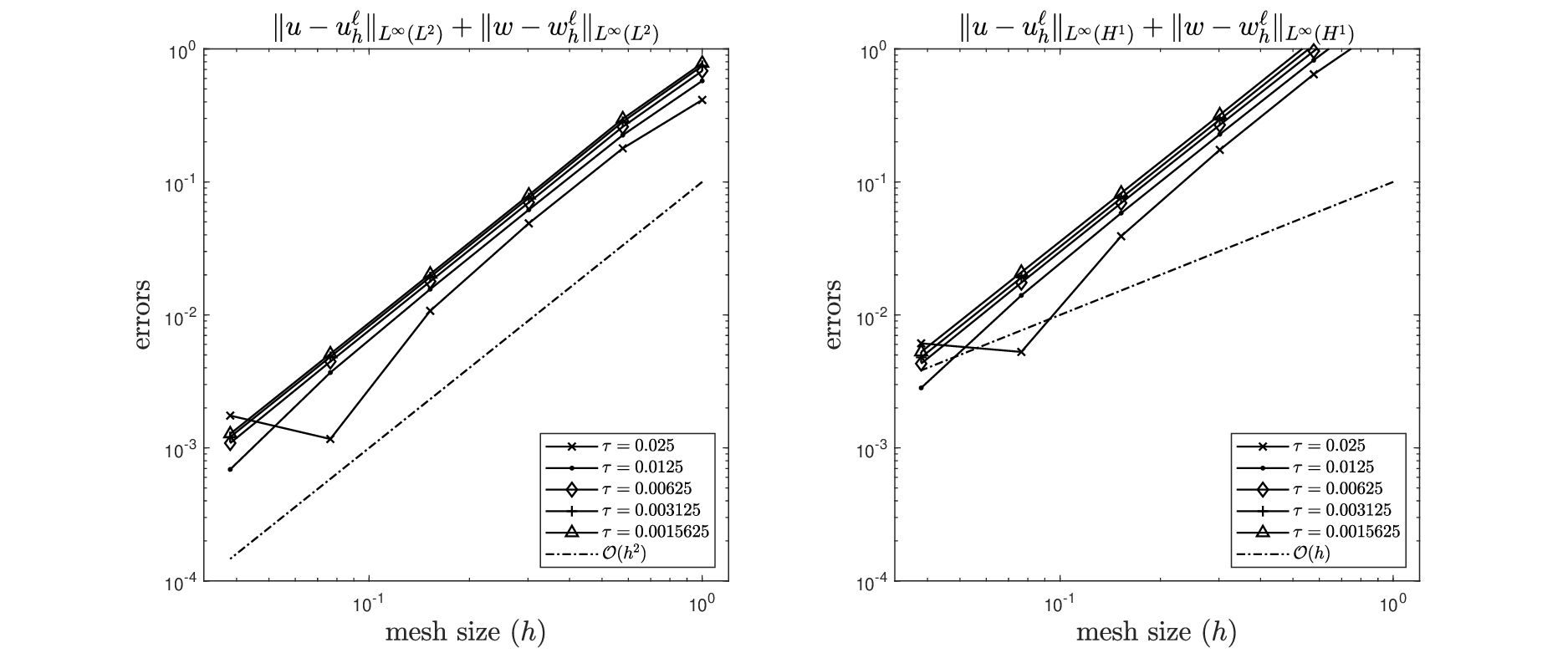}
	\caption{Spatial convergence of the BDF3 / linear ESFEM discretisation for the non-linear Cahn--Hilliard equation on an evolving ellipsoid}
	\label{figure:BDF3_convplot_space}
\end{figure}
\begin{figure}[htbp]
	\includegraphics[width=\textwidth]{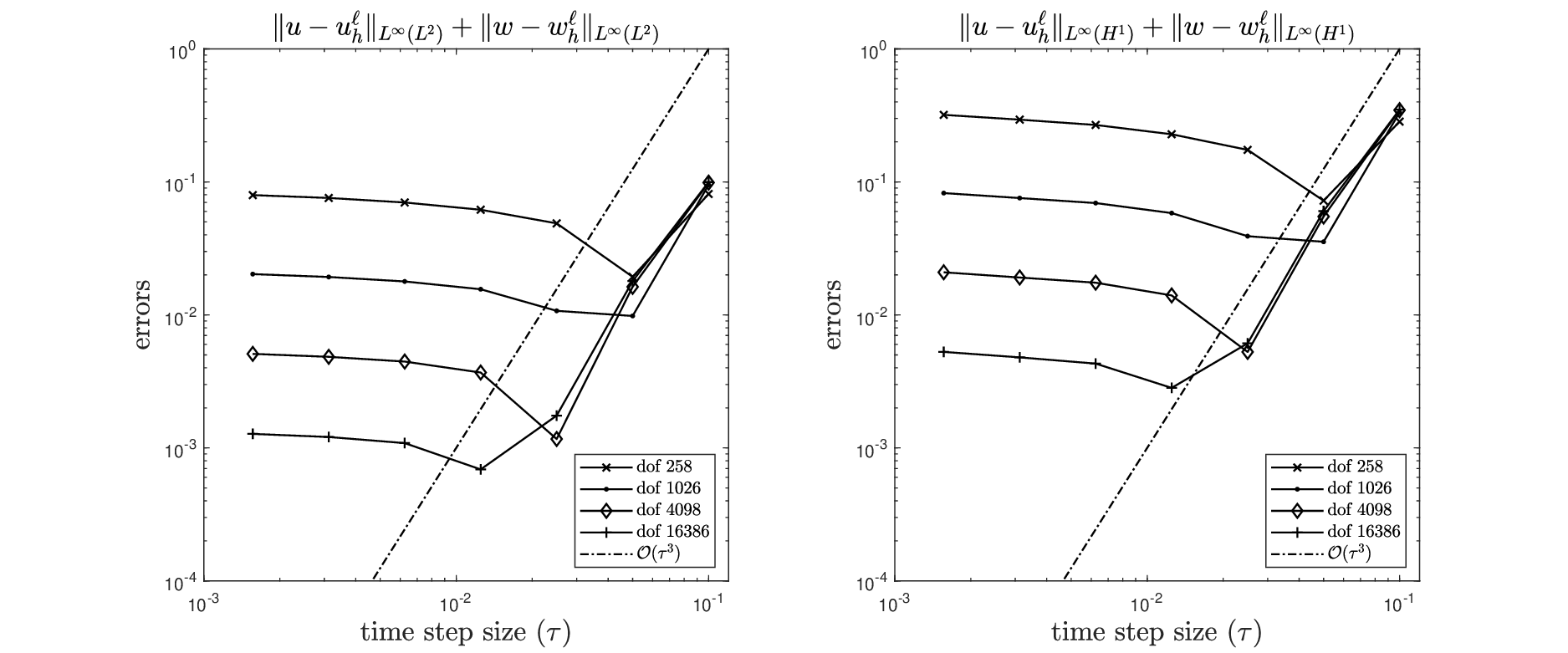}
	\caption{Temporal convergence of the BDF3 / linear ESFEM discretisation for the non-linear Cahn--Hilliard equation on an evolving ellipsoid}
	\label{figure:BDF3_convplot_time}
\end{figure}

\subsection{The Ginzburg--Landau energy}

The numerical experiments in \cite[Section~6.2]{ElliottRanner_CH} reporting on the Ginzburg--Landau energy were repeated here for high-order BDF methods.

We again consider the non-linear Cahn--Hilliard equation \eqref{eq:CH system - inhomogene}, with $\eps = 0.1$ and with $b = 0$ on the same evolving surface $\Ga\t$ as before, but with $a\t = 1 + 0.25\sin(10 \pi t)$, and with initial value
\begin{equation*}
	u_0(x) = 0.1 \cos(2 \pi x_1) \cos(2 \pi x_2) \cos(2 \pi x_3) .
\end{equation*}
This setting is the same as in \cite[Section~6.2]{ElliottRanner_CH}.
\\
\\
In Figure~\ref{figure:Ginzburg--Landau energy - T=0.2} and \ref{figure:Ginzburg--Landau energy - T=1} we report on the time evolution of the Ginzburg--Landau energy (until $T=0.2$ and $T=1$) of the BDF2 / linear ESFEM discretisation. In both plots we have used the time step size $\tau = 10^{-4}$ (the same as \cite[Section~6.2]{ElliottRanner_CH}), and eight different mesh refinement levels (higher numbering denotes finer meshes). The meshes are not nested refinements of a single coarse grid. The coarsest mesh has $54$ while the finest has $10146$ nodes. 

As it was pointed out by Elliott and Ranner \cite{ElliottRanner_CH} \emph{``the energy does not decrease monotonically along solutions''}, see Figure~\ref{figure:Ginzburg--Landau energy - T=0.2}, and as they predicted the solutions converge to a time-periodic solution, the periodicity in their energies is nicely observed in Figure~\ref{figure:Ginzburg--Landau energy - T=1}.

% Ginzburg--Landau energy
\begin{figure}[htbp]
	\includegraphics[width=\textwidth]{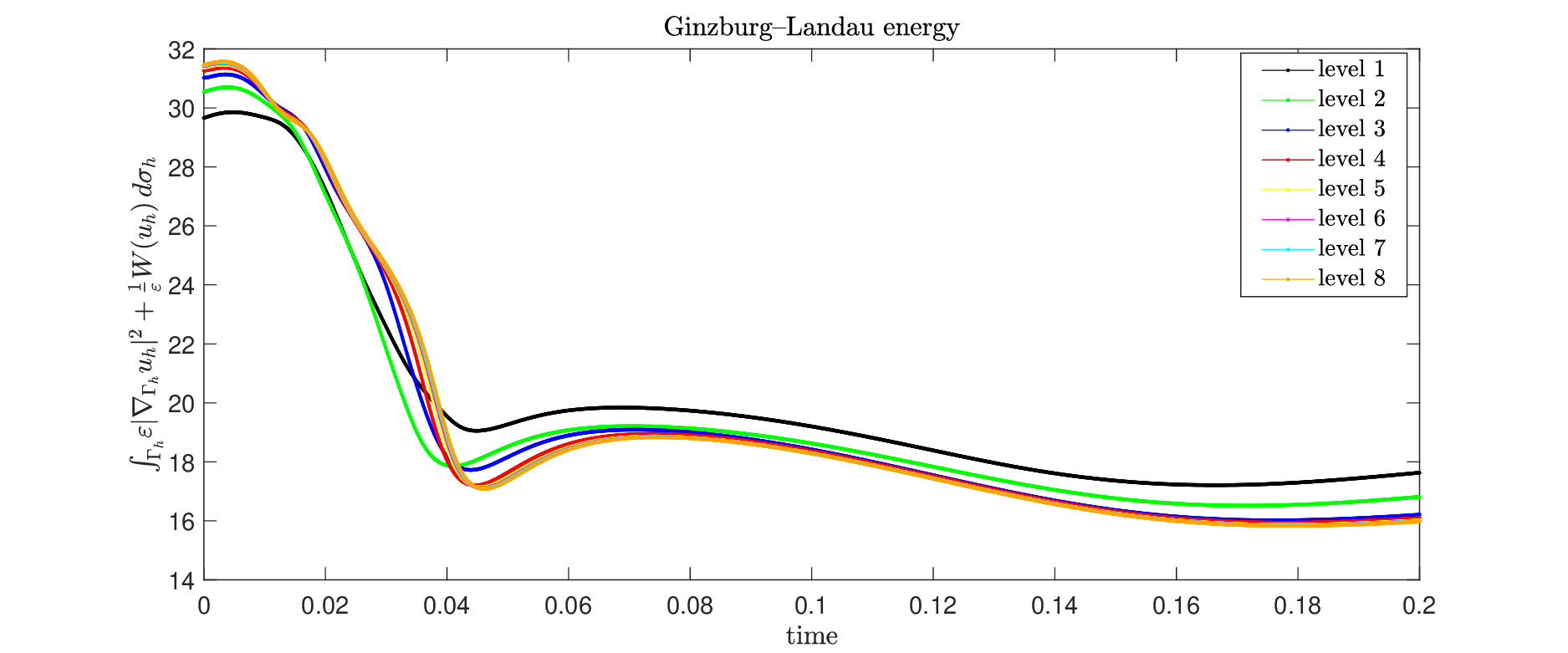}
	\caption{The Ginzburg--Landau energy over $[0,0.2]$ for BDF2 / linear ESFEM discretisation with $\tau = 10^{-4}$ and over several spatial refinements.}
	\label{figure:Ginzburg--Landau energy - T=0.2}
\end{figure}
\begin{figure}[htbp]
	\includegraphics[width=\textwidth]{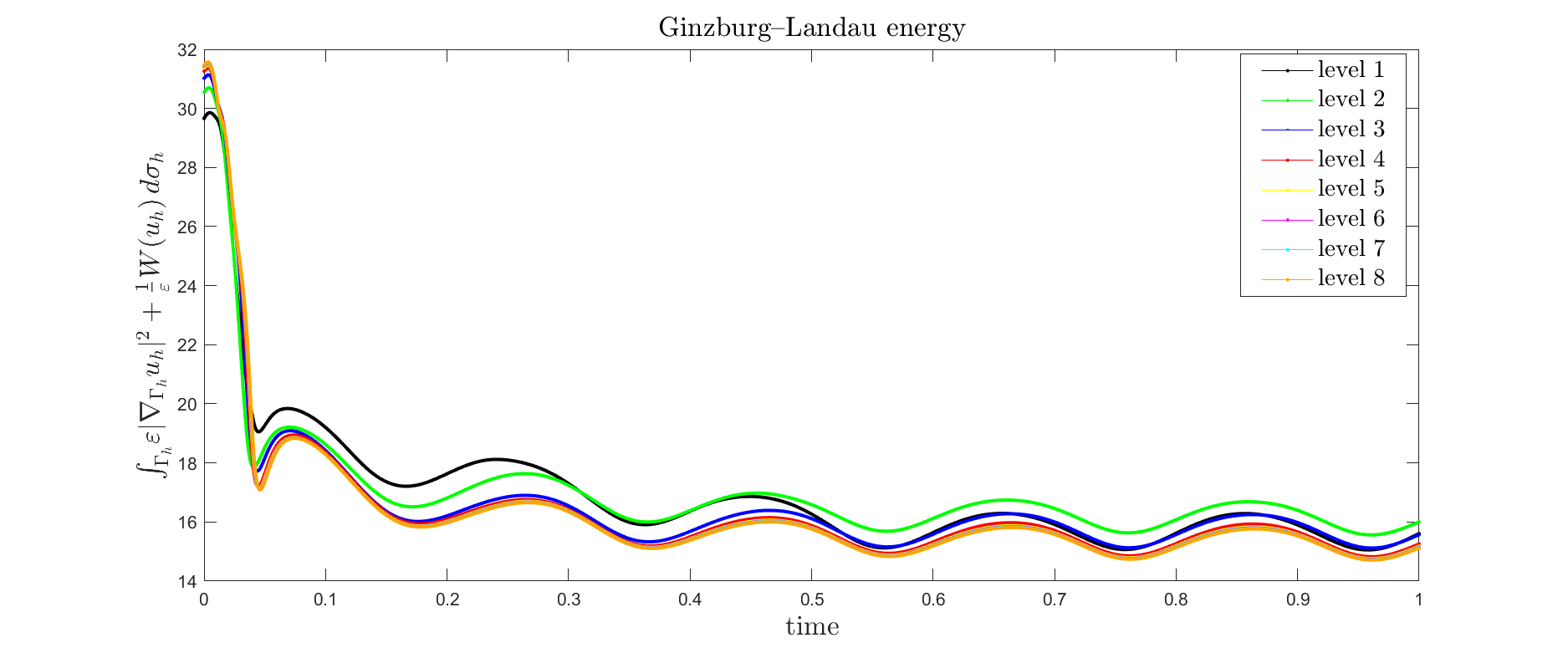}
	\caption{The Ginzburg--Landau energy over $[0,1]$ for BDF2 / linear ESFEM discretisation with $\tau = 10^{-4}$ and over several spatial refinements.}
	\label{figure:Ginzburg--Landau energy - T=1}
\end{figure}

\subsection{The effect of $\boldsymbol{\vartheta}$}

We report on the effect of $\boldsymbol{\vartheta}$ by presenting the computed numerical solution obtained from the scheme \eqref{eq:matrix-vector form - pre - diff} and \eqref{eq:matrix-vector form - diff} with the interpolation and the Ritz map as initial values, respectively. 

We again use the evolving ellipsoid example with $a\t = 1 + 0.5\sin(\frac{2 \pi t}{5})$, cf.~\eqref{eq: evolvingEllipsoid} and an initial sphere of radius $R = 5$, while the starting value is $u^0 = \frac{225}{56693}(x_1 + x_1^2x_2^2x_3)$ (such that $\max |u^0| = 1$). The discrete initial values are the interpolation of $u^0$ for \eqref{eq:matrix-vector form - pre - diff} and the Ritz map \eqref{eq:definition Ritz map} of $u^0$ for \eqref{eq:matrix-vector form - diff}. The nodal vector $\boldsymbol{\vartheta}$ and the Ritz map are each obtained by solving an elliptic problem.

Figure~\ref{figure:solutions} presents the numerical solutions with the two different discrete initial values, without (left) and with $\boldsymbol{\vartheta}$ (right), for different times $t = 0, 1, 2, 3, 5$, computed on a mesh with $4098$ nodes and using a time step size $\tau = 0.0125$.

\begin{figure}[htbp]
	\includegraphics[width=\textwidth,trim={105 85 60 50},clip]{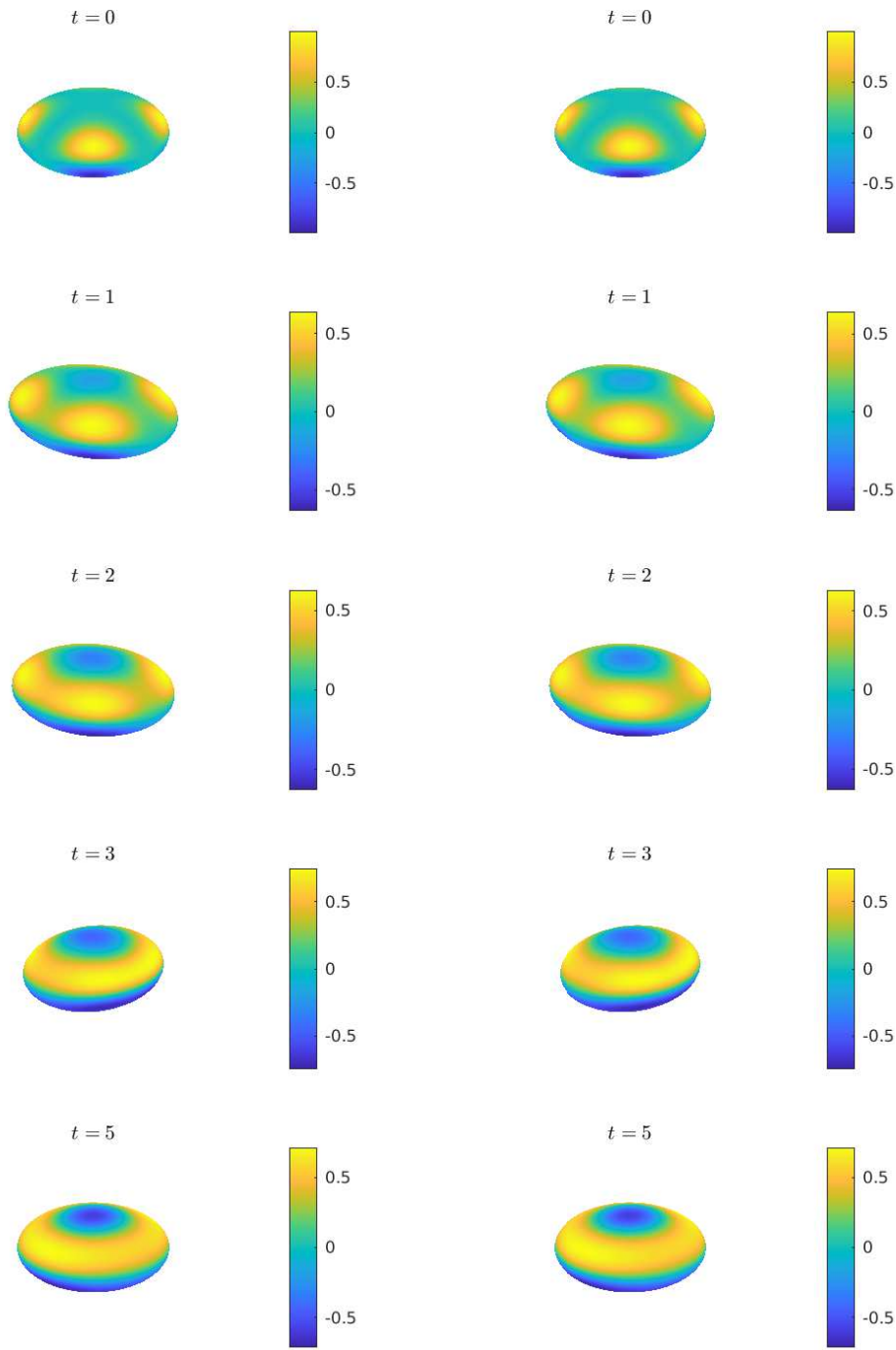}
	\caption{The numerical solutions obtained from \eqref{eq:matrix-vector form - pre - diff} and \eqref{eq:matrix-vector form - diff} -- without and with $\boldsymbol{\vartheta}$ -- with the  interpolation and Ritz map as initial values (on the left- and right-hand columns, respectively).}
	\label{figure:solutions}
\end{figure}

\section*{Acknowledgement}
We thank Christian Lubich for helpful discussions, in particular on initial values.

We would like to thank two Referees whose comments have helped us to improve the presentation of the paper.

The manuscript was partially written when Bal\'azs Kov\'acs had been working at the University of T\"ubingen. We gratefully acknowledge their support.

The work of Cedric Aaron Beschle is funded by the Deutsche Forschungsgemeinschaft (DFG, German Research Foundation) -- Project-ID 251654672 -- TRR 161. 

The work of Bal\'azs Kov\'acs is supported by Deutsche Forschungsgemeinschaft -- Project-ID 258734477 -- SFB 1173, and by the Heisenberg Programme of the Deutsche Forschungsgemeinschaft (DFG, German Research Foundation) -- Project-ID 446431602.

%\clearpage
\bibliographystyle{abbrv}
\bibliography{CHsurf_references}

\end{document}